\documentclass[a4paper,11pt]{article}
\usepackage{graphicx}
\usepackage[T1]{fontenc}
\usepackage[utf8]{inputenc}
\usepackage[english]{babel}
\usepackage{csquotes}
\usepackage{authblk}

\usepackage{amsmath,amsthm}

\usepackage[bitstream-charter,greeklowercase= upright]{mathdesign}

\usepackage{subcaption}
\usepackage{geometry}
\usepackage{microtype,titlesec}
\usepackage{nicefrac}
\usepackage{empheq}

\usepackage[svgnames,x11names,table]{xcolor} 

\usepackage{pgfplots}
\usepackage{hyperref}
\hypersetup{
	colorlinks=true,        
	linkcolor=DodgerBlue,   
	citecolor=DarkRed,      
	urlcolor=Teal           
}

\definecolor{myblue}{RGB}{0,114,178}
\definecolor{myred}{rgb}{1.0, 0.0, 0.22}
\definecolor{orange}{HTML}{CF4A30}
\definecolor{darkgreen}{HTML}{006400}

\usepackage[disable,colorinlistoftodos,prependcaption,backgroundcolor=black!5!white,bordercolor=red, textsize=small]{todonotes}

\usepackage{framed}
\usepackage[noend]{algorithmic}
\usepackage{algorithm}

\usepackage{cleveref}
\usepackage{pca}
\usepackage{cancel}
\usepackage{enumerate}
\usepackage{enumitem}

\usepackage{bbm}

\numberwithin{equation}{section}
\numberwithin{theorem}{section}

\title{Debiasing optimal transport: classical and entropic}

\author[1]{Pierre-Cyril Aubin-Frankowski\thanks{pierre-cyril.aubin@enpc.fr}}
\author[1,2]{Virginie Ehrlacher\thanks{virginie.ehrlacher@enpc.fr}}
\author[1]{Gabriele Todeschi\thanks{gabriele.todeschi@enpc.fr}}

\affil[1]{CERMICS, CNRS, ENPC, Institut Polytechnique de Paris, Marne-la-Vallée, France}
\affil[2]{INRIA, PARMA team-project, Saclay, France}

\begin{document}
	\maketitle

    \begin{abstract}
		We study the notion of debiasability for cost functions arising in optimal transport. We call a
		symmetric cost function $c:\mathcal{X}\times\mathcal{X}\to\mathbb{R}\cup\{+\infty\}$ debiasable if
		it satisfies $c(x,y)\ge \tfrac{1}{2}c(x,x)+\tfrac{1}{2}c(y,y)$ for all $x,y\in\mathcal{X}$.
		Building on an equivalent characterization by an inf-representation
		$c(x,y)=\inf_{z\in\calZ}\psi(x,z)+\psi(y,z)$ for some set $\calZ$ and some function
		$\psi:\calX\times\calZ\to\mathbb{R}\cup\{+\infty\}$, interpreted as a generalization of the
		midpoint identity for squared geodesic distances, we investigate the debiasability of costs
		defined on spaces of probability measures. Our primary focus is the entropic regularization of
		optimal transport across the different regimes of the regularization parameter
		$\eps\in[0,+\infty]$, encompassing classical optimal transport ($\eps=0$), entropic
		optimal transport ($\eps>0$), and the Maximum Mean Discrepancy ($\eps=+\infty$).
		For $\eps\in(0,+\infty)$ we obtain debiasability under negative definiteness of the ground
		cost, through an explicit Gaussian factorization, and, more generally, under positive semidefiniteness of the kernel $e^{-c/\eps}$
		alone, conditions also proved necessary. All our results extend naturally to unbalanced optimal transport settings
		and we generalize in this way the findings of \cite{feydy2019interpolating} and
		\cite{sejourne2019sinkhorn}. Extending our factorizations to include a time parameter
		$t\in(0,1)$, we construct an entropic interpolation curve given by an intermediate measure $\eta_t\in\calP(\calZ)$ based on a barycentric formulation, which coincides with the
		marginals of the Schr\"odinger bridge between $\mu$ and $\nu$. As a
		byproduct, we derive a new proof
		for the entropic Benamou-Brenier formula under very mild assumptions.
	\end{abstract}

	
    \section{Introduction}
	\paragraph{Context.}
	For a long time, the main focus in optimal transport was the Wasserstein distance $W_p$. Being a distance, it is already debiased, i.e.\ nonnegative and with null self-transportation cost. However with entropic regularizations, the self-transportation becomes positive. This justifies debiasing, e.g.\ as done below when defining the Sinkhorn divergence for entropic optimal transport.
	
	\begin{definition}\label{def:debiasable}
		Given a set $\calX$ and a symmetric function $c: \calX \times \calX \to \R\cup\{+\infty\}$, which we call the cost, we define the debiased cost $c_0: \calX \times \calX \to \R\cup\{\pm\infty\}$ as
		\begin{equation}
			c_0(x,y)=c(x,y)-\frac{c(x,x)}{2}-\frac{c(y,y)}{2},
		\end{equation}
		with the rule $\infty-\infty=+\infty$.
		We say that the cost $c$ is (positively) debiasable iff $c_0(x,y)\ge 0$ for all $x,y \in \calX$. We say that the cost $c$ is strictly (positively) debiasable if $c_0(x,y)> 0$ for $x\neq y \in \calX$.
	\end{definition}

	\paragraph{Notation.}
    For $\bbX$ a measurable space we write $\calP(\bbX)$ for the Borel probability measures,
	$\calM_+(\bbX)$ for the $\sigma$-finite nonnegative Borel measures, $\calM^b_+(\bbX)\subset\calM_+(\bbX)$
	for the finite ones, and $\calM^b(\bbX)$ for the finite signed measures. For any $\mu, \nu \in \calM_+(\bbX)$, we use the shorthand $\mu \ll \nu $ to denote that $\mu$ is absolutely continuous w.r.t.\ $\nu$. When $\mu, \nu \in \calP(\bbX)$, we denote by $\Pi(\mu,\nu)$ the set of couplings in $\calP(\bbX \times \bbX)$ having first marginal $\mu$ and second marginal $\nu$. For any Hilbert space $H$, we denote by $( \cdot, \cdot )_H$ its inner product and by $|\cdot|_H$ its associated norm. We use the shorthand $\R_{\ge 0}$ for $[0,+\infty)$. Given a measurable map $T:\bbX \to \bbX $ and $\mu\in \calP(\bbX)$, $T_{\#}\mu$ is the pushforward measure of $\mu$ by $T$.
    
    Following \cite[Chapter 4]{villani2008optimal}, we shall repeatedly use the following elementary remark to dispense with integrability issues: if a cost $c:\bbX\times\bbX\to\R\cup\{+\infty\}$ satisfies $c(x,y)\ge a(x)+b(y)$ for some Borel functions $a,b:\bbX\to\R$ with $a\in L^1(\mu)$ and $b\in L^1(\nu)$, then $\int c\,\de\pi$ is well defined in $(-\infty,+\infty]$ for every $\pi\in\Pi(\mu,\nu)$, and so is $\OT_c(\mu,\nu)$. A debiasable cost is of this form, with $a=b=\tfrac12c(\cdot,\cdot)$, so that $\OT_c(\mu,\nu)$ is well defined for all $\mu,\nu$ in the class of \emph{diagonally integrable} measures
    \begin{equation}\label{eq:def_P_k}
        \calP_c(\bbX)\coloneqq\Big\{\mu\in\calP(\bbX):\ \int_\bbX |c(x,x)|\,\de\mu(x)<+\infty\Big\}.
    \end{equation}
	We now list some examples of costs based on optimal transport theory, and which are covered by the results we prove hereafter.
	
	\begin{example}[Optimal transport (OT) cost] The most usual optimal transport cost is the $c$-Wasserstein cost defined for $\mu,\nu \in\calP(\bbX)$ as
		\begin{equation}\label{eq:W_c}
			\OT_{c}(\mu,\nu)\coloneqq\inf_{\pi \in \Pi(\mu,\nu)} 
			\int_{\bbX\times\bbX} c(x,y) \de\pi(x,y),
		\end{equation}
		where $c: \bbX \times \bbX \to \mathbb{R}\cup \{+\infty\}$ is the \emph{ground cost}. Particular examples are when $(\bbX,d_\bbX)$ is a Polish space, $c=d_\bbX^p$ for $p\in(0,+\infty)$, in which case $\OT_{d^p_\bbX}$ is a power of a distance with possibly infinite values \cite[Chapter 6]{villani2008optimal} and hence debiasable.
	\end{example}
	\begin{example}[Csisz\'ar $f$-divergence] An entropy function $f:\R_{\ge 0}\to \R \cup \{+\infty\}$ is a convex, lower semicontinuous function such that $f(1)=0$. Its recession constant is defined as $\displaystyle f_\infty=\lim_{r\to+\infty} \frac{f(r)}{r}$.
		For $\mu,\nu\in\calM_+(\bbX)$, the Csisz\'ar divergence associated with $f$ is
		\begin{equation}\label{eq:Csiszar-div}
			F(\mu,\nu)\coloneqq\int_\bbX f\left(\frac{\de\mu}{\de\nu}(x)\right)\de\nu(x)+f_\infty \int_\bbX \de\mu^\perp
		\end{equation}
		where $\mu^\perp$ is the orthogonal part of the Lebesgue decomposition of $\mu$ with respect to $\nu$. For $f(r)=r \log(r)$ with $r\geq 0$, a prime example is the (unnormalized) {\it Kullback--Leibler divergence} or {\it relative entropy} with respect to a reference measure $\nu \in \calM_+(\bbX)$ (note that it can be negative if $\nu \notin \calP(\bbX)$). It is defined as:
		\begin{equation*}
			\KL(\mu, \nu) \coloneqq \left\{
			\begin{array}{ll}
				\int_{\bbX}\log \left(\frac{\de \mu}{\de \nu}(x)\right) \de  \mu(x)& \mbox{if } \mu \ll \nu ,  \\
				+\infty & \mbox{otherwise.  }
			\end{array}\right.
		\end{equation*}	
		The $\KL$ is not symmetric and does not satisfy the triangle inequality although $\KL(\mu,\mu)=0$. A symmetrization of $\KL$ such as the Jensen-Shannon divergence for $\mu,\nu\in\calP(\bbX)$ is thus debiasable. In the rest of the article, we will use the more classical notation $\KL(\mu|\nu)$ for $\KL(\mu,\nu)$. 
	\end{example}

	One can use entropy functions on the transport plans to relax optimal transport to mass unbalanced cases, or to regularize it, often at the price of introducing bias.

	\begin{example}[Unbalanced OT]
		\label{ex:UOT} 
		The unbalanced optimal transport problem as written in \cite[Problem 3.1]{Liero2017} reads for $\mu,\nu\in\calM^b_+(\bbX)$
		\begin{equation}\label{eq:cost_UOT}
			\UOT_{c,F}(\mu,\nu)\coloneqq\inf_{\pi \in \calM^b_+(\bbX \times \bbX)} 
			\int_{\bbX\times\bbX} c(x,y) \de\pi(x,y)+F((\p_1)_\#\pi,\mu)+F((\p_2)_\#\pi,\nu),
		\end{equation}
		with $\p_i$ the projection on the $i$-th component and $F$ some Csisz\'ar divergence.
	\end{example}
	
	\begin{example}[Entropic transport and Sinkhorn divergence]
		\label{ex:Sinkhorn}  For a given $\eps \ge 0$, the {\it entropic optimal transport (EOT) dissimilarity} is defined as 
		\begin{equation}\label{eq:oteps}
			\OT^\eps_c(\mu,\nu)\coloneqq\inf_{\pi \in \Pi(\mu,\nu)} 
			\int_{\bbX\times\bbX} c(x,y) \de\pi(x,y)+\eps\KL(\pi|\mu\otimes \nu)=\inf_{\pi \in \Pi(\mu,\nu)} \eps\KL(\pi|e^{-c/\eps}\mu\otimes\nu).
		\end{equation}
		As $\OT^\eps_c(\mu,\mu)\not =0$ in general \cite[Section 1.2]{feydy2019interpolating}, the {\it Sinkhorn divergence} was defined in \cite{genevay2018learning} as  
		\begin{equation}\label{eq:s_eps}
			\Seps(\mu,\nu)\coloneqq\OT^\eps_c(\mu,\nu)-\frac12\OT^\eps_c(\mu,\mu)-\frac12\OT^\eps_c(\nu,\nu)
		\end{equation}
		which indeed fulfills $\Seps(\mu,\mu)=0$. We refer to the introduction of \cite{Carlier2017}  for a thorough presentation  on $\OT^\eps_c$ and to that of \cite[p.~5]{lavenant2024riemannian} for $\Seps$.
        
        The entropic regularized unbalanced optimal transport problem as studied in \cite[Section 3]{Chizat2018} reads:
		\begin{equation}\label{eq:cost_UOT-eps}
			\UOT_{c,F}^\eps(\mu,\nu)=\inf_{\pi \in \calM_+(\bbX \times \bbX)} 
			\int_{\bbX\times\bbX} c(x,y) \de\pi(x,y)+F((\p_1)_\#\pi,\mu)+F((\p_2)_\#\pi,\nu)+\eps\KL(\pi|\mu\otimes \nu).
		\end{equation}
	\end{example}
	\begin{example}[Maximum Mean Discrepancy]
		\label{ex:MMD} Taking formally $\eps\to +\infty$ in \eqref{eq:oteps} and \eqref{eq:s_eps}, we obtain
		\begin{equation}\label{eq:s_eps_infty_intro}
			\OT^\infty_c(\mu,\nu) \coloneqq \int_{\bbX\times \bbX} c(x,y) \de \mu(x) \de\nu(y), \quad \operatorname{S}^\infty_c(\mu,\nu)\coloneqq \frac{1}{2}\int_{\bbX\times \bbX} (-c)(x,y)\de (\mu-\nu)(x) \de (\mu-\nu)(y).
		\end{equation}
		Clearly $\operatorname{S}^\infty_c$ corresponds to a quadratic form, whose sign depends on whether $(-c)$ is associated with a positive semidefinite operator or not, in which case the quantity $\operatorname{S}^\infty_c$ is an example of squared Maximum Mean Discrepancy, originally introduced in \cite{Gretton2006}.
	\end{example}
	Our key question is whether having a debiasable ground cost $c$ lifts to having a debiasable cost on measures in the above examples. Unfortunately, this is not in general the case, but sufficient conditions exist, and being a reproducing kernel proves necessary.
	
	\paragraph{Main contribution.} Our main result reads as follows. We refer to the corresponding statements for the mild assumptions concerning integrability.

    \begin{theorem} Fix a debiasable ground cost $c:\bbX\times\bbX\to\R\cup\{+\infty\}$ over a Polish space $\bbX$. Then, for a given $\eps\in [0,+\infty]$, $\OT^\eps_c$ is debiasable provided that one of the following sufficient conditions holds, the condition of item 3 for $\eps\in(0,+\infty)$ and negative definiteness for $\eps=+\infty$ being moreover necessary:
			\begin{itemize}
			\item $\eps=0$, in which case \eqref{eq:oteps} is just classical OT (\Cref{prop:direct_proofs_ot} and \Cref{fact:ot});
			\item $\eps\in(0,+\infty)$ and either (under 1.\ and 2., $k$ is positive semidefinite)
			\begin{enumerate}
				\item $c$ satisfies the Chapman--Kolmogorov-type factorization \eqref{eq:condition-const},
				which holds for the squared Euclidean cost with $\lambda$ the Lebesgue measure
				(\Cref{cor:EOT_debias_Gaussian_like});
				\item $c$ is continuous and negative definite
				(\Cref{prop:Gaussian-factorisation-separable});
				\item $k=e^{-c/\eps}$ is a positive semidefinite kernel and $c$ is real-valued (\Cref{thm:EOT_debias_kernel_compact});
			\end{enumerate}
			\item $\eps=+\infty$ and $c$ is negative definite (\Cref{thm:eps_infty_neg_def}).
		\end{itemize}
		In all these cases, we can write
		\begin{equation}\label{eq:oteps_inf_intro}	\OT_c^\eps(\mu,\nu)=\inf_{\eta\in \calZ} \Psi(\mu,\eta)+\Psi(\nu,\eta),
		\end{equation}
		for some set $\calZ$ and some function $\Psi:\calP(\bbX)\times \calZ\to \R\cup\{+\infty\}$. In addition, we have that unbalanced versions \eqref{eq:cost_UOT-eps} of \eqref{eq:oteps} are also debiasable (\Cref{lem:OT_implies_UOT}). Finally, for any $\eps\in[0,+\infty]$, if $\OT_c^\eps$ is (strictly) debiasable, then so is $c$ by considering Dirac masses.
	\end{theorem}

	Negative and positive semidefinite kernels are defined in \Cref{prop:def_rep_kernels} along with their equivalent characterizations. Optimal transport is naturally related to $c$-transforms and $\inf$-formulas like \eqref{eq:feature_map_carac}. On the other hand, entropic OT has to do with soft-mins and sum-exp, due to the entropy/log-sum-exp duality. Sum-exp formulas moreover are related to reproducing kernels, which explains the coexistence of these various concepts in our proofs.

	While these factorizations serve primarily to establish debiasability, they also expose a rich interpolation structure. By extending the inf-representations to include an asymmetric time parameter $t \in (0,1)$, we show that finding the intermediary variable constructs natural entropic interpolation curves. More precisely, for $\lambda$ the Lebesgue measure over $\R^d$, $c(x,y)=|x-y|^2$ and $\psi_t(x,z)=\frac{|z-x|^2}{t}$ we obtain that
    \begin{equation}\label{eq:eta_t_Gaussien_intro}
			\OT_c^\eps(\mu,\nu)
			=\inf_{\eta\in\calP({\R^d})}  \OT^\eps_{\psi_{t}}(\mu,\eta)+\OT^\eps_{\psi_{1-t}}(\nu,\eta)+ \eps \KL(\eta|\lambda)+\frac{d\eps}{2}\log(t(1-t)\pi\eps)
		\end{equation}
    We show that the minimizer $\eta_t \in \calP(\calZ)$ coincides with the Schr\"odinger bridge marginals (\Cref{prop:equivalence_schrodinger}), providing a purely static, variational, barycentric formulation of the dynamical Schr\"odinger bridge. Iterating the inf-representation \eqref{eq:eta_t_Gaussien_intro} also provides a novel proof for the entropic Benamou-Brenier formula (\Cref{thm:entropic_BB}).
		
	\paragraph{Related works.} We refer to \cite{Sejourne2023} for an introductory review of our various examples of costs on measures spaces. We chose to focus on assessing debiasability and do not touch the very interesting direction of the statistical properties of estimating the debiased cost. Our formulas however sometimes correspond to doubly regularized entropic Wasserstein barycenters \cite{Chizat2025}, as seen in \eqref{eq:eta_t_Gaussien_intro}. Concerning $\Seps$, a self-contained discussion on the limiting behaviours for $\eps\to 0$ or $\eps\to \infty$ can be found in \cite{Neumayer2021}. As evidenced in \cite[Section 7.1]{lavenant2024riemannian},  neither  $\Seps$ nor $\sqrt{\Seps}$  satisfies the triangle inequality. Nevertheless, if $\bbX\subset \R^n$ is compact and $\exp(-c/\eps)$ is a positive definite and universal reproducing kernel, then \cite[Theorem 1]{feydy2019interpolating} shows that $\Seps$ is nonnegative, symmetric, separately convex and metrizes the convergence in law, i.e.\ in this case $\OT_c^\eps$ is debiasable. We refer to \cite{sejourne2019sinkhorn} for the extension to the unbalanced case, and to \cite{houry2026gromovwassersteinscalesquarednorms} for handling Gromov-Wasserstein settings based on similar arguments to \cite{feydy2019interpolating}. 
	
    \paragraph{Structure of the article.} \Cref{sec:charac} develops the general theory of debiasable
costs, outside of any optimal transport context: we prove the inf-representation characterization
(\Cref{prop:tpsd_feat_map}), and discuss strict debiasability (\Cref{corr:strict_tpsd}). \Cref{sec:eps0} reduces the unbalanced problem to the
balanced one (\Cref{lem:OT_implies_UOT}) and treats classical OT, first by a direct argument
(\Cref{prop:direct_proofs_ot}) and then, to prepare the entropic case, by lifting the
inf-representation of the ground cost to transport plans (\Cref{fact:ot}). 
\Cref{sec:eps_finite} addresses entropic OT for $\eps\in(0,+\infty)$ under progressively weaker
assumptions: costs admitting a log-sum-exp decomposition (\Cref{sec:EOT_lse_gauss}), continuous
negative definite costs (\Cref{sec:EOT_neg_kernel}), and finally merely positive semidefinite kernels
(\Cref{sec:selfantipolar}), where debiasability is obtained from a scale-invariant form of the
entropic dual and is shown to be \emph{equivalent} to the existence of an inf-representation indexed by
self-antipolar subsets of the cone of kernel mean embeddings. \Cref{sec:eps_infty} handles the limiting cost $\eps=+\infty$, for which negative definiteness turns
out to be not only sufficient but necessary. \Cref{sec:baryc_interpolation} discusses entropic interpolation and its relation to the Schrödinger bridge as a consequence of a time-dependent inf-representation. It also includes a new proof for the entropic Benamou-Brenier formula. Useful auxiliary results are collected in the Appendix. 

	\section{Characterization of debiasable costs}\label{sec:charac}
	
	We start with general results on (strictly) debiasable costs, outside of any optimal transport context.
	
	\subsection{Inf-representation of debiasable costs}
	
	Our main idea is to use the following result which characterizes any debiasable cost through an inf-representation.
	
	\medskip
	
	\begin{proposition}[Characterization of debiasable costs]\label{prop:tpsd_feat_map}Let $\calX$ be a set and take $c: \calX \times \calX \to \R\cup\{+\infty\}$.
		The following properties are equivalent
		\begin{enumerate}[labelindent=0cm,leftmargin=*,topsep=0.1cm,partopsep=0cm,parsep=0.1cm,itemsep=0.1cm,label=\roman*)]
			\item\label{it_tpsd} $c$ is symmetric and debiasable;
			\item\label{it_featMap} there exists a set $\calZ$ and a function $\psi: \calX\times \calZ \to \R\cup\{+\infty\}$ such that
			\begin{equation}\label{eq:feature_map_carac}
				c(x,y) = \inf_{z\in \calZ} \psi(x,z) + \psi(y,z).
			\end{equation}
		\end{enumerate}
		A couple $(\calZ, \psi)$ such that \eqref{eq:feature_map_carac} holds is called an \emph{inf-representation} of $c$.
		
		\medskip

        If $\calX$ is equipped with a (Hausdorff) topology, then $\calZ$ can be chosen (Hausdorff)
topological as well. Assume from now on that $\calX$ is Hausdorff. If $c(x,\cdot)$ is l.s.c.\ for every $x\in\calX$ and $x\mapsto c(x,x)$ is upper
semicontinuous, then $\psi$ can be chosen so that $\psi(x,\cdot)$ is l.s.c.\ on $\calZ$ for every
$x\in\calX$. If $c$ is
moreover jointly l.s.c., then,
equipping $\calX\times\calZ$ with the product topology, $\psi$ can be chosen jointly l.s.c.
	\end{proposition}
	
	We postpone the proof to \Cref{proof:tpsd_feat_map}. Eq.\eqref{eq:feature_map_carac} first appeared in \cite[Proposition 3.3]{aubin2022tropical}, we here expand the result to later tackle measurability issues through lower semicontinuity. The proof rests upon a constructive $\psi$ for which $\calZ=\calX \times \calX$ and $\psi$ is defined as: $\psi(x,(x,y))=c(x,x)/2$ and $\psi(x,(y,x))=c(x,y)- c(y,y)/2$ for all $x,y\in \calX$. All the other values of $\psi(x,(u,v))$, for $x,u,v\in \calX$ such that $x\not \in \{u,v\}$, are set to $+\infty$. Hence $\psi(x,z)$ is finite-valued only if $z$, seen as an assignment, has $x$ as an endpoint.
	
	\medskip
	
	In  (max,+)-terminology, a cost $c$ is debiasable if and only if $-c$ is tropically positive semidefinite, and the inf-representation is the (max,+)-analogue of a factorization of $c$, seen as matrix, through rank-1 operators. Since we are more interested in $\inf$ than $\sup$, we do not delve further into (max,+)-analysis concepts and refer to \cite{aubin2022tropical} for the latter.

	\paragraph{Examples of debiasable $c$ and inf-representations $(\calZ,\psi)$.}\label{examples_c} Note that if \eqref{eq:feature_map_carac} holds for $\calZ=\calX$ and $\psi=c$, then $c$ satisfies a triangular inequality and is thus a pseudo-metric if $c(x,x)=0$. The following are cases where $c\neq \psi$.
	\begin{enumerate}[label=\roman*),labelindent=0em,leftmargin=2em,topsep=0cm,partopsep=0cm,parsep=0cm,itemsep=2mm]
		\item For $\calX$ a Hilbert space and $c(x,y)=-(x,y)_\calX$, we can choose $\calZ=\calX$ and $\psi(x,z)=-\frac{1}{2}|x|^2_\calX+|x-z|_\calX^2$ for all $x,z\in \calX$.
		\item For $\calX$ a metric space such that there exist {\it $\delta$-midpoints} in the following sense 
		\begin{equation*}
			\forall \delta>0,\, \forall \, x,x'\in \calX, \, \exists y\in \calX , \, \max(d(x,y),d(x',y))\le \frac12 d(x,x')+\delta,
		\end{equation*}
		and $c(x,y)=d^2(x,y)$, by \cite[Lemma 3.2]{aubin2024beyond} we can choose $\calZ=\calX$ and $\psi(x,z)= 2 d^2(x,z)$ for all $x,z\in \calX$. This is the case of intrinsic metric spaces.
		\item Take $\calX$ to be any set, $\eps>0$ and $c(x,y)=-\eps \log(k(x,y))$ for all $x,y\in \calX$, where $k: \calX \times \calX \to \R_{>0}$ is a bounded positive semidefinite (psd) kernel  (see the definition in \Cref{prop:def_rep_kernels}). Then $c$ is debiasable (see \Cref{lem:inf-rep_kernel}) with $\psi$ given in \Cref{thm:selfantipolar} in the particular case of Dirac masses. 
	\end{enumerate}
	\medskip 
	
	As intuited from the metric example ii), eq.\eqref{eq:feature_map_carac} can be interpreted as a generalization of the concept of metric induced by $\psi$ and the case where there is a minimizer as an extension of the midpoint property.	The last example shows however that we can have $\calZ\neq \calX$ and fairly complex inf-representations. Even for finite sets $\calX$ with cardinality larger than $5$, $\calZ$ may have to have a strictly larger cardinality than $\calX$, see \cite[Proposition 3.3-Examples]{aubin2022tropical} for a counter-example based on the bipartite graph.
	
	\medskip

	We now detail some operations preserving debiasability which will be of use for studying optimal transport. They can also be used to check the debiasability of a given $c$.
	
	\begin{lemma}\label{lem:operations_preserving} Let $c:\calX\times \calX\to\R\cup \{+\infty\}$ be a debiasable cost, and take any $g:\calX\to\R\cup \{+\infty\}$ and $f:\calX\times \calX\to\R\cup \{+\infty\}$. Let $\calW$ be some (parameter) set and for all $w\in \calW$, let $c_w:\calX\times \calX\to\R\cup \{+\infty\}$ be a family of debiasable costs depending on the parameter $w\in\calW$. Let also $\{\alpha_w\}_{w\in\calW}\subset \R_{\ge 0}$. Then the following costs are debiasable:
		\begin{enumerate}[label=\roman*),labelindent=0em,leftmargin=2em,topsep=0cm,partopsep=0cm,parsep=0cm,itemsep=2mm]
			\item \label{fact:sum-c} $\displaystyle \sum_{w\in\calW} \alpha_w c_w$;
			\item \label{fact:inf-c} $\displaystyle \inf_{w\in\calW} c_w$;
			\item \label{fact:g-c} $\calX\times \calX \ni (x,y)\mapsto c(x,y)+g(x)+g(y)$;
			\item \label{fact:reg-c} $\displaystyle \calX \times \calX \ni (x,y)\mapsto \inf_{\tilde x, \tilde y\in\calX} f(\tilde x,x)+f(\tilde y, y)+c(\tilde x,\tilde y)$.
		\end{enumerate}
	\end{lemma} 
	\begin{proof} Properties \ref{fact:sum-c} and \ref{fact:g-c} are immediate. Concerning \ref{fact:inf-c}, set $\widetilde{c} = \inf_{w\in\calW} c_w$. Then, for any $w\in \calW$, it holds that for all $x,y\in\calX$,
		\[
		\begin{aligned}
			c_w(x,y)-\frac12 \tilde{c}(x,x)-\frac12 \tilde{c}(y,y)&=c_w(x,y)-\frac12\inf_{w\in \calW} c_w(x,x)-\frac12\inf_{w\in \calW} c_w(y,y) \\
			&\ge c_w(x,y)-\frac12 c_w(x,x)-\frac12 c_w(y,y)\ge 0,
		\end{aligned}
		\]
		and we conclude taking the infimum over $w\in \calW$.

		Concerning \ref{fact:reg-c}, set $\displaystyle \widetilde{c}(x,y) = \inf_{\tilde x, \tilde y\in\calX} f(\tilde x,x)+f(\tilde y, y)+c(\tilde x,\tilde y)$. Then, we have for all $x,y, \tilde x, \tilde y\in\calX$,
		\begin{multline*}
			f(\tilde x,x)+f(\tilde y, y)+c(\tilde x,\tilde y) -\frac12 \tilde c(x,x)-\frac12 \tilde c(y,y)  \\
			\begin{aligned}
				&\ge f(\tilde x,x)+f(\tilde y, y)+c(\tilde x,\tilde y)-\frac12 \left[ f(\tilde x,x)+f(\tilde x, x)+c(\tilde x,\tilde x)\right]-\frac12 \left[ f(\tilde y,y)+f(\tilde y, y)+c(\tilde y,\tilde y)\right] \\
				&=c(\tilde x,\tilde y)-\frac12  c(\tilde x,\tilde x)-\frac12 c(\tilde y,\tilde y)\ge 0,
			\end{aligned}
		\end{multline*}
		and we conclude taking the infimum over $\tilde x,\tilde y \in \calX$.
	\end{proof}

	\subsection{Injectivity and strictly debiasable costs}
	
	We now discuss situations where debiasability is strict.

	\begin{corollary}[Characterization of strictly debiasable costs for $\inf=\min$]\label{corr:strict_tpsd}
		If $c$ is debiasable, and furthermore the infimum is attained in \eqref{eq:feature_map_carac} on a set $\calZ_{x,y}\neq \emptyset$ for all $x,y\in \calX$, then the following properties are equivalent:
		\begin{enumerate}[labelindent=0cm,leftmargin=*,topsep=0.1cm,partopsep=0cm,parsep=0.1cm,itemsep=0.1cm,label=\roman*)]
			\item\label{it_tpsd_strict} $c$ is strictly debiasable, i.e.\ for all $x\neq y \in \calX$, $c(x,x)+c(y,y) < c(x,y)+c(y,x)$; 
			\item\label{it_featMap_strict} for all $x\neq y \in \calX$, we have $\emptyset=\calZ_{x,x} \cap \calZ_{y,y}$.
		\end{enumerate}
	\end{corollary}
	\begin{proof}
		We are going to show equivalence of the negations:
		\begin{enumerate}[labelindent=0cm,leftmargin=*,topsep=0.1cm,partopsep=0cm,parsep=0.1cm,itemsep=0.1cm,label=\roman*)]
			\item[a)]\label{it_tpsd_strict_neg} there exists $x\neq y \in \calX$ such that $c(x,x)+c(y,y) = c(x,y)+c(y,x)$;
			\item[b)]\label{it_featMap_strict_neg} there exists $x\neq y\in \calX$  such that $\emptyset\neq \calZ_{x,x} \cap \calZ_{y,y}$.
		\end{enumerate}
		
		\medskip
		
		a) $\Rightarrow$ b). Fix any $x\neq y \in \calX$ as in a). Take $z_{x,y}\in \calZ_{x,y}$. Then ``$c(x,x)+c(y,y) = c(x,y)+c(y,x)$'' reads ``$ \inf_{z\in \calZ} \psi(x,z) +  \inf_{z'\in \calZ} \psi(y,z') = \psi(x,z_{x,y})+\psi(y,z_{x,y})$''. Hence $z_{x,y}\in \calZ_{x,x} \cap \calZ_{y,y}$.
		
		b) $\Rightarrow$ a). Fix any $x\neq y\in \calX$ as in b). Take $z_{x,y}\in \calZ_{x,x} \cap \calZ_{y,y}$. Then 
		\begin{equation*}
			c(x,x)+c(y,y) = 2\psi(x,z_{x,y})+2\psi(y,z_{x,y}) \ge 2\inf_{z\in \calZ} \psi(x,z)+\psi(y,z)= 2c(x,y).
		\end{equation*}
		The reverse inequality holds because $c$ is assumed to be debiasable.
		
	\end{proof}

	\paragraph{Direct proofs of strict debiasability.} Direct proofs to obtain that $c$ is strictly debiasable can be obtained showing that $c_0$ dominates a strictly debiased $\widetilde{c}_0$. In \cite[Section 2.3]{feydy2019interpolating} this is done for $\calX$ a convex subset of a Banach space using a symmetrized Bregman divergence. The reasoning goes as follows. Take a symmetric separately convex $c$, if $x\mapsto f(x)= -c(x,x)$ is additionally (strictly) convex, then $c_0$ is clearly separately (strictly) convex. It is furthermore debiasable for points $x,y$ such that there exists $p_x\in \partial f(x)$ and such that $-p_x\in \partial(c(x,\cdot))(x)$, indeed convexity of $c$ then gives $c(x,x)+\bracket{y-x,-p_x}\le c(x,y)$ and $c(y,y)+\bracket{x-y,-p_y}\le c(x,y)$, summing we obtain $2c_0(x,y)\ge \bracket{x-y,p_x-p_y}\ge 0$ since $\partial f$ is monotone. Hence (strict) debiasability is a consequence of a (strict) convexity of $x\mapsto -c(x,x)$ and assumptions on subdifferentials. Unfortunately this case is rather rare.
	
	Alternatively, instead of doing a step back from $c$ to $\psi$ via $\calZ_{x,y}$ as in \Cref{corr:strict_tpsd}, one can move forward by defining
	\begin{equation}\label{eq:def_tilde_c_one_step}
		\tilde{c}(x,y)=\inf_{x'\in \calX}c(x,x')+c(x',y)-c(x',x').
	\end{equation}
	Obviously $\tilde{c}$ is always debiasable by \Cref{prop:tpsd_feat_map}. Moreover we have that
	\begin{equation*}
		\text{$\frac{c(x,x)}{2}+\frac{c(y,y)}{2}\stackrel{\eqref{eq:def_tilde_c_one_step}+c \text{ debiasable}}{\le} \tilde{c}(x,y) \stackrel{x'=x}{\le} c(x,y)$.}
	\end{equation*}
	Setting $x=y$ in the above gives $\tilde{c}(x,x)=c(x,x)$. So if $\tilde{c}$ is strictly debiasable, then so is $c$.

	\section{Debiasing (U)OT for $\eps=0$}\label{sec:eps0}

	We start with a simple lemma showing that we can focus on $\OT_c^\eps$ 
    exclusively and deduce from it the debiasability of $\UOT_{c,F}^\eps$.

	\begin{lemma}\label{lem:OT_implies_UOT} Let $\eps\in[0,+\infty)$. If $\OT_c^\eps$ is debiasable on $\calP(\bbX)$ and $\UOT^\eps_{c,F}>-\infty$ on $\calM^b_+(\bbX)\times\calM^b_+(\bbX)$, then $\UOT_{c,F}^\eps$ is debiasable on $\calM^b_+(\bbX)$.
\end{lemma}
The assumption $\UOT^\eps_{c,F}>-\infty$ holds for instance if $\eps>0$ and $c$ is bounded from below, or if $\eps=0$, $c\ge 0$ and $F\ge 0$.

	\begin{proof}
		We treat $\eps>0$; for $\eps=0$ the same computation applies after discarding every $\KL$ term. We can easily rewrite $\UOT_{c,F}^\eps$ of \eqref{eq:cost_UOT-eps} in a different form as follows
		\begin{equation}\label{eq:UOT-explicit}
			\UOT_{c,F}^\eps(\mu,\nu)=\inf_{\substack{m\ge 0, \tilde\pi \in \Pi(\tilde \mu,\tilde \nu) \\ \tilde \mu, \tilde \nu \in\calP(\bbX)}} 
			m\int_{\bbX\times\bbX} c(x,y) \de\tilde\pi(x,y)+\eps\KL(m\tilde\pi| \mu\otimes \nu)+F(m\tilde \mu,\mu)+F(m\tilde \nu,\nu),
		\end{equation}
		where we made explicit in the optimization problem: $m=\int \de \pi$ the total mass of $\pi$, $\pi=m\tilde \pi$ where $\tilde \pi\in\calP(\bbX \times \bbX)$ is the mass one normalization of $\pi$, and $(\tilde \mu, \tilde \nu)$ are the marginals of $\tilde\pi$. 
        
        For $m=0$, we get $F(0,\mu)+F(0,\nu)$ in \eqref{eq:UOT-explicit}, which is debiasable. Assume now that $m>0$. Since the minimization of $\KL(m\tilde\pi | \mu\otimes \nu)$ forces $\tilde\mu\ll\mu$, $\tilde\nu\ll\nu$, and therefore $\tilde\mu\otimes\tilde\nu\ll\mu\otimes\nu$, using \eqref{eq:KL-product} we can write
		\[
		\begin{aligned}
			\KL(m\tilde\pi|\mu\otimes\nu)&=m\int_{\bbX\times\bbX} \log\left(m\frac{\de\tilde\pi}{\de\mu\otimes\nu}(x,y)\right)\de\tilde\pi(x,y)\\
			&=m\int_{\bbX\times\bbX} \left[\log\left(m\frac{\de\tilde\pi}{\de\tilde\mu\otimes\tilde\nu}(x,y)\right)+\log\left(\frac{\de\tilde\mu\otimes\tilde\nu}{\de\mu\otimes\nu}(x,y)\right)\right]\de\tilde\pi(x,y)\\
			&=m\KL(\tilde\pi|\tilde\mu\otimes\tilde\nu)+m\KL(\tilde\mu|\mu)+m\KL(\tilde\nu|\nu)+m\log m.
		\end{aligned}
		\]
        Using this formula in \eqref{eq:UOT-explicit}, we obtain
		\begin{align*}
			\UOT_{c,F}^\eps(\mu,\nu)&=\inf_{\substack{m\ge 0,\tilde\pi \in \Pi(\tilde \mu,\tilde \nu) \\ \tilde \mu, \tilde \nu \in\calP(\bbX)}} 
			m\int_{\bbX\times\bbX} c(x,y) \de\tilde\pi(x,y)+\eps m\KL(\tilde\pi|\tilde\mu\otimes \tilde\nu)\\
            &\hspace{1cm}+\eps m\KL(\tilde\mu|\mu)+\eps m\KL(\tilde\nu|\nu)+F(m\tilde \mu,\mu)+F(m\tilde \nu,\nu)+\eps m\log m\\
			&=\inf_{m\ge 0,\tilde \mu, \tilde \nu \in\calP(\bbX)} 
			m\OT_c^\eps(\tilde\mu,\tilde\nu)+\eps m\KL(\tilde\mu|\mu)+\eps m\KL(\tilde\nu|\nu)+F(m\tilde \mu,\mu)+F(m\tilde \nu,\nu)+\eps m\log m.
		\end{align*}    
		Since we assumed that $ \OT_c^\eps(\cdot,\cdot)$ is debiasable, for fixed $m>0$, $m\,\OT^\eps_c$ is debiasable by
		\Cref{lem:operations_preserving}-\ref{fact:sum-c}; hence, setting
		$f_m(\alpha,\beta)\coloneqq\eps m\KL(\alpha|\beta)+F(m\alpha,\beta)$ for $\alpha\in\calP(\bbX)$
		and $f_m(\alpha,\beta)\coloneqq+\infty$ otherwise, the map $(\mu,\nu)\mapsto\inf_{\tilde\mu,\tilde\nu}m\,\OT^\eps_c(\tilde\mu,\tilde\nu)
		+f_m(\tilde\mu,\mu)+f_m(\tilde\nu,\nu)$ is debiasable by
		\Cref{lem:operations_preserving}-\ref{fact:reg-c}, and adding the constant $\eps m\log m$
		preserves debiasability. We conclude taking the infimum over
		$m\ge0$ and using \Cref{lem:operations_preserving}-\ref{fact:inf-c}.
	\end{proof}

	Focusing now on $\eps=0$ and $\OT_{c}$, we can achieve a direct proof of debiasability without using the $\inf$-representation.

    	\begin{proposition}[Direct proof for optimal transport]\label{prop:direct_proofs_ot}
		Let $\bbX$ be a Polish space and $c:\bbX\times\bbX\to\R\cup\{+\infty\}$ be symmetric, Borel and finite on the diagonal, and let $\calP_c(\bbX)$ be the class \eqref{eq:def_P_k} of diagonally integrable measures. If $c$ is debiasable, then $\OT_c$ is debiasable on $\calP_c(\bbX)$. Moreover, if $c$ is strictly debiasable, then $\OT_c(\mu,\nu)>\frac12\OT_c(\mu,\mu)+\frac12\OT_c(\nu,\nu)$ for all $\mu\neq\nu$ in $\calP_c(\bbX)$ such that \eqref{eq:W_c} admits a minimizer. Finally, if $\OT_c$ is (strictly) debiasable, then so is $c$.
	\end{proposition}
	\begin{proof}%
		Assume that $c$ is debiasable, let $s(x)\coloneqq\frac12c(x,x)$, so that the debiased cost is $c_0=c-s\oplus s\ge0$, and let $\mu,\nu\in\calP_c(\bbX)$. By \Cref{lem:prob-form},
		\begin{equation}\label{eq:W_c-prob}
			\OT_c(\mu,\nu)=\inf_{(T_1)_\# \lambda=\mu,\ (T_2)_\# \lambda=\nu}  \int_{[0,1]} c(T_1(w),T_2(w)) \,\de \lambda(w),
		\end{equation}
		the infimum being over $\calL$-measurable maps $T_1,T_2:[0,1]\to\bbX$. For such $T_1$ and $T_2$, the pairs $(T_1,T_1)$ and $(T_2,T_2)$ are admissible for $\OT_c(\mu,\mu)$ and $\OT_c(\nu,\nu)$, and $s(T_1),s(T_2)\in L^1(\lambda)$. Hence $\OT_c(\mu,\mu)\le\int c(T_1,T_1)\,\de\lambda=2\int s\,\de\mu$, the reverse inequality following from $c\ge s\oplus s$, and likewise for $\nu$. Therefore
		\begin{align*}
		    \int_{[0,1]} c(T_1,T_2)\, \de \lambda - \frac12 \OT_c(\mu,\mu)-\frac12 \OT_c(\nu,\nu)
			&=\int_{[0,1]} \Big[ c(T_1,T_2)-\frac12 c(T_1,T_1)-\frac12 c(T_2,T_2) \Big]\de \lambda\\
			&=\int_{[0,1]}c_0(T_1,T_2)\,\de\lambda \ \ge\ 0,
		\end{align*}
		and taking the infimum over $T_1$ and $T_2$ shows that $\OT_c$ is debiasable on $\calP_c(\bbX)$.
		
		If $c$ is strictly debiasable, $\mu\neq\nu$ and $\pi$ is optimal in \eqref{eq:W_c}, then $\pi=(T_1,T_2)_\#\lambda$ for some admissible $T_1,T_2$ by \Cref{lem:prob-form}, so that the infimum in \eqref{eq:W_c-prob} is attained. Since $\mu\neq\nu$, the set $\{T_1\neq T_2\}$ is not $\lambda$-negligible, and $c_0>0$ off the diagonal gives $\int c_0(T_1,T_2)\,\de\lambda>0$. The converse follows by considering Dirac masses $\mu=\delta_x$ and $\nu=\delta_y$, which belong to $\calP_c(\bbX)$, for which $\OT_c(\delta_x,\delta_y)=c(x,y)$.
	\end{proof}
    
	When trying to mimic the same proof with entropic regularization, the critical issue is that now $T_1$ and $T_2$ get entangled through $\pi$ in the term $\eps\KL(\pi|\mu\otimes \nu)$, and their contributions are hard to separate. Rather than showing directly that the debiased cost is nonnegative as in \Cref{prop:direct_proofs_ot}, the main idea will be to make explicit an inf-representation \eqref{eq:feature_map_carac}, obtained performing a lifting to work on plans, as well as using the chain rule for $\KL$. To foster intuition, we first do an alternative proof to \Cref{prop:direct_proofs_ot} using \Cref{prop:tpsd_feat_map} and \Cref{lem:comm_inf-int}.

        \begin{lemma}\label{fact:ot}
		Let $\bbX$ and $\calZ$ be Polish spaces, $c:\bbX\times\bbX\to\R\cup\{+\infty\}$ be Borel, and $\psi:\bbX \times \calZ\to \R\cup\{+\infty\}$ be jointly Borel, such that, for all $x,y\in\bbX$,
		\[
		c(x,y) = \inf_{z\in \calZ} \psi(x,z) + \psi(y,z).
		\]
		Then, for all $\mu,\nu\in\calP_c(\bbX)$,
		\begin{equation}\label{eq:ot-explicit}
			\OT_{c}(\mu,\nu)=\inf_{\eta\in \calP(\calZ)} \OT_{\psi}(\mu,\eta)+\OT_{\psi}(\nu,\eta).    
		\end{equation}
	\end{lemma}
    \noindent The lemma applies to every Borel debiasable cost, the inf-representation built in \Cref{proof:tpsd_feat_map}, with $\calZ=\bbX\times\bbX$, being Borel as soon as $c$ is.
	\begin{proof} The function $s\coloneqq\frac12c(\cdot,\cdot)=\inf_z\psi(\cdot,z)$ is $\mu$- and $\nu$-integrable, hence finite a.e. Replacing $c$ by $c-s\oplus s\ge0$ and $\psi$ by $\psi-s\ge0$, with $\infty-\infty\coloneqq+\infty$, preserves the inf-representation and shifts both sides of \eqref{eq:ot-explicit} by $\int s\,\de\mu+\int s\,\de\nu$, so that we may assume $\psi\ge0$. 
		We first show that the r.h.s.\ of \eqref{eq:ot-explicit} equals
        \begin{equation}\label{eq:ot-explicit-3plan}
            \inf_{\substack{\eta\in \calP(\calZ)\\ \gamma\in \Pi(\mu,\nu,\eta)}} \int_{\bbX\times\bbX\times \calZ} \left[ \psi(x,z) + \psi(y,z) \right] \de \gamma(x,y,z),
        \end{equation}
        where $\Pi(\mu,\nu,\eta)$ is the set of probability measures on $\bbX\times\bbX\times\calZ$ with marginals $\mu$, $\nu$ and $\eta$. For $\gamma\in\Pi(\mu,\nu,\eta)$, the plans $\pi_1\coloneqq(\p_1,\p_3)_\#\gamma\in\Pi(\mu,\eta)$ and $\pi_2\coloneqq(\p_2,\p_3)_\#\gamma\in\Pi(\nu,\eta)$ satisfy
		\[
			\int_{\bbX\times \bbX\times \calZ} \left[\psi(x,z) + \psi(y,z)\right]\de\gamma
			= \int_{\bbX\times \calZ} \psi\, \de\pi_1+ \int_{\bbX\times \calZ} \psi\, \de\pi_2
			\ \ge\ \OT_{\psi}(\mu,\eta)+\OT_{\psi}(\nu,\eta),
        \]
		so that \eqref{eq:ot-explicit-3plan} is larger than the r.h.s.\ of \eqref{eq:ot-explicit}. Conversely, fix $\eta\in\calP(\calZ)$ and minimizing sequences $(\pi_1^n)_n\subset\Pi(\mu,\eta)$ and $(\pi_2^n)_n\subset\Pi(\nu,\eta)$ for $\OT_\psi(\mu,\eta)$ and $\OT_\psi(\nu,\eta)$. Their gluings $\overline\gamma^n\in\Pi(\mu,\nu,\eta)$ along $\eta$ (\Cref{proof:chain rule}) satisfy $(\p_1,\p_3)_\#\overline\gamma^n=\pi_1^n$ and $(\p_2,\p_3)_\#\overline\gamma^n=\pi_2^n$, so that
		\begin{align*}
		    \inf_{\gamma\in\Pi(\mu,\nu,\eta)}\int\big[\psi(x,z)+\psi(y,z)\big]\de\gamma
		\ &\le\ \int\big[\psi(x,z)+\psi(y,z)\big]\de\overline\gamma^n\\
		&=\int\psi\,\de\pi_1^n+\int\psi\,\de\pi_2^n
		\ \xrightarrow[n\to\infty]{}\ \OT_\psi(\mu,\eta)+\OT_\psi(\nu,\eta),
		\end{align*}
		and minimizing over $\eta$ shows that \eqref{eq:ot-explicit-3plan} is smaller than the r.h.s.\ of \eqref{eq:ot-explicit}.

		By the proof of \Cref{lem:prob-form}, every $\gamma\in\calP(\bbX\times\bbX\times\calZ)$ is the law of a measurable map $(T_1,T_2,T_3):[0,1]\to\bbX\times\bbX\times\calZ$, and conversely every such map has a law in $\Pi(\mu,\nu,(T_3)_\#\lambda)$ as soon as $(T_1)_\#\lambda=\mu$ and $(T_2)_\#\lambda=\nu$. Hence \eqref{eq:ot-explicit-3plan} equals
		\[
			\inf_{(T_1)_\#\lambda=\mu,\ (T_2)_\#\lambda=\nu}\ \inf_{T_3} \int_{[0,1]}\left[\psi(T_1(w),T_3(w)) +\psi(T_2(w),T_3(w)) \right] \de \lambda(w),
		\]
		the map$T_3$ being unconstrained since $\eta$ is free. For fixed $T_1,T_2$, \Cref{lem:comm_inf-int}, applied on the complete probability space $([0,1],\calL,\lambda)$ to the nonnegative $\calL\otimes\calB(\calZ)$-measurable function $(w,z)\mapsto\psi(T_1(w),z)+\psi(T_2(w),z)$, gives
		\[
			\inf_{T_3} \int_{[0,1]} \left[\psi(T_1,T_3) +\psi(T_2,T_3) \right] \de \lambda = \int_{[0,1]} \inf_{z\in \calZ} \left[\psi(T_1(w),z) +\psi(T_2(w),z) \right] \de \lambda(w)
			=\int_{[0,1]} c(T_1,T_2)\,\de \lambda,
		\]
		and taking the infimum over $T_1,T_2$ yields $\OT_c(\mu,\nu)$ by \Cref{lem:prob-form}.
	\end{proof}

	\begin{example}
		If $\bbX=\R^d$ and $c(x,y)=|x-y|^2$ with $|\cdot|$ the Euclidean norm, then $\calZ = \R^d$ and  $\psi(x,z)=2|x-z|^2$ for all $x,z\in\R^d$. In addition, $\OT_{\psi}(\mu,\eta)=\inf_{\pi\in\Pi(\mu,\eta)} \int_{\bbX\times \calZ} 2|x-z|^2  \de \pi(x,z)=2 W_2^2(\mu,\eta)$ and we obtain again the midpoint formula for the geodesic distance $W_2$.
	\end{example}
    	\section{Debiasing entropic (U)OT for $\eps\in (0,+\infty)$} \label{sec:eps_finite}

	In this section, we study the debiasability of $\OT_c^\eps$ for $\eps\in (0,+\infty)$. We recall that by \Cref{lem:OT_implies_UOT} we only have to study the balanced case to deduce the unbalanced one.

	\subsection{Preliminaries}

	Fix $\eps\in (0,+\infty)$. Canonically for any cost $c: \bbX \times \bbX \to \mathbb{R}\cup\{+\infty\}$, we define
	\begin{equation}
		\forall x,y\in \bbX, \quad k(x,y) \coloneqq \exp\big( -c(x,y)/\eps\big).
	\end{equation}
	Introducing this kernel is motivated by the last formula for $\OT_c^\eps$ in \eqref{eq:oteps}. Debiasability of $\OT_c^\eps$ can thus be studied either through properties of $c$ or $k$. For any debiasable cost $c$ which has an inf-representation $(\calZ, \psi)$, i.e.\ such that $\displaystyle c(x,y) = \inf_{z\in \calZ} \psi(x,z) + \psi(y,z)$ for all $x,y\in \bbX$, we have
	\begin{equation}\label{eq:sup-exp_representation_k}
		\forall x,y\in \bbX, \quad  k(x,y)=\sup_{z\in \calZ} e^{-\psi(x,z)/\eps} e^{-\psi(y,z)/\eps}.
	\end{equation}
	Unfortunately having a debiasable ground cost $c$ is merely a necessary condition and no longer sufficient in entropic contexts. It is necessary since $\OT_c^\eps(\delta_x,\delta_y)=c(x,y)$ for all $x,y\in \bbX$. It is not sufficient as seen from the following example.
	\begin{example}\label{ex:counter-example_debiasable}
		Consider $\bbX=\{1,2,3 \}$ and the (debiased) cost
		\begin{equation}\label{eq:counter-example_debiasable_eq}
			C=[c(x,y)]_{x,y\in \bbX}=
			\begin{pmatrix}
				0 & 1 & 0 \\
				1 & 0 & 0 \\
				0 & 0 & 0
			\end{pmatrix}.
		\end{equation}
		Take $\mu=\frac12 \delta_1+\frac12\delta_2$ and $\nu=\delta_3$. Then, the only admissible plan for $(\mu,\nu)$ is the completely independent one, i.e.\ the tensor product measure, $\pi=\mu\otimes\nu$ and $\OT^\eps_{c}(\mu,\nu)=0$ since it costs nothing to send mass from $1,2$ to $3$. On the other hand, $\OT^\eps_{c}(\mu,\mu)>0$. Note that since $c\mapsto \OT^\eps_{c}(\mu,\nu)$ is 1-Lipschitz in $L^\infty$, the example of \eqref{eq:counter-example_debiasable_eq} is not isolated, and that the associated kernel $k=e^{-c/\eps}$ is not positive semidefinite, consistently with \Cref{thm:EOT_debias_kernel_compact}: testing against $a=(1,1,-2)$ gives $\sum_{i,j}a_ia_jk(x_i,x_j)=2(e^{-1/\eps}-1)<0$.
	\end{example}

	We thus require extra assumptions on $c$ than being merely debiasable. Some interesting examples of costs will be related to negative definite kernels and their positive semidefinite counterparts, whose standard definitions and characterizations we recall now:
	\begin{proposition}[\cite{aronszajn50theory} and Prop 3.2 \cite{Berg1984}] \label{prop:def_rep_kernels}
		Let $k:\bbX\times \bbX \to \R$ be defined over a set $\bbX$. Then the following properties are equivalent
		\begin{enumerate}[labelindent=0cm,leftmargin=*,topsep=0.1cm,partopsep=0cm,parsep=0.1cm,itemsep=0.1cm,label=\roman*)]
			\item\label{it_pos_def} $k$ is positive semidefinite (psd), i.e.\ $k(x,y)=k(y,x)$ and $\displaystyle \sum_{i,j=1}^N a_ia_j k(x_i,x_j) \geq 0$, for all $N\in \mathbb{N}_{>0}$, all $\{x_i\}_{i=1}^N\subset \bbX$ and all $\{a_i\}_{i=1}^N\subset \R$;
			\item\label{it_pos_def_prof} $k$ is a reproducing kernel, i.e.\ there exists a Hilbert space $H_{} $ and a mapping $\phi:\bbX \to H_{}$ such that
			\begin{equation}\label{eq:pos_def}
				k(x,y)=(\phi(x),\phi(y))_{H_{}}.
			\end{equation}
		\end{enumerate}
		Let $c:\bbX\times \bbX \to \R$. Then the following properties are equivalent
		\begin{enumerate}[labelindent=0cm,leftmargin=*,topsep=0.1cm,partopsep=0cm,parsep=0.1cm,itemsep=0.1cm,label=\roman*)]
			\item\label{it_neg_def} $c$ is negative definite, i.e.\ $c(x,y)=c(y,x)$ and $\displaystyle \sum_{i,j=1}^N a_ia_j c(x_i,x_j) \leq 0$, for all $N\in \mathbb{N}_{>0}$, all $\{x_i\}_{i=1}^N\subset \bbX$ and all $\{a_i\}_{i=1}^N\subset \R$ such that $\displaystyle \sum_{i=1}^N a_i=0$;
			\item\label{it_neg_def_distance} there exists a Hilbert space $H_c $, a mapping $\phi:\bbX \to H_c$ and a function $f:\bbX\to \R$ such that
			\begin{equation}\label{eq:neg_def}
				c(x,y)=f(x)+f(y)+|\phi(x)-\phi(y)|^2_{H_c}.
			\end{equation}
			Note that we necessarily have $f(x)=\frac{c(x,x)}{2}$, taking $x=y$ in \eqref{eq:neg_def}.
		\end{enumerate}
	\end{proposition}

	The following captures the key relations between these notions, the first equivalence stemming from \cite[Prop.2.7]{Berg1984}. Given $c:\bbX\times \bbX \to \R$ symmetric, Borel and bounded from below, we have
	\begin{align*}
		\text{$c$ is negative definite as in \eqref{eq:neg_def}} &\Longleftrightarrow \text{$k=e^{-c/\eps}$ is psd for all $\eps>0$ (i.e.\ $e^{-c}$ is infinitely divisible)} \\
		&\implies \text{$k=e^{-c/\eps}$ is psd for some $\eps>0$  (assumption used in \cite{feydy2019interpolating})} \\
		&\Longleftrightarrow \text{$\OT^\eps_{c}$ is debiasable (\Cref{thm:EOT_debias_kernel_compact})} \\
		&\implies \text{$c$ is debiasable (\Cref{lem:inf-rep_kernel}).}
	\end{align*}
For $\eps=+\infty$, \Cref{thm:eps_infty_neg_def}-\ref{it_infty_char} shows that debiasability of $\OT^\infty_c$ is equivalent to negative definiteness of $c$ itself.

	\begin{lemma}[Debiasability of log of kernel]\label{lem:inf-rep_kernel} Let $k:\bbX\times \bbX\to \R_{\ge 0}$ be symmetric and such that $k(x,x) k(y,y) \ge k(x,y)^2$ for all $x,y \in \bbX$. Define for $\eps >0$ and $x,y\in \bbX$,
		\begin{equation}\label{eq:def_c_kernel}
			c(x,y) \coloneqq -\eps \log k(x,y) \in (-\infty, +\infty].
		\end{equation}
		Then $c$ is debiasable. Moreover,  $c$ is strictly debiasable if and only if $k(x,x) k(y,y) > k(x,y)^2$ for $x\neq y\in \bbX$.
	\end{lemma}
	In particular, if $k$ is psd, the assumed inequality is precisely Cauchy-Schwarz owing to \eqref{eq:pos_def}.
	\begin{proof} The following equivalences hold for all $x,y\in \bbX$:
		\begin{align*}
			2c(x,y) \ge c(x,x)+c(y,y) &\Leftrightarrow -2 \log k(x,y) \ge - \log(k(x,x)k(y,y)) \\
			&\Leftrightarrow k(x,x) k(y,y) \ge k(x,y)^2.
		\end{align*}
		Hence $c$ is debiasable. The strict case is a consequence of the same equivalences studied with strict inequalities for $x\neq y \in \bbX$.
	\end{proof}
	\begin{remark}
		The issue with \eqref{eq:sup-exp_representation_k} is that it holds for any debiasable $c$ and $k=e^{-c/\eps}$, even for $k$ not a reproducing kernel. If we had instead of the $\sup$ in \eqref{eq:sup-exp_representation_k} an integral representation as in $k(x,y)=\int_{z\in \calZ} e^{-\psi(x,z)} e^{-\psi(y,z)} \de\lambda(z)$ with $\lambda\in\calM_+(\calZ)$ a nonnegative measure, then the kernel would be completely positive, i.e.\ a $L^2$-product of two vectors with nonnegative components. Unfortunately in general, for reproducing kernels $k$, we only have a Mercer-like expression $k(x,y)=\int_{z\in \calZ} \rho(x,z)\rho(y,z) \de\lambda(z)$ with $\rho$ having possibly negative values.
	\end{remark}

    In the coming sections, we will gradually weaken our assumptions and provide several inf-representations, going from costs that are
	translations of their log-sum-exp approximation (\Cref{sec:EOT_lse_gauss}), to negative definite
	costs (\Cref{sec:EOT_neg_kernel}), to conclude with merely positive semidefinite kernels
	(\Cref{sec:selfantipolar}), where the assumption turns out to be necessary as well.

	\subsection{Log-sum-exp approximation and Gaussian/Euclidean-like settings}\label{sec:EOT_lse_gauss}

	Actually, it is possible to build a log-sum-exp approximation of any debiasable $c$ to achieve an entropic cost that is debiasable. More precisely, for any measurable function $\psi:\bbX\times\calZ\to\R\cup\{+\infty\}$ and nonnegative $\sigma$-finite measure $\lambda\in\calM_+(\calZ)$ such that $\int_{\calZ} e^{-\frac{2\psi(x,z)}{\eps}}\de\lambda(z)<+\infty$ for any $x\in\bbX$, consider the log-sum-exp cost and kernel functions defined for all $x,y\in \bbX$ as
	\begin{gather}\label{eq:lse-cost}
		c_{\eps,\lambda}(x,y) \coloneqq -\eps \log \int_{\calZ} e^{-\frac{\psi(x,z)+\psi(y,z)}{\eps}}\de\lambda(z),\\
		k_{\eps,\lambda}(x,y) \coloneqq e^{-\frac{c_{\eps,\lambda}(x,y)}{\eps}} = \int_{\calZ} e^{-\frac{\psi(x,z)+\psi(y,z)}{\eps}}\de\lambda(z). \label{eq:lse-kernel}
	\end{gather}
	The values of the cost $c_{\eps,\lambda}$ always belong to $(-\infty,+\infty]$ since by the Cauchy-Schwarz inequality and the integrability assumption on $\lambda$
	\[
	\int_{\calZ} e^{-\frac{\psi(x,z)+\psi(y,z)}{\eps}}\de\lambda(z)\le \left(\int_{\calZ} e^{-\frac{2\psi(x,z)}{\eps}}\de\lambda(z)\right)^\frac12\left(\int_{\calZ} e^{-\frac{2\psi(y,z)}{\eps}}\de\lambda(z)\right)^\frac12<+\infty.
	\]
	On the one hand, for $\lambda\in \calP(\calZ)$ a probability measure, it is immediate that $c_{\eps,\lambda} \ge c\coloneqq \inf\psi+\psi$. This can be seen either because $c_{\eps,\lambda}$ is a soft-min or by bounding the content of the integral. On the other hand, $k_{\eps,\lambda}$ is always a reproducing kernel as a positive sum of rank-one kernels; equivalently, by \eqref{eq:lse-kernel} it is the Gram kernel in $L^2(\calZ,\lambda)$ of the nonnegative feature map $x\mapsto e^{-\psi(x,\cdot)/\eps}$. This construction is reminiscent of Laplace's principle, $-\eps\log\int_\calZ e^{-h/\eps}\de\lambda\to\inf h$ as $\eps\to0$; here, however, we look for an exact identity at fixed $\eps$, and an asymptotic statement would in any case not suffice, since \Cref{ex:counter-example_debiasable} shows that debiasability of $c$ alone does not yield that of $\OT^\eps_c$.

	\begin{lemma}\label{lem:c_eps_lambda_debiasable}
		For all $\eps>0$ and $\lambda\in\calM_+(\calZ)$ with $\int_{\calZ} e^{-\frac{2\psi(x,z)}{\eps}}\de\lambda(z)<+\infty$, the cost function $c_{\eps,\lambda}$ is debiasable. Moreover, if for all $x\neq y\in\bbX$, $\psi(x,\cdot)-\psi(y,\cdot)$ is not a constant function $\lambda$-a.e., then $c_{\eps,\lambda}$  is strictly debiasable.
	\end{lemma}
    The cost $c_{\eps,\lambda}$ is therefore a natural candidate to provide a debiasable entropic optimal transport cost.
	\begin{proof}
		It holds, using again the Cauchy-Schwarz inequality and its conditions for equality,
		\[
		2 c_{\eps,\lambda}(x,y)-c_{\eps,\lambda}(x,x)-c_{\eps,\lambda}(y,y)=\eps\log\left( \frac{\int_\calZ e^{-\frac{2\psi(x,z)}{\eps}}\de\lambda(z)\int_\calZ e^{-\frac{2\psi(y,z)}{\eps}}\de\lambda(z)}{\left(\int_\calZ e^{-\frac{\psi(x,z)+\psi(y,z)}{\eps}}\de\lambda(z)\right)^2}\right)\ge 0
		\]
		with equality if and only if there exists $\alpha,\beta\ge0$, not both zero, such that $\alpha e^{-\frac{2\psi(x,z)}{\eps}}=\beta e^{-\frac{2\psi(y,z)}{\eps}}$ for $\lambda$-a.e. $z\in\calZ$.
	\end{proof}

	\begin{theorem}[Entropic optimal transport for log-sum-exp]\label{thm:EOT_debiased}
		Let $\bbX$ and $\calZ$ be Polish spaces, let $\psi:\bbX\times\calZ\to\R\cup\{+\infty\}$ be jointly
		measurable, let $\lambda \in \calM_+(\calZ)$ and $\eps>0$, and let
		$c_{\eps,\lambda}$ and $k_{\eps,\lambda}$ be defined by \eqref{eq:lse-cost}--\eqref{eq:lse-kernel},
		with $\int_{\calZ} e^{-\frac{2\psi(x,z)}{\eps}}\de\lambda(z)<+\infty$ for every $x\in\bbX$. Then
		$\OT^\eps_{c_{\eps,\lambda}}$ is debiasable on the class $\calP_{k_{\eps,\lambda}}(\bbX)$ of
		\eqref{eq:def_P_k}, 
        and for all $\mu,\nu\in\calP_{k_{\eps,\lambda}}(\bbX)$
		\begin{equation}\label{eq:eot_c_eps_lambda}
			\OT_{c_{\eps,\lambda}}^\eps(\mu,\nu)=\inf_{\eta\in\calP(\calZ)}  \OT^\eps_{\psi}(\mu,\eta)+\OT^\eps_{\psi}(\nu,\eta)+ \eps \KL(\eta|\lambda),
		\end{equation}
 If $\OT^\eps_{c_{\eps,\lambda}}(\mu,\nu)<+\infty$, the infimum in \eqref{eq:eot_c_eps_lambda} is attained at a unique $\eta^*\in\calP(\calZ)$, given by \eqref{eq:def_eta_star} below. Finally, if $(\calZ,\psi)$ is an inf-representation of some cost $c:\bbX\times\bbX\to\R\cup\{+\infty\}$ and $\lambda\in\calP(\calZ)$, then $\OT_{c_{\eps,\lambda}}^\eps \ge \OT^\eps_c$ as $c_{\eps,\lambda}\ge c$.
		\end{theorem}
	\begin{remark}[Normalization of $\lambda$]\label{rem:normalization_lambda}
Since $\lambda$ is $\sigma$-finite, one may choose a Borel $w:\calZ\to(0,+\infty)$ with $\int_\calZ w\,\de\lambda=1$ and replace $(\psi,\lambda)$ by $(\psi+\tfrac{\eps}{2}\log w,\,w\lambda)$: this leaves $e^{-2\psi/\eps}\lambda$, hence $k_{\eps,\lambda}$, $c_{\eps,\lambda}$ and $\Psi$, unchanged, and reduces to the case $\lambda\in\calP(\calZ)$, for which $\KL(\cdot|\lambda)\ge0$. The Lebesgue measure of \Cref{ex:computation_midpoint_Gaussian} and the Gaussian measure of \Cref{prop:Gaussian-factorisation-separable} are related in this way, the additive constant $\frac{d\eps}{2}\log(\pi\eps/4)$ of \eqref{eq:eta_2_Gaussien} being the price of not normalizing.
	\end{remark}
	\begin{proof}\label{proof:EOT_debiased}

    We start with finiteness issues. Let $\mu\in\calP_{k_{\eps,\lambda}}(\bbX)$, $\eta\in\calP(\calZ)$ and $\pi\in\Pi(\mu,\eta)$. The measure $e^{-2\psi/\eps}\mu\otimes\lambda$ is finite, of mass $\int_\bbX k_{\eps,\lambda}\,\de\mu$ by \eqref{eq:def_P_k}. %
    Writing the density of $\pi$ with respect to $e^{-2\psi/\eps}\mu\otimes\lambda$ as $e^{2\psi/\eps}\frac{\de\pi}{\de(\mu\otimes\lambda)}$ and using the chain rule \eqref{eq:chain-KL} along the second marginal, we obtain%
\begin{equation}\label{eq:tilted_reference_identity}
\eps\KL\big(\pi\,\big|\,e^{-2\psi/\eps}\mu\otimes\lambda\big)
=2\int_{\bbX\times\calZ}\psi\,\de\pi+\eps\KL(\pi|\mu\otimes\eta)+\eps\KL(\eta|\lambda),
\end{equation}
whenever the r.h.s.\ is well defined. %
By Jensen, we have
\begin{equation}\label{eq:DV_lower_bound_psi}
\Psi(\mu,\eta)\coloneqq\inf_{\pi\in\Pi(\mu,\eta)}\frac{\eps}{2}\KL\big(\pi\,\big|\,e^{-2\psi/\eps}\mu\otimes\lambda\big)
+\frac{\eps}{2}\KL(\pi|\mu\otimes\eta)
\ \ge\ -\frac{\eps}{2}\log\int_\bbX  k_{\eps,\lambda}\,\de\mu\ >\ -\infty ,
\end{equation}
and $\Psi(\mu,\eta)=\OT^\eps_{\psi}(\mu,\eta)+\frac{\eps}{2}\KL(\eta|\lambda)$ whenever the latter sum makes sense. The quantity $\Psi(\mu,\eta)$ is thus always well defined in $(-\infty,+\infty]$, being an infimum of sums of relative entropies with respect to finite measures. Hence the r.h.s.\ in \eqref{eq:eot_c_eps_lambda} is bounded from below uniformly in $\eta$ and well-defined.

Let $\mu,\nu\in\calP_{k_{\eps,\lambda}}(\bbX)$. Debiasability follows from \eqref{eq:eot_c_eps_lambda} and \Cref{prop:tpsd_feat_map}, applied with the index set $\calP(\calZ)$ and the $\R\cup\{+\infty\}$-valued function $\Psi$ of \eqref{eq:DV_lower_bound_psi}. To prove \eqref{eq:eot_c_eps_lambda}, we show that both of its sides equal to
\begin{equation}\label{eq:OTeps-decomp-psi-gamma}
\inf_{\substack{\eta \in \calP(\calZ) \\ \gamma\in \Pi(\mu,\nu,\eta)}}\eps\KL\Big(\gamma\,\Big|\,e^{-\frac{\psi(x,z)+\psi(y,z)}{\eps}}\mu\otimes\nu\otimes\lambda\Big),
\end{equation}
where $\Pi(\mu,\nu,\eta)$ is the set of probability measures on $\bbX\times\bbX\times\calZ$ with marginals $\mu$, $\nu$ and $\eta$. The reference measure in \eqref{eq:OTeps-decomp-psi-gamma} has mass $\iint k_{\eps,\lambda}\,\de\mu\,\de\nu\le\big(\int k_{\eps,\lambda}\,\de\mu\big)^{1/2}\big(\int k_{\eps,\lambda}\,\de\nu\big)^{1/2}<+\infty$. The relative entropy in \eqref{eq:OTeps-decomp-psi-gamma} thus takes values in $(-\infty,+\infty]$, and it equals $\int[\psi(x,z)+\psi(y,z)]\,\de\gamma+\eps\KL(\gamma|\mu\otimes\nu\otimes\lambda)$ whenever this sum is defined.

\emph{Step 1: \eqref{eq:OTeps-decomp-psi-gamma} equals the r.h.s.\ of \eqref{eq:eot_c_eps_lambda}.} Let $\gamma\in\Pi(\mu,\nu,\eta)$, $\pi_1\coloneqq(\p_1,\p_3)_\#\gamma$ and $\pi_2\coloneqq(\p_2,\p_3)_\#\gamma$. We have
\[
\KL\Big(\gamma\,\Big|\,e^{-\frac{\psi(x,z)+\psi(y,z)}{\eps}}\mu\otimes\nu\otimes\lambda\Big)
=\tfrac12\KL\Big(\gamma\,\Big|\,e^{-\frac{2\psi(x,z)}{\eps}}\mu\otimes\nu\otimes\lambda\Big)
+\tfrac12\KL\Big(\gamma\,\Big|\,e^{-\frac{2\psi(y,z)}{\eps}}\mu\otimes\nu\otimes\lambda\Big).
\]
The chain rule of \Cref{lem:chain-KL-sigma} along the $(x,z)$-marginal $\pi_1$, followed by \eqref{eq:marginalization} for the projection onto $(y,z)$, gives%
\[
\begin{aligned}
\KL\Big(\gamma\,\Big|\,e^{-\frac{2\psi(x,z)}{\eps}}\mu\otimes\nu\otimes\lambda\Big)
&=\KL\big(\pi_1\,\big|\,e^{-2\psi/\eps}\mu\otimes\lambda\big)+\int\left( \int \log\left(\tfrac{\de \gamma_{x,z}}{\de\nu}\right)\de \gamma_{x,z} \right)\,\de\pi_1(x,z)
\\
&=\KL\big(\pi_1\,\big|\,e^{-2\psi/\eps}\mu\otimes\lambda\big)+\iint \log\left(\tfrac{\de \gamma_{x,z}\otimes \pi_1}{\de\nu\otimes\pi_1}\right)\de \gamma_{x,z} \de\pi_1(x,z)
\\
&\ge\ \KL\big(\pi_1\,\big|\,e^{-2\psi/\eps}\mu\otimes\lambda\big)+\KL(\pi_2|\nu\otimes\eta).    
\end{aligned}
\]
Equality holds when $\gamma_{x,z}=(\pi_2)_z$, i.e.\ when $\gamma=\big((\pi_1)_z\otimes(\pi_2)_z\big)\otimes\eta$ is the gluing of $\pi_1$ and $\pi_2$ along $\eta$.
Exchanging the roles of $(\mu,\pi_1)$ and $(\nu,\pi_2)$, we obtain also that $\KL\Big(\gamma\,\Big|\,e^{-\frac{2\psi(x,z)}{\eps}}\mu\otimes\nu\otimes\lambda\Big)
\ge \KL\big(\pi_2\,\big|\,e^{-2\psi/\eps}\nu\otimes\lambda\big)+\KL(\pi_1|\mu\otimes\eta)$.

Summing these two lower bounds and using \eqref{eq:DV_lower_bound_psi}, the objective of \eqref{eq:OTeps-decomp-psi-gamma} at $\gamma$ is larger than $\Psi(\mu,\eta)+\Psi(\nu,\eta)$, so that \eqref{eq:OTeps-decomp-psi-gamma} is larger than the r.h.s.\ of \eqref{eq:eot_c_eps_lambda}. Conversely, for $\eta\in\calP(\calZ)$ and minimizing sequences $(\pi_1^n)_n$, $(\pi_2^n)_n$ for $\Psi(\mu,\eta)$ and $\Psi(\nu,\eta)$, the equality case applied to their gluings $\overline\gamma^n$ shows that
\[
\inf_{\gamma\in\Pi(\mu,\nu,\eta)}\eps\KL\Big(\gamma\,\Big|\,e^{-\frac{\psi(x,z)+\psi(y,z)}{\eps}}\mu\otimes\nu\otimes\lambda\Big)
\ \le\ \eps\KL\Big(\overline\gamma^n\,\Big|\,e^{-\frac{\psi(x,z)+\psi(y,z)}{\eps}}\mu\otimes\nu\otimes\lambda\Big)
\ \xrightarrow[n\to\infty]{}\ \Psi(\mu,\eta)+\Psi(\nu,\eta),
\]
and minimizing over $\eta$ shows that \eqref{eq:OTeps-decomp-psi-gamma} is smaller than the r.h.s.\ of \eqref{eq:eot_c_eps_lambda}.

\emph{Step 2: \eqref{eq:OTeps-decomp-psi-gamma} equals $\OT^\eps_{c_{\eps,\lambda}}(\mu,\nu)$.} The reference measure in \eqref{eq:OTeps-decomp-psi-gamma} has $(x,y)$-marginal $k_{\eps,\lambda}\,\mu\otimes\nu=e^{-c_{\eps,\lambda}/\eps}\mu\otimes\nu$. Its disintegration along this marginal is given, for $k_{\eps,\lambda}\,\mu\otimes\nu$-a.e.\ $(x,y)$, by the Gibbs probability measures
\begin{equation}\label{eq:gibbs_measure}
\de\rho^*_{x,y}(z)\coloneqq\frac{e^{-\frac{\psi(x,z)+\psi(y,z)}{\eps}}}{k_{\eps,\lambda}(x,y)}\,\de\lambda(z).
\end{equation}
For $\gamma\in\Pi(\mu,\nu,\eta)$, with $\pi\coloneqq(\p_1,\p_2)_\#\gamma\in\Pi(\mu,\nu)$ and $\gamma=\gamma_{x,y}\otimes\pi$, the chain rule of \Cref{lem:chain-KL-sigma} gives
\[
\eps\KL\Big(\gamma\,\Big|\,e^{-\frac{\psi(x,z)+\psi(y,z)}{\eps}}\mu\otimes\nu\otimes\lambda\Big)
=\eps\KL\big(\pi\,\big|\,e^{-c_{\eps,\lambda}/\eps}\mu\otimes\nu\big)+\eps\int_{\bbX\times\bbX}\KL\big(\gamma_{x,y}\big|\rho^*_{x,y}\big)\,\de\pi(x,y).
\]
The first term is the objective of \eqref{eq:oteps} for the cost $c_{\eps,\lambda}$. The second is nonnegative and vanishes if and only if $\gamma_{x,y}=\rho^*_{x,y}$ for $\pi$-a.e.\ $(x,y)$. 
Thus, \eqref{eq:OTeps-decomp-psi-gamma} equals $\OT^\eps_{c_{\eps,\lambda}}(\mu,\nu)$.

\emph{Step 3: attainment and uniqueness.} Assume $\OT^\eps_{c_{\eps,\lambda}}(\mu,\nu)<+\infty$, and let $\pi^*$ be the unique optimal plan of \eqref{eq:oteps} for $c_{\eps,\lambda}$ \cite[Theorem 1.10]{Nutz2021IntroductionTE}. By Step 2, $\gamma^*\coloneqq\rho^*_{x,y}\otimes\pi^*$ is the unique minimizer of \eqref{eq:OTeps-decomp-psi-gamma}. By Step 1, the infimum in \eqref{eq:eot_c_eps_lambda} is then attained at
\begin{equation}\label{eq:def_eta_star}
\eta^*\coloneqq(\p_3)_\#\gamma^*,
\qquad
\frac{\de\eta^*}{\de\lambda}(z)=\int_{\bbX\times\bbX}
\frac{e^{-\frac{\psi(x,z)+\psi(y,z)}{\eps}}}{k_{\eps,\lambda}(x,y)}\,\de\pi^*(x,y).
\end{equation}
Conversely, if $\eta$ attains the infimum, the gluings $\overline\gamma^n$ of Step 1 form a minimizing sequence for \eqref{eq:OTeps-decomp-psi-gamma} in the weakly compact set $\Pi(\mu,\nu,\eta)$. Relative entropy with respect to a finite measure being weakly lower semicontinuous, any weak limit point of $(\overline\gamma^n)_n$ equals $\gamma^*$, whence $\eta=\eta^*$.

Finally, if $(\calZ,\psi)$ is an inf-representation of $c$ and $\lambda\in\calP(\calZ)$, then $e^{-c(x,y)/\eps}=\sup_{z}e^{-\frac{\psi(x,z)+\psi(y,z)}{\eps}}\ge k_{\eps,\lambda}(x,y)$ for all $x,y\in\bbX$ by \eqref{eq:sup-exp_representation_k}. That is, $c\le c_{\eps,\lambda}$, and the objective of \eqref{eq:oteps} is nondecreasing in the cost.
	\end{proof}

    	Let us discuss now the strict debiasability of the cost $\OT^\eps_{c_{\eps,\lambda}}$.
	We start by recalling some known facts about the entropic optimal transport problem \cite[Theorem 4.2 and Remarks 4.1 and 4.12]{Nutz2021IntroductionTE}, which will be also useful later on. For a Polish space $\bbX$ and any $\mu,\nu\in\calP(\bbX)$ with $\OT^\eps_c(\mu,\nu)<+\infty$ and $\iint e^{-c/\eps}\de\mu\,\de\nu<+\infty$, there exists a unique optimal plan $\pi\in \Pi(\mu,\nu)$ achieving the optimum in \eqref{eq:oteps} and, if moreover $c<+\infty$ $\mu\otimes\nu$-a.e., there exist measurable functions $f,g:\bbX\to\R$, which belong to $L^1(\mu)$ and $L^1(\nu)$ respectively when $c\in L^1(\mu\otimes\nu)$, such that this plan has the form
	\begin{equation}\label{eq:opti-pi-characterization}
		\pi(x,y)=e^{\frac{-c(x,y)+f(x)+g(y)}{\eps}}\mu\otimes\nu (x,y).
	\end{equation}
	The functions $f,g$, which are called the optimal (entropic) dual potentials, satisfy Sinkhorn's fixed-point equations, a.k.a.\ the Schrödinger system:
	\begin{equation}\label{eq:sinkhorn-fixed-point}
		\begin{aligned}
			f(x)=-\eps\log\left(\int_\bbX e^{\frac{g(y)-c(x,y)}{\eps}}\de\nu(y)\right),\\
			g(y)=-\eps\log\left(\int_\bbX e^{\frac{f(x)-c(x,y)}{\eps}}\de\mu(x)\right).
		\end{aligned}
	\end{equation}
    These facts are usually stated for costs bounded from below, but they only use the finiteness of the reference measure $e^{-c/\eps}\mu\otimes\nu$, which holds on the classes \eqref{eq:def_P_k} considered here; they thus apply to costs such as $c_{\eps,\lambda}$, which are bounded from below only up to functions of one variable. The potentials $f,g$ are defined $\mu$-a.e.\ and $\nu$-a.e.\ respectively, but \eqref{eq:sinkhorn-fixed-point} can be taken as their pointwise definition, and they are unique up to the transformation $(f,g)\mapsto(f+a,g-a)$, $a\in\R$. When $c\in L^1(\mu\otimes\nu)$, one moreover has $\OT^\eps_c(\mu,\nu)=\int f\,\de\mu+\int g\,\de\nu$. Finally, if $\mu=\nu$, then $(g,f)$ is a pair of potentials as well, so that by uniqueness one can take $f=g$. We refer to \cite[Section 4]{Nutz2021IntroductionTE} for more details.

	\begin{corollary}[Strict debiasability of $\OT^\eps_{c_{\eps,\lambda}}$]\label{cor:strict_deb_lse}
		With the assumptions of \Cref{thm:EOT_debiased}, suppose further that $k_{\eps,\lambda}>0$ on $\bbX\times\bbX$, which holds e.g.\ when $\psi$ is real-valued, and that the mean embedding
		\[
		T_\psi:\ \xi\ \longmapsto\ \int_\bbX e^{-\psi(x,\cdot)/\eps}\,\de\xi(x)
		\]
		is injective, modulo $\lambda$-negligible sets, on
		$\calM^k_+(\bbX)\coloneqq\big\{\xi\in\calM_+(\bbX):\int_{\bbX\times\bbX}k_{\eps,\lambda}\,\de\xi\,\de\xi<+\infty\big\}$.
		Then $\OT^\eps_{c_{\eps,\lambda}}$ is strictly debiasable on
		$\big\{\mu\in\calP_{k_{\eps,\lambda}}(\bbX):\OT^\eps_{c_{\eps,\lambda}}(\mu,\mu)<+\infty\big\}$.
	\end{corollary}
	By \eqref{eq:T_psi_gram} below, this injectivity is exactly the requirement
	\begin{equation}\label{eq:ispd_k_eps_lambda}
		\int_{\bbX\times\bbX}k_{\eps,\lambda}\,\de(\xi_1-\xi_2)\,\de(\xi_1-\xi_2)>0
		\qquad\text{for all }\xi_1\neq\xi_2\in\calM^k_+(\bbX),
	\end{equation}
	i.e.\ the integral strict positive definiteness of $k_{\eps,\lambda}$ in the sense of
	\cite[Definition 2.1]{Ziegel2024} and \cite[p2392]{sriperumbudur11a}, but stated on differences of
	$\sigma$-finite nonnegative measures of finite $k_{\eps,\lambda}$-energy rather than on finite signed
	measures. 

    	Condition \eqref{eq:ispd_k_eps_lambda} bears on $k_{\eps,\lambda}$ and not on $e^{-\psi/\eps}$: by
	\eqref{eq:lse-kernel}, $k_{\eps,\lambda}$ is the Gram kernel in $L^2(\calZ,\lambda)$ of the
	nonnegative feature map $x\mapsto e^{-\psi(x,\cdot)/\eps}$, and \eqref{eq:T_psi_gram} shows that
	\eqref{eq:ispd_k_eps_lambda} is exactly the injectivity of the associated mean embedding $T_\psi$. It holds for instance for $\bbX=\R^d$,
	$c(x,y)=|x-y|^2$, $\psi(x,z)=2|x-z|^2$ and $\lambda$ the Lebesgue measure, since
	\eqref{eq:condition-const_2} gives $k_{\eps,\lambda}(x,y)=(\pi\eps/4)^{d/2}e^{-|x-y|^2/\eps}$, a
	Gaussian kernel which is integrally strictly positive
	definite \cite[Theorem 3.1]{Ziegel2024} over $\calM^b$.%
    
	\begin{proof}
Expanding the square, for any $\xi_1,\xi_2\in\calM^k_+(\bbX)$,
\begin{equation}\label{eq:T_psi_gram}
\int_\calZ\big(T_\psi\xi_1-T_\psi\xi_2\big)^2\de\lambda
=\int_{\bbX\times\bbX}k_{\eps,\lambda}(x,y)\,\de\rho(x)\de\rho(y),\qquad \rho\coloneqq\xi_1-\xi_2,
\end{equation}
the four pairings $\iint k_{\eps,\lambda}\de\xi_i\de\xi_j$ being finite by the Cauchy--Schwarz inequality, which proves the equivalence between the injectivity of $T_\psi$ and \eqref{eq:ispd_k_eps_lambda}.

Let now $\mu\neq\nu$. If $\OT^\eps_{c_{\eps,\lambda}}(\mu,\nu)=+\infty$ the strict inequality is trivial, both self-transportation costs being finite; otherwise \Cref{thm:EOT_debiased} provides the unique minimizers $\eta_\mu$, $\eta_\nu$ and $\eta_{\mu\nu}$ of \eqref{eq:eot_c_eps_lambda} for the three pairs, so that, arguing as in \Cref{corr:strict_tpsd}, it suffices to show that $\eta_\mu=\eta_\nu$ implies $\mu=\nu$. By \eqref{eq:opti-pi-characterization} and the symmetry of the self-transport problem, the optimal plan of $\OT^\eps_{c_{\eps,\lambda}}(\mu,\mu)$ reads $\de\pi_\mu(x,y)=k_{\eps,\lambda}(x,y)\,\de\mu^\sharp\otimes\mu^\sharp(x,y)$ with $\mu^\sharp\coloneqq e^{f_\mu/\eps}\mu\in\calM_+(\bbX)$, for some optimal entropic potential $f_\mu$, which exists since $k_{\eps,\lambda}>0$. Then \eqref{eq:def_eta_star} gives
\begin{gather*}
    \frac{\de \eta_\mu}{\de\lambda}(z)=\int_{\bbX\times\bbX} e^{-\psi(x,z)/\eps}e^{-\psi(y,z)/\eps}\, \de\mu^\sharp\otimes\mu^\sharp(x,y)=\big(T_\psi\mu^\sharp\big)^2(z),\\
\|T_\psi\mu^\sharp\|^2_{L^2(\calZ,\lambda)}=\iint_{\bbX\times\bbX}k_{\eps,\lambda}\,\de\mu^\sharp\de\mu^\sharp=\pi_\mu(\bbX\times\bbX)=1,
\end{gather*}
so that $\mu^\sharp\in\calM^k_+(\bbX)$, and similarly for $\nu$. If $\eta_\mu=\eta_\nu$, then $T_\psi\mu^\sharp=T_\psi\nu^\sharp$ $\lambda$-a.e., both being nonnegative, hence $\mu^\sharp=\nu^\sharp$ by the injectivity assumption. Sinkhorn's equations \eqref{eq:sinkhorn-fixed-point} for the two self-transports then read
\[
e^{-\frac{f_\mu(x)}{\eps}}=\int_\bbX k_{\eps,\lambda}(x,y)\,\de\mu^\sharp(y)=\int_\bbX k_{\eps,\lambda}(x,y)\,\de\nu^\sharp(y)=e^{-\frac{f_\nu(x)}{\eps}},
\]
whence $f_\mu=f_\nu$ and $\mu=e^{-f_\mu/\eps}\mu^\sharp=e^{-f_\nu/\eps}\nu^\sharp=\nu$.
	\end{proof}

	Note that given $\psi$, it is in general much simpler to compute $c_{\eps,\lambda}$ than $c$. However, most often we are given $c$, so what about $\OT^\eps_{c}$? First, in some cases $c$ and $c_{\eps,\lambda}$ actually do not differ much.

	\begin{corollary}\label{cor:EOT_debias_Gaussian_like} Let $k(x,y)\coloneqq e^{-\frac{c(x,y)}{\eps}}$. With the notations of \Cref{thm:EOT_debiased} assume further that there exists a nonnegative $\sigma$-finite measure $\lambda \in \calM_+(\calZ)$ and a Borel function $f:\bbX\to (0,+\infty)$, which may depend on $\lambda$ and $\eps$, such that for all $x,y\in \bbX$
		\begin{equation}\label{eq:condition-const}
			k(x,y)=\frac{1}{f(x)f(y)}\int_\calZ e^{-\frac{\psi(x,z)+\psi(y,z)}{\eps}}\de\lambda(z)\stackrel{\eqref{eq:lse-kernel}}{=}\frac{1}{f(x)f(y)}k_{\eps,\lambda}(x,y).
		\end{equation}
		Taking $-\eps\log$ of \eqref{eq:condition-const}, this condition can be equivalently written as $c_{\eps,\lambda}(x,y)=c(x,y)-\eps \log f(x)-\eps\log f(y)$ for all $x,y\in \bbX$. Then, for all $\mu,\nu\in\calP_{k_{\eps,\lambda}}(\bbX)$ with $\int_\bbX|\log f|\,\de\mu+\int_\bbX|\log f|\,\de\nu<+\infty$,
		\begin{multline}\label{eq:eot_c_eps_lambda_c}
			\OT^\eps_c(\mu,\nu)=\OT^\eps_{c_{\eps,\lambda}}(\mu,\nu)+\eps\int_\bbX \log f(x)\de\mu(x)+\eps\int_\bbX \log f(y)\de\nu(y) \\
			\stackrel{\eqref{eq:eot_c_eps_lambda}}{=}\eps\int_\bbX \log f(x)\de\mu(x)+\eps\int_\bbX \log f(y)\de\nu(y)+\inf_{\eta\in\calP(\calZ)}  \OT^\eps_{\psi}(\mu,\eta)+\OT^\eps_{\psi}(\nu,\eta)+ \eps \KL(\eta|\lambda),
		\end{multline}
		so $\OT^\eps_{c}$ is debiasable.
	\end{corollary}
	\begin{proof}
		The debiasability stems from \eqref{eq:eot_c_eps_lambda} and
		\Cref{lem:operations_preserving}-\ref{fact:g-c}.
	\end{proof}
	By \eqref{eq:eot_c_eps_lambda_c}, strict debiasability holds under the same assumptions as \Cref{cor:strict_deb_lse}.  Condition \eqref{eq:condition-const} can be interpreted either as: i) having a simple relation between $\inf$ (given by $c$) and soft-min (given by $c_{\eps,\lambda}$); ii) a form of Chapman-Kolmogorov equation, as evidenced in \eqref{eq:condition-const_2} and further discussed in \eqref{eq:CK}, interpreting $e^{-\psi(x,z)/\eps}$ as a Markov kernel so that \eqref{eq:condition-const} describes a factorization of the covariance $k$, more general factorizations being given in \eqref{eq:completely_pos_kernel}. We now move to an example of explicit computation of the midpoint for the squared Euclidean cost and two Dirac masses.

	\begin{example}[Squared Euclidean distance]\label{ex:computation_midpoint_Gaussian}
		Consider $\bbX=\R^d$ equipped with the squared Euclidean distance, $c(x,y)=|x-y|^2$. This debiased cost admits the decomposition $|x-y|^2=\inf_{z\in\R^d} 2|z-x|^2+2|z-y|^2$. Then, if $\lambda$ is chosen to be the Lebesgue measure on $\R^d$, using that $\left|z- \frac{x+y}{2}\right|^2=\frac{|z-x|^2}{2}+\frac{|z-y|^2}{2}-\frac{|x-y|^2}{4}$, we obtain \eqref{eq:condition-const} in the form
		\begin{equation}\label{eq:condition-const_2}
			\int_{\R^d} e^{-\frac{2|z-x|^2+2|z-y|^2}{\eps}}\de\lambda(z)=e^{-\frac{|x-y|^2}{\eps}}\int_{\R^d} e^{-\frac{4\left|z-\frac{x+y}{2}\right|^2}{\eps}}\de\lambda(z)=e^{-\frac{|x-y|^2}{\eps}}\left(\frac{\pi\eps}{4}\right)^{\frac{d}{2}},
		\end{equation}
		so that here $f$ is the constant $(\pi\eps/4)^{d/4}$, $k_{\eps,\lambda}(x,x)=(\pi\eps/4)^{d/2}$
		is constant, $\calP_{k_{\eps,\lambda}}(\R^d)=\calP(\R^d)$,
		and therefore $\OT^\eps_{c_{\eps,\lambda}}(\mu,\nu)= \OT^\eps_c(\mu,\nu)-\frac{d\eps}{2}\log(\frac{\pi\eps}{4})$. Hence \eqref{eq:eot_c_eps_lambda_c} gives:
		\begin{align} \nonumber
			\OT_c^\eps(\mu,\nu)
			&=\inf_{\eta\in\calP({\R^d})}  \OT^\eps_{2c}(\mu,\eta)+\OT^\eps_{2c}(\nu,\eta)+ \eps \KL(\eta|\lambda)+\frac{d\eps}{2}\log(\frac{\pi\eps}{4})\\
			&=\inf_{\eta\in\calP({\R^d})}  2\OT^{\eps/2}_{c}(\mu,\eta)+2\OT^{\eps/2}_{c}(\nu,\eta)+ \eps \KL(\eta|\lambda)+\frac{d\eps}{2}\log(\frac{\pi\eps}{4}).\label{eq:eta_2_Gaussien}
		\end{align}
		We generalize this formula beyond midpoints in \eqref{eq:eta_t_Gaussien}. Note that for $\mu=\delta_x$ and $\nu=\delta_y$ the Gibbs formula \eqref{eq:gibbs_measure} directly gives the minimizer $\eta^*=\rho^*_{x,y}=\calN\big(\frac{x+y}{2},\frac{\eps}{8}I_d\big)$, whose normalizing constant is the Gaussian integral of \eqref{eq:condition-const_2}, as recovered in the explicit computation below.

		As an example, take $\mu=\delta_{x}$, $\nu=\delta_{y}$ 
        and let us test $\eta\in\calP(\R^d)$ of the form $\eta=\frac{1}{(2\pi)^{d/2} \sigma^d}e^{-\frac{|z-m|^2}{2\sigma^2}}$, for $m\in\R^d,\sigma>0$. In this case all computations are explicit and we try to optimize with respect to $m,\sigma$ to obtain \eqref{eq:eta_2_Gaussien}. It holds
		\[
		\begin{aligned}
			\OT^\eps_{c}(\mu,\nu)&=|x-y|^2,\\
			\OT^\eps_{\psi}(\mu,\eta)&=\frac{2}{(2\pi)^{d/2}\sigma^d}\int_{\R^d} |x-z|^2 e^{-\frac{|z-m|^2}{2\sigma^2}} \de\lambda(z)=2|x-m|^2+2d\sigma^2 \quad \text{(and similarly for $\OT^\eps_{\psi}(\nu,\eta)$)},\\
			\KL(\eta|\lambda)&
			\begin{aligned}[t]
				&=\frac{1}{(2\pi)^{d/2}\sigma^d}\int_{\R^d} \log\left(\frac{1}{(2\pi)^{d/2}\sigma^d} e^{-\frac{|z-m|^2}{2\sigma^2}}\right) e^{-\frac{|z-m|^2}{2\sigma^2}} \de\lambda(z)\\
				&=\frac{1}{(2\pi)^{d/2}\sigma^d}\int_{\R^d} \left[ \log(\frac{1}{(2\pi)^{d/2}\sigma^d}) -\frac{|z-m|^2}{2\sigma^2} \right] e^{-\frac{|z-m|^2}{2\sigma^2}}\de\lambda(z)=\log(\frac{1}{(2\pi)^{d/2}\sigma^d})-\frac{d}{2},
			\end{aligned}
		\end{aligned}
		\]
		which leads us to solving
		\[
		\min_{m,\sigma} 2|x-m|^2+2|y-m|^2+4d\sigma^2+\eps \log(\frac{1}{(2\pi)^{d/2}\sigma^d})-\frac{d\eps}{2}.
		\]
		The optimal solution is given by $\overline m=\frac{x+y}{2}$, $\overline\sigma=\sqrt{\frac{\eps}{8}}$, and we obtain
		\[
		2|x-\overline m|^2+2|y-\overline m|^2+4d\overline\sigma^2+\eps \log(\frac{1}{(2\pi)^{d/2}\overline\sigma^d})-\frac{d\eps}{2}=|x-y|^2-\frac{d\eps}{2}\log\left(\frac{\pi\eps}{4}\right),
		\]
		so that we get the desired equality. This implies that the optimal $\eta$ is indeed a Gaussian measure which is centered in the mid-point between $x$ and $y$. Similar explicit results can be obtained in the case where $\mu$ and $\nu$ are Gaussian measures, based on the formulas provided in \cite[Theorems 1 and 3]{janati2020}.
	\end{example}

	\begin{remark}[Relation with Entropic Barycenter] Note that we can write \eqref{eq:eot_c_eps_lambda} and hence \eqref{eq:eot_c_eps_lambda_c} as
		\begin{equation}\label{eq:eot_c_eps_lambda_doubly_reg}
			\OT_{c_{\eps,\lambda}}^\eps(\mu,\nu)=2\inf_{\eta\in\calP(\calZ)} \frac{1}{2} \OT^\eps_{\psi}(\mu,\eta)+\frac{1}{2}\OT^\eps_{\psi}(\nu,\eta)+ \frac{\eps}{2}\,\KL(\eta|\lambda).
		\end{equation}
		Eq.\eqref{eq:eot_c_eps_lambda_doubly_reg} is a specific case of a doubly regularized entropic Wasserstein barycenter \cite[Definition 1]{Chizat2025}. Actually this $(\eps,\frac{\eps}{2})$-scaling between the inner and outer regularization has been proven to have particularly favorable approximation properties (ibid., Theorem 3.2).
	\end{remark}
	The issue with condition \eqref{eq:condition-const} is that it essentially requires some sort of translation-invariant property for the measure $\lambda$. In $\R^d$, one can rely on the Lebesgue measure, which is in fact the only nontrivial translation-invariant locally finite Borel measure (up to positive scaling). The same reasoning is valid if $c$ is a more general negative definite cost on a finite dimensional space, that is, it is essentially Euclidean up to composition with a certain mapping $\phi$. In that case, one has a natural notion of Lebesgue measure and can construct the measure $\lambda$ by pulling this back via $\phi$.
	However, no such measure exists in infinite dimensional spaces. In probabilistic terms, this obstruction is also related to the Lévy-Khintchine representation of Lévy processes. All infinitely divisible random processes can be modified into a Lévy process. However formula \mbox{\eqref{eq:neg_def}} for infinitely divisible kernels (beyond the translation-invariant case) does not separate between a finite-dimensional part (drift and diffusion) and a jump part, it all regroups them into $\phi$ and $H_c$, hence we need a stronger assumption \mbox{\eqref{eq:condition-const}} to focus only on finite dimensions.

	\subsection{Negative definite kernels}\label{sec:EOT_neg_kernel}

We now show that for $c$ negative definite, $k=e^{-c/\eps}$ is completely positive for all $\eps>0$ and deduce from it an inf-representation.

	\begin{theorem}
		\label{prop:Gaussian-factorisation-separable}
		Let $\bbX$ be a Polish space and let $c:\bbX\times\bbX\to\R$ be symmetric, Borel and negative definite, with $\phi(\bbX)$ separable in $H_c$ for $(H_c,\phi)$ as in \eqref{eq:neg_def}. Fix $\eps>0$. Then there exists a Polish space $\calZ$, a probability measure $\lambda \in \calP(\calZ)$ and a jointly measurable function $\rho_\eps:\bbX\times \calZ\to(0,\infty)$ such that for every $x,y\in \bbX$
		\begin{equation}\label{eq:completely_pos_kernel}
			\exp\!\Big(-\frac{c(x,y)}{\eps}\Big)=\int_{\calZ} \rho_\eps(x,z)\,\rho_\eps(y,z)\,\de\lambda(z).
		\end{equation}
		For $\tilde \psi(x,z)\coloneqq -\eps \log \rho_\eps(x,z)$ and
		$\tilde c_{\eps,\lambda}(x,y) \coloneqq -\eps \log \int_{\calZ} e^{-\frac{\tilde \psi(x,z)+\tilde \psi(y,z)}{\eps}}\de\lambda(z)$,
		we then have that $c=\tilde c_{\eps,\lambda}$ and $k_{\eps,\lambda}=e^{-c/\eps}$; consequently, for
		every $\mu,\nu\in \calP_{e^{-c/\eps}}(\bbX)$, 
        we have
		\begin{equation}\label{eq:eot_c_eps_lambda_neg_def}
			\OT^\eps_c(\mu,\nu)=\inf_{\eta\in\calP(\calZ)} \OT^\eps_{\tilde \psi}(\mu,\eta)+\OT^\eps_{\tilde \psi}(\nu,\eta)+\eps \KL(\eta|\lambda).
		\end{equation}
		An explicit choice of $\rho_\eps$ and $\tilde \psi$ is given in \eqref{eq:rho-explicit} and \eqref{eq:tildepsi-explicit} based on \eqref{eq:c_neg_def_kernel}.
	\end{theorem}
	\Cref{ex:computation_midpoint_Gaussian} falls within the setting of the theorem, the main difference being that $\lambda$ is now taken to be a probability measure rather than the Lebesgue measure, the two being related as in \Cref{rem:normalization_lambda}. Separability of $\phi(\bbX)\subset H_c$ is intended in the sense that the semi-metric $d_c(x,y)\coloneqq\big(c(x,y)-\tfrac12c(x,x)-\tfrac12c(y,y)\big)^{1/2}=|\phi(x)-\phi(y)|_{H_c}$ makes $(\bbX,d_c)$ separable. This holds in particular whenever $\bbX$ is Polish and $c$ is continuous, since $\phi$ is then continuous by \eqref{eq:neg_def}.%
	\begin{proof}
		Applying \Cref{prop:def_rep_kernels}, we take a Hilbert space $H_c$, a map $\phi:\bbX\to H_c$ and a scalar function $s:\bbX\to\R$ such that for all $x,y\in \bbX$
		\begin{equation}\label{eq:c_neg_def_kernel}
			c(x,y)=|\phi(x)-\phi(y)|_{H_c}^2+s(x)+s(y),
		\end{equation}
		where necessarily $s(x)=\frac{c(x,x)}{2}$, a Borel function. Fixing $x_0\in\bbX$ and replacing
		$\phi$ by $\phi-\phi(x_0)$, which does not affect \eqref{eq:c_neg_def_kernel}, we may assume
		$\phi(x_0)=0$, so that
		$(\phi(x),\phi(y))_{H_c}=\tfrac12\big[c(x,x_0)+c(y,x_0)-c(x,y)-c(x_0,x_0)\big]$ is Borel in each
		variable. By assumption the set $\phi(\bbX)$ is separable, so that $H_c$ may be taken to be the closed
		linear span of $\phi(\bbX)$, which is separable, and $\phi$ is weakly Borel with separable range,
		hence strongly Borel measurable by Pettis' theorem.

		Fix a countable orthonormal basis $\{e_i\}_{i\in I}$ of $H_c$. For each $x\in \bbX$ write the coordinates of $\phi(x)$ in this basis:
		\[
		\phi(x)=\sum_{i\in I}\phi_i(x)\,e_i,\qquad (\phi_i(x))_{i\in I}\in\ell^2(I),
		\]
		the functions $\phi_i=(\phi(\cdot),e_i)_{H_c}$ being Borel.

		Let $\calZ=\R^I$ be the product space of real coordinate sequences indexed by $I$.  On the coordinate cylinder $\sigma$-algebra of $\R^I$ consider the product Gaussian probability measure $\lambda=\bigotimes_{i\in I} \calN(0,1)$. We write $z\in \calZ$ as $z=(z_i)_{i\in I}$ with $z_i\sim \calN(0,1)$ i.i.d.\ under $\lambda$. 
For each $x\in\bbX$, the Gaussian series $\langle\phi(x),z\rangle_{H_c}\coloneqq\sum_{i\in I}\phi_i(x)z_i$ converges $\lambda$-a.s.\ and in $L^2(\calZ,\lambda)$, since $(\phi_i(x))_{i\in I}\in\ell^2(I)$; setting it to $0$ on the (measurable) set where it diverges defines a jointly measurable function, and for fixed $x,y\in\bbX$ the pair $\big(\langle\phi(x),z\rangle_{H_c},\langle\phi(y),z\rangle_{H_c}\big)$ is centered Gaussian under $\lambda$ with covariance $(\phi(x),\phi(y))_{H_c}$. Define
		$\rho_\eps:\bbX\times \calZ\to(0,\infty)$ by the explicit formula
		\begin{equation}\label{eq:rho-explicit}
			\rho_\eps(x,z)\coloneqq \exp\left( \sqrt{\frac{2}{\eps}} \langle \phi(x), z \rangle_{H_c} - \frac{2}{\eps} |\phi(x)|_{H_c}^2 - \frac{s(x)}{\eps} \right)
			=\exp\!\left(\sqrt{\frac{2}{\eps}}\sum_{i\in I} \phi_i(x)\,z_i \;-\; \frac{2}{\eps}\sum_{i\in I} \phi_i(x)^2 \;-\; \frac{s(x)}{\eps}\right),
		\end{equation}
		for which
		\begin{equation}\label{eq:tildepsi-explicit}
			\tilde \psi(x,z) \coloneqq -\eps \log \rho_\eps(x,z)= - \sqrt{2\eps} \langle \phi(x), z \rangle_{H_c} + 2|\phi(x)|_{H_c}^2 + s(x).
		\end{equation}
		The exponential in \eqref{eq:rho-explicit} is a.s.\ finite and strictly positive, $\rho_\eps$ is jointly measurable, and in particular $\rho_\eps(x,\cdot)\in L^2(\calZ,\lambda)$ for every $x\in \bbX$.

Fix $x,y\in \bbX$. Under $\lambda$, the variable $\langle\phi(x)+\phi(y),z\rangle_{H_c}$ is centered Gaussian with variance $|\phi(x)+\phi(y)|^2_{H_c}$, so that $\int_\calZ e^{\sqrt{2/\eps}\langle\phi(x)+\phi(y),z\rangle_{H_c}}\de\lambda(z)=e^{|\phi(x)+\phi(y)|^2_{H_c}/\eps}$ and, by \eqref{eq:rho-explicit},
\begin{align*}
    \int_{\calZ} \rho_\eps(x,z)\rho_\eps(y,z)\,\de\lambda(z)
&=\exp\!\Big(-\frac{s(x)+s(y)}{\eps}-\frac{2|\phi(x)|^2_{H_c}+2|\phi(y)|^2_{H_c}-|\phi(x)+\phi(y)|^2_{H_c}}{\eps}\Big)\\
&=\exp\!\Big(-\frac{s(x)+s(y)+|\phi(x)-\phi(y)|^2_{H_c}}{\eps}\Big)
= e^{-\frac{c(x,y)}{\eps}},
\end{align*}
which is the claimed identity; \eqref{eq:eot_c_eps_lambda_neg_def} then follows from
		\Cref{thm:EOT_debiased} applied to $(\tilde\psi,\lambda)$, for which
		$k_{\eps,\lambda}=e^{-c/\eps}$ is precisely the
		membership of $\mu,\nu$ in $\calP_{k_{\eps,\lambda}}(\bbX)$.
	\end{proof}

	\begin{corollary}[Strict debiasability for negative definite costs]\label{thm:strict_neg_def}
Under the assumptions of \Cref{prop:Gaussian-factorisation-separable}, let moreover $\phi:\bbX\to H_c$ be injective and set $k=e^{-c/\eps}$. Let $\mu,\nu\in\calP_{k}(\bbX)$ and $\OT^\eps_c(\mu,\mu),\OT^\eps_c(\nu,\nu)<+\infty$, and let $f_\mu,f_\nu$ be the optimal potentials of these two self-transport problems, as in \eqref{eq:sinkhorn-fixed-point}. If the tilted measures $\mu^\sharp\coloneqq e^{f_\mu/\eps}\mu$ and $\nu^\sharp\coloneqq e^{f_\nu/\eps}\nu$ belong to $\calP_{k}(\bbX)$, 
then $\OT^\eps_c(\mu,\nu)>\tfrac12\OT^\eps_c(\mu,\mu)+\tfrac12\OT^\eps_c(\nu,\nu)$ whenever $\mu\neq\nu$. In particular $\OT^\eps_c$ is strictly debiasable on any set of probability measures for which these conditions hold.
	\end{corollary}
	The self-transportation costs are finite as soon as $\int_\bbX|s|\,\de\mu+\int_\bbX|\phi|^2_{H_c}\,\de\mu<+\infty$ and likewise for $\nu$, the plan $\mu\otimes\mu$ being then admissible in \eqref{eq:oteps} with zero entropy and finite cost by \eqref{eq:c_neg_def_kernel}. Since $k(x,x)^{1/2}=e^{-s(x)/\eps}$ by \eqref{eq:c_neg_def_kernel}, $\mu^\sharp \in \calP_{k}(\bbX)$
	says precisely that the push-forward measures $\phi_\#\big(e^{-s/\eps}\mu^\sharp\big)$ is a finite measure on $H_c$. This holds when $c\ge0$ on the diagonal and $\mu^\sharp,\nu^\sharp$ are finite, in particular whenever
	$c(x,x)=0$ and the potentials are bounded from above.
	One can find an injective $\phi$ for instance for $c(x,y)=|x-y|_H^p$ for $p\in(0,2]$ on any Hilbert
	space $H$, since \cite[Exercise 2.13, p79]{Berg1984} gives that these $c$ are negative definite and
	we have $s(x)=\frac{c(x,x)}{2}=0$, so $|\phi(x)-\phi(y)|_{H_c}^2=|x-y|_H^p$. 
    Injectivity of $\phi$ for other
	negative definite kernels can be handled through similar arguments.
    	\begin{proof}
 
Let $\mu\neq\nu$ be as in the statement. If $\OT^\eps_c(\mu,\nu)=+\infty$ there is nothing to prove, the self-transportation costs being finite, so assume it is finite. By \Cref{prop:Gaussian-factorisation-separable}, $c=\tilde c_{\eps,\lambda}$ for $\tilde\psi=-\eps\log\rho_\eps$ and the product Gaussian measure $\lambda$, with $k_{\eps,\lambda}=k$, and $\mu,\nu$ belong to the class \eqref{eq:def_P_k}. \Cref{thm:EOT_debiased} thus provides unique minimizers $\eta_{\mu\nu},\eta_\mu,\eta_\nu$ of \eqref{eq:eot_c_eps_lambda} for the pairs $(\mu,\nu)$, $(\mu,\mu)$ and $(\nu,\nu)$, and, all the terms below being finite,
\[
\OT^\eps_c(\mu,\nu)=\Psi(\mu,\eta_{\mu\nu})+\Psi(\nu,\eta_{\mu\nu})
\ \ge\ \Psi(\mu,\eta_\mu)+\Psi(\nu,\eta_\nu)=\tfrac12\OT^\eps_c(\mu,\mu)+\tfrac12\OT^\eps_c(\nu,\nu).
\]
Due to uniqueness of minimizers, equality would force $\eta_\mu=\eta_{\mu\nu}=\eta_\nu$; we show that $\eta_\mu=\eta_\nu$ implies $\mu=\nu$.

Like in the proof of \Cref{cor:strict_deb_lse}, taking symmetric potentials as recalled after \eqref{eq:sinkhorn-fixed-point}, the optimal plan of $\OT^\eps_c(\mu,\mu)$ is $\de\pi_\mu=k\,\de\mu^\sharp\otimes\mu^\sharp$. Writing $T\alpha(z)\coloneqq\int_\bbX\rho_\eps(x,z)\,\de\alpha(x)$ for $\alpha\in\{\mu^\sharp,\nu^\sharp\}$, Tonelli's theorem, \eqref{eq:def_eta_star} and $\int_\calZ\rho_\eps(x,z)\rho_\eps(y,z)\,\de\lambda(z)=k(x,y)$ give
\[
\frac{\de\eta_\mu}{\de\lambda}=\big(T\mu^\sharp\big)^2,\qquad
\int_\calZ\big(T\mu^\sharp\big)^2\de\lambda=\iint k\,\de\mu^\sharp\de\mu^\sharp=\pi_\mu(\bbX\times\bbX)=1,
\]
and the same holds for $\nu$.%
If $\eta_\mu=\eta_\nu$, then $T\mu^\sharp=T\nu^\sharp$ $\lambda$-a.e., both functions being nonnegative. The four pairings $\iint k\,\de\alpha_i\,\de\alpha_j=\int T\alpha_i\,T\alpha_j\,\de\lambda$ are bounded by $1$ by the Cauchy--Schwarz inequality in $L^2(\calZ,\lambda)$. By \eqref{eq:c_neg_def_kernel}, $k(x,y)=e^{-\frac{s(x)+s(y)}{\eps}}e^{-\frac{|\phi(x)-\phi(y)|^2_{H_c}}{\eps}}$, and $e^{-s/\eps}\mu^\sharp$ and $e^{-s/\eps}\nu^\sharp$ are finite measures, since $e^{-s/\eps}=k(\cdot,\cdot)^{1/2}$. The change of variables $h=\phi(x)$ therefore gives, with the finite signed measure $\hat\rho\coloneqq\phi_\#\big(e^{-s/\eps}\mu^\sharp\big)-\phi_\#\big(e^{-s/\eps}\nu^\sharp\big)$ on $H_c$,
\[
0=\int_\calZ\big(T\mu^\sharp-T\nu^\sharp\big)^2\de\lambda
=\iint k\,\de\mu^\sharp\de\mu^\sharp-2\iint k\,\de\mu^\sharp\de\nu^\sharp+\iint k\,\de\nu^\sharp\de\nu^\sharp
=\iint_{H_c\times H_c}e^{-\frac{|h-h'|^2_{H_c}}{\eps}}\,\de\hat\rho(h)\,\de\hat\rho(h').
\]
The Gaussian kernel being integrally strictly positive definite on the separable Hilbert space $H_c$ \cite[Definition 2.1, Theorem 3.1]{Ziegel2024}, we get $\hat\rho=0$. Now $\phi$ is Borel (proof of \Cref{prop:Gaussian-factorisation-separable}) and injective from the Polish space $\bbX$ into $H_c$, so that $\phi(A)$ is Borel and $\phi^{-1}(\phi(A))=A$ for every Borel $A\subseteq\bbX$, by the Lusin--Souslin theorem \cite[Theorem 15.1]{kechris1995classical}. Hence
\[
\int_Ae^{-s/\eps}\,\de\mu^\sharp=\phi_\#\big(e^{-s/\eps}\mu^\sharp\big)(\phi(A))=\phi_\#\big(e^{-s/\eps}\nu^\sharp\big)(\phi(A))=\int_Ae^{-s/\eps}\,\de\nu^\sharp,
\]
that is, $e^{-s/\eps}\mu^\sharp=e^{-s/\eps}\nu^\sharp$, and $\mu^\sharp=\nu^\sharp$ since $e^{-s/\eps}>0$. Finally, Sinkhorn's equations \eqref{eq:sinkhorn-fixed-point} for the two self-transports read $e^{-f_\mu/\eps}=\int k(\cdot,y)\,\de\mu^\sharp(y)=\int k(\cdot,y)\,\de\nu^\sharp(y)=e^{-f_\nu/\eps}$, whence $f_\mu=f_\nu$ and $\mu=e^{-f_\mu/\eps}\mu^\sharp=e^{-f_\nu/\eps}\nu^\sharp=\nu$.
	\end{proof}

    \subsection{Reproducing kernels and a self-antipolar inf-representation}\label{sec:selfantipolar}

	Throughout this section $\eps>0$ and $c:\bbX\times\bbX\to\R$ is symmetric, Borel and bounded from
	below. Equivalently the kernel
	\[
	k(x,y)\coloneqq e^{-c(x,y)/\eps}\quad\text{satisfies}\quad 0<k\le M\coloneqq\sup_{x,y}k(x,y)<\infty ,
	\]
	so that $\calP_k(\bbX)=\calP(\bbX)$ in the notation \eqref{eq:def_P_k}. For $\alpha,\beta\in\calM^b_+(\bbX)$,
	\[
	k_\alpha(x)\coloneqq\int_\bbX k(x,y)\,\de\alpha(y),\quad 
	\langle k_\alpha,k_\beta \rangle\coloneqq \iint_{\bbX\times\bbX}k\,\de\alpha\,\de\beta, \quad \calK\coloneqq\{k_\alpha:\alpha\in\calM^b_+(\bbX)\setminus\{0\}\}.
	\]
	Thus $\calK$ is a convex cone of bounded positive Borel functions, on which $\langle\cdot,\cdot\rangle$ is well defined, as $\langle k_\alpha,k_\beta \rangle=\int k_\alpha\,\de\beta=\int k_\beta\,\de\alpha$ is symmetric, bilinear and positive, as $k>0$. When $k$ is positive semidefinite, $\langle\cdot,\cdot\rangle$ is the inner product of the RKHS $\calH_k$ and $k_\alpha$ is the kernel mean embedding of $\alpha$. We abbreviate
	\begin{gather*}
	m_{\rho\sigma}\coloneqq e^{-\OT^\eps_c(\rho,\sigma)/\eps}\in[0,\infty),\qquad
	\Xi_\rho(\alpha)\coloneqq e^{-\KL(\rho|\alpha)},\qquad
    \calA_\rho\coloneqq\{\alpha\in\calM^b_+(\bbX)\setminus\{0\}:\KL(\rho|\alpha)<\infty\}\ni \rho.
	\end{gather*}
    By Jensen's inequality, $-\KL(\rho|\alpha)=\int\log\frac{\de\alpha}{\de\rho}\de\rho\le\log\alpha(\bbX)$, so that $\Xi_\rho(\alpha)\in(0,\alpha(\bbX)]$ for $\alpha\in\calA_\rho$, and $\Xi_\rho(t\alpha)=t\,\Xi_\rho(\alpha)$ for $t>0$. Thus $\Xi_\rho$ carries a positive homogeneity, while $m_{\rho\sigma}$ turns debiasability into the Cauchy--Schwarz-type inequality $m_{\mu\nu}^2\le m_{\mu\mu}m_{\nu\nu}$, similarly to \Cref{lem:inf-rep_kernel}.

    Similarly to \cite{feydy2019interpolating}, we start by rewriting $\OT^\eps_c$ based on its dual formulation as a supremum over measures.
	\begin{lemma}[Scale-invariant dual]\label{lem:dual_scaleinv}
		For all $\mu,\nu\in\calP(\bbX)$,
		\begin{equation}\label{eq:dual_recall}
			\tfrac1\eps\OT^\eps_c(\mu,\nu)=\sup_{\alpha,\beta\in\calM^b_+(\bbX)\setminus\{0\}}
			\Big[-\KL(\mu|\alpha)-\KL(\nu|\beta)-\log \langle k_\alpha,k_\beta \rangle\Big],
		\end{equation}
		equivalently
		\begin{equation}\label{eq:m_inf}
			m_{\mu\nu}=\inf_{\alpha\in\calA_\mu,\ \beta\in\calA_\nu}\frac{\langle k_\alpha,k_\beta \rangle}{\Xi_\mu(\alpha)\Xi_\nu(\beta)} .
		\end{equation}
		In particular $\OT^\eps_c\ge-\eps\log M$, so $m_{\mu\nu}\le M$. Moreover the self-transport dual may be
		restricted to the diagonal and then normalized: with
		\begin{equation}\label{eq:def_bk}
		    \calB_k\coloneqq\big\{\alpha\in\calM^b_+(\bbX)\setminus\{0\}:\ \langle k_\alpha,k_\alpha \rangle\le1\big\},
		\end{equation}
		which contains $k(x,x)^{-1/2}\delta_x$ for all $x\in\bbX$, one has, for every $\rho\in\calP(\bbX)$,
		\begin{equation}\label{eq:diagonal_dual}
			m_{\rho\rho}=\inf_{\alpha\in\calA_\rho}\frac{\langle k_\alpha,k_\alpha \rangle}{\Xi_\rho(\alpha)^2}
			\qquad\text{and}\qquad
			\tfrac12\OT^\eps_c(\rho,\rho)=-\eps\inf_{\alpha\in\calB_k}\KL(\rho|\alpha)\ \in\ \big[-\tfrac\eps2\log M,+\infty\big].
		\end{equation}
	\end{lemma}
	\begin{proof}

        We first show that
        \[
		\OT^\eps_c(\mu,\nu)=\sup_{u,\,v\ \text{bounded measurable}}
		\Big[\int u\,\de\mu+\int v\,\de\nu-\eps\iint e^{(u\oplus v-c)/\eps}\de\mu\,\de\nu+\eps\Big].
		\]
        If $c$ is bounded, this is a consequence of \cite[Theorem 4.7, Lemma 4.9]{Nutz2021IntroductionTE}. Otherwise, consider the sequence of lower semi-continuous and bounded functions $c_n=c\wedge n$, converging increasingly to $c$. Since $c_n$ is increasing, $\lim_n\OT^\eps_{c_n}(\mu,\nu)\le\OT^\eps_c(\mu,\nu)$, and it remains to show the converse inequality. Consider the corresponding sequence $\pi_n$ of optimal plans for $\OT^\eps_{c_n}(\mu,\nu)$. It holds $\eps\KL(\pi_n|\mu\otimes\nu)\le\OT^\eps_c(\mu,\nu)-\inf c$, so that the sequence is narrowly compact and we can extract a converging subsequence, still labeled $\pi_n$, converging to $\overline{\pi}$. By monotone convergence of $\int c_n\de\overline\pi$, for any $\delta>0$ there exists $m$ such that $\int c_m\de\overline\pi\ge \int c\de\overline\pi-\delta$ and therefore, by lower semi-continuity of $\pi\mapsto \int c_m\de\pi+\eps\KL(\pi|\mu\otimes\nu)$,
        \[
        \begin{aligned}
        \lim_n \OT^\eps_{c_n}(\mu,\nu)&=\lim_n \int c_n\de \pi_n+\eps\KL(\pi_n|\mu\otimes\nu)\stackrel{c_n\ge c_m}{\ge} \lim_n \int c_m\de \pi_n+\eps\KL(\pi_n|\mu\otimes\nu) \\
        &\ge \int c_m\de \overline\pi+\eps\KL(\overline\pi|\mu\otimes\nu)
        \ge \int c\de \overline\pi+\eps\KL(\overline\pi|\mu\otimes\nu)-\delta\ge \OT^\eps_c(\mu,\nu)-\delta.
        \end{aligned}
        \]
        whence the result taking $\delta$ to $0$.

		Substitute $\alpha\coloneqq e^{u/\eps}\mu$ and $\beta\coloneqq e^{v/\eps}\nu$. Then
		$\tfrac1\eps\int u\,\de\mu=\int\log\frac{\de\alpha}{\de\mu}\de\mu=-\KL(\mu|\alpha)$, likewise for
		$\beta$, and $\iint e^{(u\oplus v-c)/\eps}\de\mu\,\de\nu=\iint k\,\de\alpha\,\de\beta=\langle k_\alpha,k_\beta \rangle$,
		so that
		\begin{equation}\label{eq:dual_unnormalized}
			\tfrac1\eps\OT^\eps_c(\mu,\nu)=\sup_{\alpha,\beta}\Big[-\KL(\mu|\alpha)-\KL(\nu|\beta)-\langle k_\alpha,k_\beta \rangle+1\Big],
		\end{equation}
		the supremum being over $\alpha\ll\mu$, $\beta\ll\nu$ with bounded log-densities. By weak duality, the bracket is also at most $\tfrac1\eps\OT^\eps_c(\mu,\nu)$ for all $\alpha\ll\mu$ and $\beta\ll\nu$ in $\calA_\mu\times\calA_\nu$. Indeed, their log-densities lie in $L^1(\mu)$ and $L^1(\nu)$, since $\big(\log\frac{\de\alpha}{\de\mu}\big)^+\le\frac{\de\alpha}{\de\mu}$. The supremum is therefore unchanged if taken over all nonzero $\alpha,\beta\in\calM^b_+(\bbX)$: pairs outside $\calA_\mu\times\calA_\nu$ give $-\infty$, and replacing $\alpha$ by its absolutely continuous part with respect to $\mu$ leaves $\KL(\mu|\alpha)$ unchanged and does not increase $\langle k_\alpha,k_\beta \rangle$, $k$ being positive. This establishes \eqref{eq:dual_recall} up to a constant.

		To deal with the latter, fix $\alpha,\beta$ and rescale $(\alpha,\beta)\mapsto(t\alpha,t\beta)$, $t>0$. Since
		$-\KL(\mu|t\alpha)=-\KL(\mu|\alpha)+\log t$, the bracket in
		\eqref{eq:dual_unnormalized} becomes
		\[
		\varphi(t)\coloneqq\Big[-\KL(\mu|\alpha)-\KL(\nu|\beta)\Big]+2\log t-t^2\langle k_\alpha,k_\beta \rangle+1 .
		\]
		As $\langle k_\alpha,k_\beta \rangle\in(0,\infty)$, $\varphi$ is maximal at $t^2=1/\langle k_\alpha,k_\beta \rangle$, with value
		$-\KL(\mu|\alpha)-\KL(\nu|\beta)-\log \langle k_\alpha,k_\beta \rangle$. Taking the supremum first in $t$ and then in
		$(\alpha,\beta)$ gives \eqref{eq:dual_recall}, and \eqref{eq:m_inf} is its exponentiated form.
		Choosing $\alpha=\mu$, $\beta=\nu$ in \eqref{eq:dual_recall} gives
		$\tfrac1\eps\OT^\eps_c(\mu,\nu)\ge-\log \langle k_\mu,k_\nu\rangle\ge-\log M$.

		For the first identity in \eqref{eq:diagonal_dual}, let $\alpha,\beta\in\calA_\rho$ and let $\gamma$ be the measure with
		$\de\gamma=\sqrt{ab}\,\de\rho$, where $\alpha=a\rho+\alpha^\perp$ and $\beta=b\rho+\beta^\perp$. Then
		$\KL(\rho|\gamma)=\tfrac12\KL(\rho|\alpha)+\tfrac12\KL(\rho|\beta)$, i.e.\
		$\Xi_\rho(\gamma)^2=\Xi_\rho(\alpha)\Xi_\rho(\beta)$, while by the Cauchy--Schwarz inequality in
		$L^2(k\,\de\rho\,\de\rho)$ and the symmetry of $k$,
        \begin{align*}
            \langle k_\gamma,k_\gamma\rangle=\iint k\sqrt{a(x)b(x)}\sqrt{a(y)b(y)}\,\de\rho\,\de\rho
		&\le\Big(\iint k\,a(x)b(y)\Big)^{\!1/2}\Big(\iint k\,b(x)a(y)\Big)^{\!1/2}\\
		&=\langle k_{a\rho},k_{b\rho}\rangle\stackrel{\alpha^\perp, \beta^\perp\ge 0}{\le} \langle k_\alpha,k_\beta \rangle.
        \end{align*}
		Hence
		$\langle k_\gamma,k_\gamma\rangle/\Xi_\rho(\gamma)^2\le \langle k_\alpha,k_\beta \rangle/(\Xi_\rho(\alpha)\Xi_\rho(\beta))$, so the
		infimum in \eqref{eq:m_inf} with $\mu=\nu=\rho$ is unchanged when restricted to $\alpha=\beta$.

		For the second identity, the ratio $\langle k_\alpha,k_\alpha \rangle/\Xi_\rho(\alpha)^2$ is invariant under
		$\alpha\mapsto t\alpha$, $t>0$, and $\langle k_\alpha,k_\alpha \rangle\in(0,+\infty)$ for every nonzero
		$\alpha\in\calM^b_+(\bbX)$; replacing $\alpha$ by $\alpha/\langle k_\alpha,k_\alpha \rangle^{1/2}$ shows that the
		infimum may be restricted to $\{\langle k_\alpha,k_\alpha \rangle=1\}$, where it equals $e^{2\KL(\rho|\alpha)}$.
		Since $\KL(\rho|\alpha/\langle k_\alpha,k_\alpha \rangle^{1/2})\le\KL(\rho|\alpha)$ for $\alpha\in\calB_k$, the
		infimum of $\KL(\rho|\cdot)$ over $\calB_k$ is attained along the sphere as well, whence
		$m_{\rho\rho}=\exp\big(2\inf_{\calB_k}\KL(\rho|\cdot)\big)$, i.e.\
		$\tfrac12\OT^\eps_c(\rho,\rho)=-\eps\inf_{\calB_k}\KL(\rho|\cdot)$; the lower bound follows from
		$\OT^\eps_c\ge-\eps\log M$.
	\end{proof}

We can now state a result on the debiasability of $\OT^\eps_c$ without giving an inf-representation: \ref{it:psd_debias} and \ref{it:psd_necessary} show that, at fixed $\eps$, positive semidefiniteness of $k$ is not merely sufficient for debiasability, as in \cite[Theorem 1]{feydy2019interpolating}, but characterizes it, while, for the sake of completeness, part \ref{it:psd_convex} derives from the same duality the convexity properties of the self-transport cost and of the Sinkhorn divergence proved in \cite{feydy2019interpolating} but with weaker assumptions. The corresponding inf-representation is provided by \Cref{thm:selfantipolar} below.

    	\begin{theorem}[Positive semidefinite kernels]\label{thm:EOT_debias_kernel_compact}
		Let $c:\bbX\times\bbX\to\R$ be symmetric and Borel, and let $k=e^{-c/\eps}$.
		\begin{enumerate}[labelindent=0cm,leftmargin=*,topsep=0.1cm,partopsep=0cm,parsep=0.1cm,itemsep=0.1cm,label=\roman*)]
			\item\label{it:psd_debias} \emph{(debiasability)} If $k$ is positive semidefinite, then
			$\OT^\eps_c$ is debiasable on $\calP_c(\bbX)$.%
			\item\label{it:psd_concave}\label{it:psd_convex} \emph{(convexity)} If $k$ is positive semidefinite, then $\mu\mapsto\OT^\eps_c(\mu,\mu)$ is concave on $\calP_c(\bbX)$, and $\Seps\in[0,+\infty]$ is convex in each of its two variables separately on $\calP_c(\bbX)$, and finite at $(\mu,\nu)$ as soon as $\iint c\,\de\rho\,\de\sigma<+\infty$ for all $\rho,\sigma\in\{\mu,\nu\}$. When $c$ is bounded from below, these properties hold on all of $\calP(\bbX)$, and $\calB_k$ defined in \eqref{eq:def_bk} is convex.
			\item\label{it:psd_necessary} \emph{(necessity)} Conversely, if $\Seps(\mu,\nu)\ge0$ for all
			finitely supported $\mu,\nu\in\calP(\bbX)$, then $k$ is positive semidefinite. 
		\end{enumerate}
        Debiasability of
			$\OT^\eps_c$ on $\calP_c(\bbX)$ is therefore \emph{equivalent} to positive semidefiniteness of
			$k$.
	\end{theorem}
    	\begin{remark}[Strict debiasability and metrization]\label{rem:metrization}
		Let $\rho\in\calP(\bbX)$ with $\OT^\eps_c(\rho,\rho)<+\infty$, and let $f_\rho$ be a symmetric optimal potential, as recalled after \eqref{eq:sinkhorn-fixed-point}, so that $\pi_\rho=k\,\rho^\sharp\otimes\rho^\sharp$ with $\rho^\sharp\coloneqq e^{f_\rho/\eps}\rho$. Assume that $\rho^\sharp$ is finite. Then $f_\rho\in L^1(\rho)$, as $e^{-f_\rho/\eps}=k_{\rho^\sharp}\le M\rho^\sharp(\bbX)$ and $f_\rho^+\le\eps e^{f_\rho/\eps}$, so that $\OT^\eps_c(\rho,\rho)=\eps\KL(\pi_\rho|k\,\rho\otimes\rho)=2\int f_\rho\,\de\rho$ and
			\begin{equation*}
				\langle k_{\rho^\sharp},k_{\rho^\sharp}\rangle=\pi_\rho(\bbX\times\bbX)=1,\qquad
				\Xi_\rho(\rho^\sharp)=e^{\frac1\eps\int f_\rho\,\de\rho}=m_{\rho\rho}^{-1/2}.
			\end{equation*}
		For $\mu,\nu$ with $\mu^\sharp,\nu^\sharp$ finite, testing $(\alpha,\beta)=(\mu^\sharp,\nu^\sharp)$ in \eqref{eq:m_inf} and, when $k$ is positive semidefinite, using $-\log \langle a,b\rangle\ge1-\langle a,b\rangle=\tfrac12|a-b|_{\calH_k}^2$ for unit vectors $a,b$, gives
        	\begin{equation*}
				\Seps(\mu,\nu)\ \ge\ -\eps\log\big\langle k_{\mu^\sharp},k_{\nu^\sharp}\big\rangle\ \ge\ \frac\eps2\,\big|k_{\mu^\sharp}-k_{\nu^\sharp}\big|^2_{\calH_k}.
			\end{equation*}
		As in \Cref{cor:strict_deb_lse}, if $k$ is moreover integrally strictly positive definite, $\Seps(\mu,\nu)=0$ thus forces $\mu^\sharp=\nu^\sharp$, hence $\mu=\nu$ by Sinkhorn's equations. This covers the setting of \cite[Theorem 1]{feydy2019interpolating}: for $\bbX$ compact and $c$ continuous, the potentials are bounded \cite[Lemma 4.9]{Nutz2021IntroductionTE}, so that $\mu^\sharp$ is finite for every $\mu\in\calP(\bbX)$, and a universal kernel is integrally strictly positive definite \cite{sriperumbudur11a}. Together with the weak continuity of $\OT^\eps_c$ and the weak compactness of $\calP(\bbX)$, one recovers that $\Seps$ metrizes the weak convergence.
	\end{remark}

    	    	\begin{proof}
		Let $k$ be positive semidefinite, $s\coloneqq\tfrac12c(\cdot,\cdot)$ and $c_0\coloneqq c-s\oplus s$. The kernel $e^{-c_0/\eps}=k(x,y)/(k(x,x)k(y,y))^{1/2}$ is positive semidefinite and bounded by $1$ by the Cauchy--Schwarz inequality, i.e.\ $c_0\ge0$. For $\mu,\nu\in\calP_c(\bbX)$, $s$ is $\mu$- and $\nu$-integrable, so that $\OT^\eps_c(\mu,\nu)=\OT^\eps_{c_0}(\mu,\nu)+\int s\,\de\mu+\int s\,\de\nu$. Hence $\Seps=\operatorname{S}^\eps_{c_0}$ on $\calP_c(\bbX)\times\calP_c(\bbX)$, while $\OT^\eps_c(\mu,\mu)$ and $\OT^\eps_{c_0}(\mu,\mu)$ differ by the affine function $2\int s\,\de\mu$ of $\mu$, and $\iint c_0\,\de\rho\,\de\sigma<+\infty$ if $\iint c\,\de\rho\,\de\sigma<+\infty$. Items \ref{it:psd_debias} and \ref{it:psd_convex} on $\calP_c(\bbX)$ thus follow from their versions on $\calP(\bbX)$ for $c_0$, and in their proofs we may assume that $c$ is bounded from below, so that \Cref{lem:dual_scaleinv} applies.
        
		\ref{it:psd_debias}. Fix $\mu,\nu$ and $\delta>0$, and assume first that $m_{\mu\mu},m_{\nu\nu}>0$. By \eqref{eq:diagonal_dual} there is $\alpha\in\calA_\mu$ with
		$\langle k_\alpha,k_\alpha \rangle/\Xi_\mu(\alpha)^2\le e^{\delta}m_{\mu\mu}$, and since this ratio is invariant under
		$\alpha\mapsto t\alpha$ we may normalize $\langle k_\alpha,k_\alpha \rangle=1$, so that
		$\Xi_\mu(\alpha)^{-2}\le e^\delta m_{\mu\mu}$. Choose $\beta\in\calA_\nu$ analogously with
		$\langle k_\beta,k_\beta\rangle=1$ and $\Xi_\nu(\beta)^{-2}\le e^\delta m_{\nu\nu}$. Positive semidefiniteness gives
		the Cauchy--Schwarz inequality $\langle k_\alpha,k_\beta \rangle\le\sqrt{\langle k_\alpha,k_\alpha \rangle\langle k_\beta,k_\beta\rangle}=1$, whence by
		\eqref{eq:m_inf}
		\[
		m_{\mu\nu}\ \le\ \frac{\langle k_\alpha,k_\beta \rangle}{\Xi_\mu(\alpha)\Xi_\nu(\beta)}
		\ \le\ \frac{1}{\Xi_\mu(\alpha)\Xi_\nu(\beta)}\ \le\ e^{\delta}\sqrt{m_{\mu\mu}m_{\nu\nu}} .
		\]
		Letting $\delta\downarrow0$ gives $m_{\mu\nu}\le\sqrt{m_{\mu\mu}m_{\nu\nu}}$, which is the claim that $\Seps(\mu,\nu)\ge0$. If $m_{\mu\mu}=0$, i.e.\ $\OT^\eps_c(\mu,\mu)=+\infty$, the same
		computation with normalized $\alpha$ such that $\Xi_\mu(\alpha)^{-2}$ is arbitrarily small, and a fixed normalized $\beta\in\calA_\nu$, e.g.\ a multiple of $\nu$, forces $m_{\mu\nu}=0$, and likewise if $m_{\nu\nu}=0$.

    		\ref{it:psd_convex}. Positive semidefiniteness makes $\rho\mapsto \langle k_{\rho},k_{\rho}\rangle^{1/2}$ a seminorm on the finite signed measures, so that $\calB_k$, the intersection of its unit ball with the convex cone $\calM^b_+(\bbX)\setminus\{0\}$, is convex. By \eqref{eq:diagonal_dual}, $\rho\mapsto\OT^\eps_c(\rho,\rho)$ is then a partial supremum of the jointly concave function $(\rho,\alpha)\mapsto-2\eps\KL(\rho|\alpha)$ over $\alpha\in\calB_k$, hence concave. Likewise, $\mu\mapsto\OT^\eps_c(\mu,\nu)$ is convex, as a partial infimum of the jointly convex function $(\pi,\mu)\mapsto\int c\,\de\pi+\eps\KL(\pi|\mu\otimes\nu)$ over the convex set $\{(\pi,\mu):\pi\in\Pi(\mu,\nu)\}$. The set $\{\rho:\OT^\eps_c(\rho,\rho)<+\infty\}$ is convex: for $\pi_i\in\Pi(\mu_i,\mu_i)$, $t\in(0,1)$ and $\mu_t\coloneqq t\mu_0+(1-t)\mu_1$, the plan $t\pi_0+(1-t)\pi_1\in\Pi(\mu_t,\mu_t)$, the convexity of $\KL(\cdot|\mu_t\otimes\mu_t)$ and the bounds $\frac{\de(\mu_0\otimes\mu_0)}{\de(\mu_t\otimes\mu_t)}\le t^{-2}$, $\frac{\de(\mu_1\otimes\mu_1)}{\de(\mu_t\otimes\mu_t)}\le (1-t)^{-2}$ give $\OT^\eps_c(\mu_t,\mu_t)\le t\,\OT^\eps_c(\mu_0,\mu_0)+(1-t)\,\OT^\eps_c(\mu_1,\mu_1)+2\eps\log2$. On this set, $\Seps(\cdot,\nu)$ is the sum of the two convex functions $\OT^\eps_c(\cdot,\nu)$ and $-\tfrac12\OT^\eps_c(\cdot,\cdot)$ and of a constant; outside of it, or if $\OT^\eps_c(\nu,\nu)=+\infty$, it equals $+\infty$ by \ref{it:psd_debias} and the rule $\infty-\infty=+\infty$. Hence $\Seps(\cdot,\nu)$ is convex, and so is $\Seps(\mu,\cdot)$ by symmetry. Finally, testing $\pi=\rho\otimes\sigma$ in \eqref{eq:oteps} gives $\OT^\eps_c(\rho,\sigma)\le\iint c\,\de\rho\,\de\sigma$, while $\OT^\eps_c\ge\inf c$, whence the finiteness assertion.

		\ref{it:psd_necessary}. Let $x_1,\dots,x_N\in\bbX$ be distinct, and let $\alpha=\sum_iu_i\delta_{x_i}$ and $\beta=\sum_iv_i\delta_{x_i}$ with $u,v\in\R^N_{>0}$. The probability measure $\pi\coloneqq k\,\alpha\otimes\beta/\langle k_\alpha,k_\beta \rangle$ has finitely supported marginals $\mu$ and $\nu$, with $\mu$ and $\alpha$ mutually absolutely continuous, and similarly for $\nu$ and $\beta$, $k$ being positive. As $\frac{\de\pi}{\de(k\,\mu\otimes\nu)}(x,y)=\frac{\de\alpha}{\de\mu}(x)\frac{\de\beta}{\de\nu}(y)/\langle k_\alpha,k_\beta \rangle$, testing $\pi$ in \eqref{eq:oteps} gives
		\[
		\tfrac1\eps\OT^\eps_c(\mu,\nu)\le\tfrac1\eps\int c\,\de\pi+\KL(\pi|\mu\otimes\nu)=\KL\big(\pi\,\big|\,k\,\mu\otimes\nu\big)=-\KL(\mu|\alpha)-\KL(\nu|\beta)-\log \langle k_\alpha,k_\beta \rangle,
		\]
		while \eqref{eq:dual_recall}, evaluated at the pairs $(\alpha,\alpha)$ and $(\beta,\beta)$, gives $\tfrac1\eps\OT^\eps_c(\mu,\mu)\ge-2\KL(\mu|\alpha)-\log \langle k_\alpha,k_\alpha \rangle$, and likewise for $\nu$. Only the inequality ``$\ge$'' in \eqref{eq:dual_recall} is used here, which follows from weak duality alone and thus holds for finitely supported measures without any lower bound on $c$. All these quantities being finite, subtracting yields
		\[
		0\ \le\ \tfrac1\eps\Seps(\mu,\nu)\ \le\ \tfrac12\log \langle k_\alpha,k_\alpha \rangle+\tfrac12\log \langle k_\beta,k_\beta\rangle-\log \langle k_\alpha,k_\beta \rangle,
		\]
		i.e.\ $u^\top Kv\le(u^\top Ku)^{1/2}(v^\top Kv)^{1/2}$ with $K\coloneqq(k(x_i,x_j))_{i,j}$, for all $u,v\in\R^N_{>0}$. Given $\rho\in\R^N$ and $\delta>0$, the vectors $u=\rho_++\delta\mathbf 1$ and $v=\rho_-+\delta\mathbf 1$ are positive and satisfy $u-v=\rho$, whence
		$\rho^\top K\rho=u^\top Ku-2u^\top Kv+v^\top Kv\ge\big((u^\top Ku)^{1/2}-(v^\top Kv)^{1/2}\big)^2\ge0$. As $N$ and the points $x_i$ were arbitrary, $k$ is positive semidefinite.
\end{proof}

\paragraph{Splitting the coupling term.}%
By \Cref{lem:dual_scaleinv}, the two marginals interact in \eqref{eq:dual_recall} only through the term $-\log \langle k_\alpha,k_\beta \rangle$, or in \eqref{eq:dual_unnormalized} through $-\langle k_\alpha,k_\beta \rangle$. The latter has an inf-representation by Cauchy-Schwarz, similarly to \Cref{lem:inf-rep_kernel}. Applied to \eqref{eq:dual_recall} this would give a $\sup\inf$ and, if strong duality holds, then we get the inf-representation for $\OT^\eps_c$ from a minimax argument. Unfortunately natural candidates fail. For instance, when $k$ is positive semidefinite, the parallelogram identity $-\langle a,b\rangle=\inf_{z\in\calH_k}\big[|z-a|_k^2-\tfrac12|a|_k^2\big]+\big[|z-b|_k^2-\tfrac12|b|_k^2\big]$ splits the term $-\langle k_\alpha,k_\beta \rangle$ of \eqref{eq:dual_unnormalized}, but its one-sided terms grow quadratically along $\alpha\mapsto t\alpha$, while $-\KL(\mu|t\alpha)=-\KL(\mu|\alpha)+\log t$ only grows logarithmically, so that the resulting supremum over $\alpha$ is infinite. Exploiting that we are considering a subset of the cone of positive functions, what does work is a \emph{reverse} Cauchy--Schwarz inequality
\[
	\lambda(a)\lambda(b)\ \le\ \langle a,b\rangle
\]
for all $a,b\in\calK$, where $\lambda:\calK\to[0,\infty)$ is a functional positively homogeneous of degree one. It gives $-\log\langle k_\alpha,k_\beta \rangle\le-\log\lambda(k_\alpha)-\log\lambda(k_\beta)$, whose one-sided terms $-\KL(\mu|\alpha)-\log\lambda(k_\alpha)$ are invariant under $\alpha\mapsto t\alpha$. Inserting this bound in \eqref{eq:dual_recall} and taking the infimum over a family of such functionals leads to a minimax problem, a supremum over $(\alpha,\beta)$ of an infimum over $\lambda$, which splits into a function of $\mu$ plus a function of $\nu$ once the order of the two optimizations is exchanged. Self-antipolar sets, introduced below, provide such a family. \Cref{thm:selfantipolar} shows that this minimax never has a duality gap, and that the relaxation of \eqref{eq:dual_recall} by the reverse Cauchy--Schwarz inequality is exact precisely when $\OT^\eps_c$ is debiasable, i.e.\ when $k$ is positive semidefinite.

	\paragraph{Antipolarity.}
	For a convex body $K\ni0$ of a Euclidean space, the polar $K^\circ=\{f:\langle f,g\rangle\le1\ \forall g\in K\}$ coincides with $K$ only for the unit ball. On the cone $\calK$, we reverse the inequality: for a nonempty $C\subseteq\calK$, the \emph{antipolar} of $C$ and the \emph{gauge} of $C$ at $f\in\calK$ are
	\[
	C^\dagger\coloneqq\{f\in\calK:\ \langle f,g\rangle\ge1\ \text{ for all }g\in C\},
	\qquad
	\lambda_C(f)\coloneqq\inf_{g\in C}\langle f,g\rangle\in[0,\infty),
	\]
	so that $C^\dagger=\{f\in\calK:\lambda_C(f)\ge1\}$ and $\lambda_C$ is positively homogeneous of degree one. We call $C$ \emph{compatible} if $\langle f,g\rangle\ge1$ for all $f,g\in C$, i.e.\ if $C\subseteq C^\dagger$, and we index our next inf-representation by the self-antipolar sets, the cone analogues of the unit ball:
	\[
	\calZ\coloneqq\{C\subseteq\calK:\ C\neq\emptyset,\ C=C^\dagger\}.
	\]
	For instance, for $x\in\bbX$, $k_x\coloneqq k_{\delta_x}\in\calK$ satisfies $\langle k_x,k_x\rangle=k(x,x)>0$, so that $\{k_x/k(x,x)^{1/2}\}$ is compatible. The gauge of $C\in\calZ$ is an \emph{antinorm} in the sense of \cite{Protasov2021}: nonnegative, positively homogeneous, concave and superadditive on $\calK$, but possibly vanishing on some directions.

	\begin{lemma}[Self-antipolar sets]\label{lem:antipolar}
    The following holds:
		\begin{enumerate}[label=\roman*)]
			\item\label{it:zorn} Every nonempty compatible $C_0\subseteq\calK$ is contained in some $C\in\calZ$; in particular $\calZ\neq\emptyset$.
			\item\label{it:gauge} For $C\in\calZ$ and $f,h\in\calK$: $f\in C$ if and only if $\lambda_C(f)\ge1$, and
			\begin{equation}\label{eq:rev_CS}
				\lambda_C(f)\lambda_C(h)\le\langle f,h\rangle,\qquad\text{in particular}\qquad \lambda_C(f)\le\langle f,f\rangle^{1/2}.
			\end{equation}
		\end{enumerate}
	\end{lemma}
	\begin{proof}
		\ref{it:zorn} Consider the compatible subsets of $\calK$ containing $C_0$, equipped with the inclusion partial order. The union of a chain is compatible, any two of its elements lying in a common member, so Zorn's lemma gives a maximal compatible $C\supseteq C_0$, i.e.\ $C\subseteq C^\dagger$. Let $f\in C^\dagger$. If $\langle f,f\rangle\ge1$, then $C\cup\{f\}$ is compatible and $f\in C$ by maximality. Otherwise $0<\langle f,f\rangle<1$, and $h\coloneqq f/\langle f,f\rangle^{1/2}\in\calK$ satisfies $\langle h,h\rangle=1$ and $\langle h,g\rangle\ge\langle f,f\rangle^{-1/2}>1$ for all $g\in C$; so $h\in C$ by maximality, and $f\in C^\dagger$ would force $1\le\langle f,h\rangle=\langle f,f\rangle^{1/2}<1$, a contradiction. Hence $C=C^\dagger\in\calZ$. Applied to the compatible singleton $C_0=\{k_x/k(x,x)^{1/2}\}$, this shows $\calZ\neq\emptyset$. Only the bilinearity, symmetry and positivity of $\langle\cdot,\cdot\rangle$ on $\calK\times\calK$ have been used.

		\ref{it:gauge} For $t>0$, $f/t\in C=C^\dagger$ if and only if $\langle f,g\rangle\ge t$ for all $g\in C$, i.e.\ $t\le\lambda_C(f)$. Hence $f\in C$ if and only if $\lambda_C(f)\ge1$, and $f/\lambda_C(f)\in C$ whenever $\lambda_C(f)>0$. If $\lambda_C(f)$ and $\lambda_C(h)$ are positive, the compatibility of $f/\lambda_C(f)$ and $h/\lambda_C(h)$ gives \eqref{eq:rev_CS}; otherwise \eqref{eq:rev_CS} is trivial, since $\langle f,h\rangle>0$. Taking $h=f$ gives the last bound.
	\end{proof}

	\begin{theorem}[Self-antipolar inf-representation]\label{thm:selfantipolar}
		For $\mu\in\calP(\bbX)$ and $C\in\calZ$, let
		\[
		\Psi_\eps(\mu,C)\coloneqq\eps\sup_{g\in C}\int_\bbX\big(-\log g\big)\,\de\mu\ \in\ \R\cup\{+\infty\},
		\]
		which is well defined since every $g\in\calK$ is bounded and positive, and, for $\mu,\nu\in\calP(\bbX)$, introduce the Lagrangian
		\begin{equation}\label{eq:antipolar_lagrangian}
			\calL\big((\alpha,\beta),C\big)\coloneqq
			\eps\log\frac{\Xi_\mu(\alpha)}{\lambda_C(k_\alpha)}+\eps\log\frac{\Xi_\nu(\beta)}{\lambda_C(k_\beta)},
			\qquad \alpha\in\calA_\mu,\ \beta\in\calA_\nu,\ C\in\calZ,
		\end{equation}
        Then $\Psi_\eps\ge-\tfrac\eps2\log M$,
		\begin{equation}\label{eq:Psi_dual}
			\Psi_\eps(\mu,C)=\eps\sup_{\alpha\in\calA_\mu}\log\frac{\Xi_\mu(\alpha)}{\lambda_C(k_\alpha)},
		\end{equation}
		and, for all $\mu,\nu\in\calP(\bbX)$, the minimax problem associated with \eqref{eq:antipolar_lagrangian} has no duality gap:
		\begin{multline}\label{eq:minimax_chain}
			\OT^\eps_c(\mu,\nu)\ \le\ \sup_{\alpha,\beta}\ \inf_{C\in\calZ}\calL\big((\alpha,\beta),C\big)
			\ =\ \inf_{C\in\calZ}\ \sup_{\alpha,\beta}\calL\big((\alpha,\beta),C\big)\\
			\stackrel{\eqref{eq:Psi_dual}}{=}\inf_{C\in\calZ}\Psi_\eps(\mu,C)+\Psi_\eps(\nu,C)
			=\max\Big(\OT^\eps_c(\mu,\nu),\ \tfrac12\OT^\eps_c(\mu,\mu)+\tfrac12\OT^\eps_c(\nu,\nu)\Big),
		\end{multline}
		the infimum over $C$ being attained whenever it is finite. In particular, $\OT^\eps_c$ is debiasable on $\calP(\bbX)$ if and only if the first inequality in \eqref{eq:minimax_chain} is an equality for all $\mu,\nu$, i.e.\ if and only if
		\begin{equation}\label{eq:selfantipolar}
			\OT^\eps_c(\mu,\nu)=\inf_{C\in\calZ}\ \Psi_\eps(\mu,C)+\Psi_\eps(\nu,C)
			\qquad\text{for all }\mu,\nu\in\calP(\bbX).
		\end{equation}
	\end{theorem}
    The r.h.s.\ of \eqref{eq:minimax_chain} is $\OT^\eps_c(\mu,\nu)+(\Seps(\mu,\nu))^-$, so that the gap is exactly $(\Seps(\mu,\nu))^-$. The quantity $\inf_C\Psi_\eps(\mu,C)+\Psi_\eps(\nu,C)$ is the smallest debiasable majorant of $\OT^\eps_c$ with the same self-transport cost: it is debiasable by \Cref{prop:tpsd_feat_map}, it coincides with $\OT^\eps_c$ on the diagonal by \eqref{eq:minimax_chain}, and any debiasable $G\ge\OT^\eps_c$ with $G(\rho,\rho)=\OT^\eps_c(\rho,\rho)$ satisfies $G(\mu,\nu)\ge\tfrac12G(\mu,\mu)+\tfrac12G(\nu,\nu)$, hence $G\ge\max(\OT^\eps_c,\tfrac12\OT^\eps_c(\mu,\mu)+\tfrac12\OT^\eps_c(\nu,\nu))$. Combined with \Cref{thm:EOT_debias_kernel_compact}, \eqref{eq:selfantipolar} holds on all of $\calP(\bbX)$ if and only if $k$ is positive semidefinite. An attaining index is constructed explicitly in Step 3 of the proof.

	\begin{proof}
		\emph{Step 1: proof of \eqref{eq:Psi_dual}.} Fix $\mu$ and $C\in\calZ$. For $\alpha\in\calA_\mu$ and $g=k_\beta\in C$, one has $\langle k_\alpha,g\rangle=\int g\,\de\alpha\in(0,\infty)$, so that $\lambda_C(k_\alpha)=\inf_{g\in C}\int g\,\de\alpha$ and, $-\log$ being decreasing and continuous on $(0,\infty)$ with $-\log0\coloneqq+\infty$,
		\[
		\sup_{\alpha\in\calA_\mu}\log\frac{\Xi_\mu(\alpha)}{\lambda_C(k_\alpha)}
		\stackrel{\text{def}}{=}\sup_{\alpha\in\calA_\mu}\ \sup_{g\in C}\Big[-\KL(\mu|\alpha)-\log\!\int g\,\de\alpha\Big]
		=\sup_{g\in C}\ \sup_{\alpha\in\calA_\mu}\Big[-\KL(\mu|\alpha)-\log\!\int g\,\de\alpha\Big].
		\]
     	Every $g\in C$ being of the form $k_\beta$, hence Borel with $0<g\le M\beta(\bbX)$, \eqref{eq:Psi_dual} thus follows from the variational formula
		\begin{equation}\label{eq:one_functional}
			\sup_{\alpha\in\calA_\mu}\Big[-\KL(\mu|\alpha)-\log\!\int g\,\de\alpha\Big]=-\int\log g\,\de\mu
			\qquad\text{in }\R\cup\{+\infty\},
		\end{equation}
		valid for bounded Borel $g:\bbX\to(0,\infty)$. It is a form of the Donsker--Varadhan principle: for $\alpha=a\mu+\alpha^\perp\in\calA_\mu$, whose singular part only increases $\int g\,\de\alpha$, and $\varphi\coloneqq\log(ga)$, the bracket is at most $\int\varphi\,\de\mu-\log\int e^\varphi\de\mu-\int\log g\,\de\mu\le-\int\log g\,\de\mu$ \cite[Lemma 1.3]{Nutz2021IntroductionTE}. Equality would hold at $\alpha=\mu/g$, which is finite only if $g$ is bounded away from $0$; in general it is approached along $\alpha_n\coloneqq\mu/(g\vee n^{-1})\in\calA_\mu$, by monotone convergence. Boundedness of $k$, hence of $g$, is used only here.

        Taking $\alpha=\mu$ in \eqref{eq:Psi_dual}, \eqref{eq:rev_CS} gives $\Psi_\eps(\mu,C)\ge-\eps\log\lambda_C(k_\mu)\ge-\frac\eps2\log \langle k_\mu,k_\mu\rangle\ge-\frac\eps2\log M$, so that no indeterminate sum occurs.

		\emph{Step 2: proof of \eqref{eq:minimax_chain}.} We claim that, for all $f,h\in\calK$,
		\begin{equation}\label{eq:sup_gauge_product}
			\max_{C\in\calZ}\lambda_C(f)\lambda_C(h)=\min\Big(\langle f,h\rangle,\ \langle f,f\rangle^{1/2}\langle h,h\rangle^{1/2}\Big).
		\end{equation}
		The inequality ``$\le$'' holds for every $C\in\calZ$ by \eqref{eq:rev_CS}. Conversely, $p\coloneqq\langle f,f\rangle^{1/2}$ and $q\coloneqq\min\big(\langle h,h\rangle^{1/2},\langle f,h\rangle/p\big)$ satisfy $\langle f,f\rangle\ge p^2$, $\langle h,h\rangle\ge q^2$ and $\langle f,h\rangle\ge pq$, so that $\{f/p,h/q\}$ is compatible. By \Cref{lem:antipolar}, it is contained in some $C\in\calZ$, for which $\lambda_C(f)\ge p$ and $\lambda_C(h)\ge q$, where $pq$ is the r.h.s.\ of \eqref{eq:sup_gauge_product}. Applying \eqref{eq:sup_gauge_product} to $f=k_\alpha$ and $h=k_\beta$ gives
		\[
		\inf_{C\in\calZ}\calL\big((\alpha,\beta),C\big)=\max\Big(\eps\log\frac{\Xi_\mu(\alpha)\Xi_\nu(\beta)}{\langle k_\alpha,k_\beta \rangle},\ \eps\log\frac{\Xi_\mu(\alpha)}{\langle k_\alpha,k_\alpha \rangle^{1/2}}+\eps\log\frac{\Xi_\nu(\beta)}{\langle k_\beta,k_\beta\rangle^{1/2}}\Big).
		\]
		Taking the supremum over $(\alpha,\beta)\in\calA_\mu\times\calA_\nu$, the first term gives $\OT^\eps_c(\mu,\nu)$ by \eqref{eq:m_inf}, and the second one, whose suprema in $\alpha$ and $\beta$ decouple, gives $\frac12\OT^\eps_c(\mu,\mu)+\frac12\OT^\eps_c(\nu,\nu)$ by the first identity in \eqref{eq:diagonal_dual}. Hence $\sup_{\alpha,\beta}\inf_C\calL$ equals the last expression of \eqref{eq:minimax_chain}, which in particular proves the first inequality there.%

        By weak duality $\inf_C\sup_{\alpha,\beta}\calL$ is at least the r.h.s.\ of \eqref{eq:minimax_chain}. It remains to construct $C\in\calZ$ such that $\Psi_\eps(\mu,C)+\Psi_\eps(\nu,C)$ is at most this right-hand side, which is $+\infty$ if $m_{\mu\mu}$, $m_{\nu\nu}$ or $m_{\mu\nu}$ vanishes; we thus assume that they are positive. By \eqref{eq:Psi_dual}, the homogeneity of $\lambda_C$ and \Cref{lem:antipolar}-\ref{it:gauge}, $\Psi_\eps(\mu,C)\le-\eps\log p$, for a given $p>0$, holds as soon as $\lambda_C\big(k_\alpha/(p\,\Xi_\mu(\alpha))\big)\ge1$ for all $\alpha\in\calA_\mu$, i.e.\ as soon as $C$ contains the set $p^{-1}E_\mu$, where, for $\rho\in\calP(\bbX)$,
        
		\[
		E_\rho\coloneqq\big\{k_\alpha/\Xi_\rho(\alpha):\ \alpha\in\calA_\rho\big\}\subseteq\calK
		\]
		gathers the kernel means of the measures $\alpha/\Xi_\rho(\alpha)$. Note that by definition of $\Xi_\rho$, we have $\KL(\rho|\alpha/\Xi_\rho(\alpha))=0$. By bilinearity, \eqref{eq:m_inf} reads, for all $\rho,\sigma\in\calP(\bbX)$,
		\begin{equation}\label{eq:E_pairing}
			m_{\rho\sigma}=\inf\big\{\langle f,g\rangle:\ f\in E_\rho,\ g\in E_\sigma\big\}.
		\end{equation}
		Consequently, if $p,q>0$ satisfy
		\begin{equation}\label{eq:compat_scalars}
			p^2\le m_{\mu\mu},\qquad q^2\le m_{\nu\nu},\qquad pq\le m_{\mu\nu},
		\end{equation}
		then, by \eqref{eq:E_pairing}, for $f,g\in p^{-1}E_\mu\cup q^{-1}E_\nu$, we have $\langle f,g \rangle \ge 1$, i.e.\ this set is compatible. \Cref{lem:antipolar}-\ref{it:zorn} then provides $C\in\calZ$ containing it, for which $\Psi_\eps(\mu,C)+\Psi_\eps(\nu,C)\le-\eps\log(pq)$. The choice $p=\sqrt{m_{\mu\mu}}$ and $q=\min\big(\sqrt{m_{\nu\nu}},m_{\mu\nu}/\sqrt{m_{\mu\mu}}\big)$ satisfies \eqref{eq:compat_scalars} with $pq=\min\big(m_{\mu\nu},\sqrt{m_{\mu\mu}m_{\nu\nu}}\big)$, i.e.\ $-\eps\log(pq)=\max\big(\OT^\eps_c(\mu,\nu),\tfrac12\OT^\eps_c(\mu,\mu)+\tfrac12\OT^\eps_c(\nu,\nu)\big)$. This proves \eqref{eq:minimax_chain} and the attainment.

		Finally, if $\OT^\eps_c$ is debiasable, the maximum in \eqref{eq:minimax_chain} equals $\OT^\eps_c(\mu,\nu)$, whence \eqref{eq:selfantipolar}. Conversely, \eqref{eq:selfantipolar} says that $(\calZ,\Psi_\eps)$ is an inf-representation of $\OT^\eps_c$, $\Psi_\eps$ being $\R\cup\{+\infty\}$-valued, so that $\OT^\eps_c$ is debiasable by \Cref{prop:tpsd_feat_map}.
	\end{proof}

	\begin{remark}[Comparison with \Cref{sec:EOT_neg_kernel}]\label{rem:antipolar_connections}
		When $k$ is completely positive as in \eqref{eq:completely_pos_kernel}, both representations are built from the functions $\tilde\psi(\cdot,z)=-\eps\log\rho_\eps(\cdot,z)$ of \eqref{eq:tildepsi-explicit}, and $\langle k_\alpha,k_\beta \rangle=\int_\calZ\hat\alpha\hat\beta\,\de\lambda$ with $\hat\alpha\coloneqq\int\rho_\eps(x,\cdot)\,\de\alpha(x)$. However, \eqref{eq:eot_c_eps_lambda_neg_def} \emph{averages} $\tilde\psi(\cdot,z)$ over $\eta\in\calP(\calZ)$ with the entropic penalty $\eps\KL(\eta|\lambda)$, whereas \eqref{eq:selfantipolar} takes a \emph{supremum} of $-\eps\log g$ over $g\in C$, the sections $\rho_\eps(\cdot,z)$ being natural candidates for $g$ that need not belong to $\calK$. This is the integral-versus-supremum divide of \eqref{eq:sup-exp_representation_k}, and neither statement specializes to the other. Where it applies, \Cref{sec:EOT_neg_kernel} remains preferable: its index space is Polish, its potential is measurable and explicit, and it extends to the time-dependent construction of \Cref{sec:baryc_interpolation}.
	\end{remark}
	
\section{Debiasing entropic OT, $\eps=+\infty$}\label{sec:eps_infty}

Throughout this section $c:\bbX\times \bbX \to \R$ is symmetric and Borel, and we recall from
\eqref{eq:s_eps_infty_intro} that
\[
\OT^\infty_{c}(\mu,\nu)\coloneqq \int_{\bbX\times \bbX} c(x,y) \de \mu(x)\de\nu(y),
\qquad
\operatorname{S}^\infty_c(\mu,\nu)\coloneqq \OT^\infty_{c}(\mu,\nu)-\tfrac12\OT^\infty_{c}(\mu,\mu)-\tfrac12\OT^\infty_{c}(\nu,\nu),
\]
these being definitions: unlike \cite{Neumayer2021}, no $\Gamma$-limit $\eps\to+\infty$ is proved here.
Unlike the cases $\eps\in[0,+\infty)$, $\OT^\infty_c$ is a \emph{bilinear} form on measures, so
expanding the square gives at once, whenever the four pairings are finite,
\begin{equation}\label{eq:S_infty_quadratic}
	\operatorname{S}^\infty_c(\mu,\nu)=-\tfrac12\int_{\bbX\times\bbX} c\,\de\rho\,\de\rho,
	\qquad \rho\coloneqq\mu-\nu .
\end{equation}
If $c$ is negative definite, \eqref{eq:c_neg_def_kernel} gives
$c(x,y)=s(x)+s(y)+|\phi(x)-\phi(y)|^2_{H_c}$ and $m_\mu\coloneqq\int_\bbX\phi\,\de\mu$, the $s$-terms
carry a factor $\rho(\bbX)=0$ and \eqref{eq:S_infty_quadratic} becomes the squared Maximum Mean
Discrepancy
\begin{equation}\label{eq:S_infty_MMD}
	\operatorname{S}^\infty_c(\mu,\nu)=\big|m_\mu-m_\nu\big|_{H_c}^2 ,
\end{equation}
which already indicates the shape of the inf-representation: a squared Hilbert distance between mean
embeddings, so the midpoint identity applies in $H_c$ with intermediate variable $z\in H_c$. A result
similar to \Cref{thm:eps_infty_neg_def}-\ref{it_infty_char} was coincidentally stated in
\cite[Theorem 4.1]{houry2026gromovwassersteinscalesquarednorms}, under the boundedness assumption of
\cite{feydy2019interpolating} which we discard.

In the theorem below, we consider the diagonally integrable measures $\mu\in\calP_c(\bbX)$ of \eqref{eq:def_P_k}, i.e.\ $\iint|c|\,\de\mu\,\de\mu<+\infty$.

\begin{theorem}\label{thm:eps_infty_neg_def}
	Let $\bbX$ be a Polish space and $c:\bbX\times \bbX \to \R$ symmetric and Borel.
	\begin{enumerate}[labelindent=0cm,leftmargin=*,topsep=0.1cm,parsep=0.1cm,itemsep=0.1cm,label=\roman*)]
		\item\label{it_infty_char} The cost $c$ is negative definite if and only if, for all $\mu,\nu\in\calP_c(\bbX)$, the four pairings $\iint|c|\,\de\xi\,\de\zeta$, $\xi,\zeta\in\{\mu,\nu\}$, are finite and
		\begin{equation}\label{eq:s_eps_infty}
			\operatorname{S}^\infty_c(\mu,\nu)= -\tfrac{1}{2}\int_{\bbX\times \bbX} c\,\de (\mu-\nu)\de (\mu-\nu) \ \ge\ 0.
		\end{equation}
		Moreover $\OT^\infty_c$ debiasable implies $c$ debiasable, by considering Dirac masses.
		\item\label{it_infty_rep} Assume in addition that $c$ is negative definite and continuous, and let $(H_c,\phi,s)$ be as in \eqref{eq:c_neg_def_kernel}. Then, for every $\mu\in\calP_c(\bbX)$, $\phi$ is strongly $\mu$-measurable and Bochner $\mu$-integrable, so that $m_\mu\in H_c$ is well defined, and, for all such $\mu,\nu$,
		\begin{equation}\label{eq:ot-explicit_infty_thm}
			\OT^\infty_{c}(\mu,\nu)=\inf_{z\in H_c}\Psi(\mu,z)+\Psi(\nu,z),
		\end{equation}
		where, with $\psi(x,z)\coloneqq s(x)+2|z-\phi(x)|^2_{H_c}$ and
		$\operatorname{Var}_\mu(\phi)\coloneqq\int|\phi|^2_{H_c}\de\mu-|m_\mu|^2_{H_c}$,
		\begin{equation}\label{eq:Psi_infty_correct}
			\Psi(\mu,z)\coloneqq\int_\bbX \psi(x,z)\,\de\mu(x)-\operatorname{Var}_\mu(\phi).
		\end{equation}
		The infimum is attained at the unique point $z^*=\tfrac12(m_\mu+m_\nu)$.
		\item\label{it_infty_strict} Under the assumptions of \ref{it_infty_rep}, $\OT^\infty_c$ is strictly debiasable on the set of measures $\calP_c(\bbX)$ if and only if the mean embedding $\mu\mapsto m_\mu$ is injective on this set.
	\end{enumerate}
\end{theorem}

\begin{proof} Assume $c$ negative definite and take $(H_c,\phi,s)$ as in
	\eqref{eq:c_neg_def_kernel}, so $s=\tfrac12c(\cdot,\cdot)$. For a base point $x_0\in\bbX$, set
	\begin{equation}\label{eq:def_tilde k}
		\tilde k(x,y)\coloneqq \tfrac12\big[c(x,x_0)+c(y,x_0)-c(x,y)-c(x_0,x_0)\big].
	\end{equation}
	Substituting \eqref{eq:c_neg_def_kernel} and $c(x_0,x_0)=2s(x_0)$, all $s$-terms cancel and, with
	$u=\phi(x)-\phi(x_0)$ and $v=\phi(y)-\phi(x_0)$, one is left with
	$\tilde k(x,y)=\tfrac12[|u|^2+|v|^2-|u-v|^2]=(u,v)_{H_c}$: thus $\tilde k$ is positive semidefinite, with $\tilde k(x,x)=|u|^2\ge0$, and
	\begin{equation}\label{eq:c_from_tilde k}
		c(x,y)=s(x)+s(y)+\tilde k(x,x)+\tilde k(y,y)-2\tilde k(x,y).
	\end{equation}
	Let $\mu\in\calP_c(\bbX)$. As $|\phi(x)-\phi(y)|^2_{H_c}=c(x,y)-s(x)-s(y)$ by \eqref{eq:c_neg_def_kernel} and $s\in L^1(\mu)$, we get $\iint|\phi(x)-\phi(y)|^2_{H_c}\,\de\mu(x)\,\de\mu(y)<+\infty$. By Fubini's theorem, $\int_\bbX\tilde k(x,x)\,\de\mu(x)=\int_\bbX|\phi(x)-\phi(x_0)|^2_{H_c}\,\de\mu(x)$ is thus finite for $\mu$-a.e.\ $x_0$, hence for all $x_0\in\bbX$, since $|\phi(x)-\phi(x_1)|^2_{H_c}\le2|\phi(x)-\phi(x_0)|^2_{H_c}+2|\phi(x_1)-\phi(x_0)|^2_{H_c}$. We fix $x_0$ from now on.
    
    Since $2|\tilde k(x,y)|\le\tilde k(x,x)+\tilde k(y,y)$,
	\eqref{eq:c_from_tilde k} gives $|c(x,y)|\le|s(x)|+|s(y)|+2\tilde k(x,x)+2\tilde k(y,y)$, so that the four pairings are finite for all $\mu,\nu\in\calP_c(\bbX)$.
	Finally, integrating \eqref{eq:c_from_tilde k} against $\rho\otimes\rho$ for $\rho=\mu-\nu$, the four terms $s(x)$, $s(y)$, $\tilde k(x,x)$, $\tilde k(y,y)$ each carry
	a factor $\rho(\bbX)=0$, so
	\begin{equation}\label{eq:S_infty_tilde k}
		-\tfrac12\int\!\!\int c\,\de\rho\,\de\rho=\int\!\!\int \tilde k\,\de\rho\,\de\rho
		=\Big|\int_\bbX\phi\,\de\rho\Big|^2_{H_c},
	\end{equation}
    the last equality holding whenever $\phi$ is Bochner integrable with respect to $|\rho|$, e.g.\ under the assumptions of \ref{it_infty_rep}.
    
	\emph{\ref{it_infty_char}} Let $\mu,\nu$ be as in the statement and $\rho=\mu-\nu$ with $c$ negative definite.
	All four pairings being finite, the expansion of $\int\!\!\int c\,\de\rho\,\de\rho$ into four terms
	is legitimate and gives the identity in \eqref{eq:s_eps_infty}. It remains to prove
	$\int\!\!\int\tilde k\,\de\rho\,\de\rho\ge0$, by \eqref{eq:S_infty_tilde k}. Positive
	semidefiniteness of $\tilde k$ gives $|\tilde k(x,y)|\le\sqrt{\tilde k(x,x)\tilde k(y,y)}$ and, by
	Cauchy--Schwarz, $\int\sqrt{\tilde k(x,x)}\de|\rho|\le|\rho|(\bbX)^{1/2}(\int\tilde k(x,x)\de|\rho|)^{1/2}<\infty$,
	so all integrals below converge absolutely. Assume $\rho\neq0$, write $\sigma=|\rho|$,
	$f=\de\rho/\de\sigma\in\{\pm1\}$, $S=\sigma(\bbX)>0$, and let $X_1,\dots,X_n$ be i.i.d.\ with law
	$\sigma/S$. Then $\sum_{i,j\le n}f(X_i)f(X_j)\tilde k(X_i,X_j)\ge0$ pointwise; taking expectations
	and separating the $n$ diagonal from the $n(n-1)$ off-diagonal terms, and using $f^2=1$,
	\[
	0\ \le\ \frac{n(n-1)}{S^2}\int\!\!\int \tilde k\,\de\rho\,\de\rho+\frac{n}{S}\int_\bbX \tilde k(x,x)\,\de\sigma ,
	\]
	so $\int\!\!\int \tilde k\,\de\rho\,\de\rho\ge -\frac{S}{n-1}\int \tilde k(x,x)\de\sigma\to0$.

	Conversely let $x_1,\dots,x_N\in\bbX$ be distinct and $a\in\R^N\setminus\{0\}$ with
	$\sum_i a_i=0$. Then $A_+=\{i:a_i>0\}$ and $A_-=\{i:\pm a_i<0\}$ are both nonempty and
	$M\coloneqq\sum_{A_+}a_i=-\sum_{A_-}a_i>0$. The measures $\mu=\frac1M\sum_{A_+}a_i\delta_{x_i}$ and
	$\nu=-\frac1M\sum_{A_-}a_i\delta_{x_i}$ are finitely supported, hence belong to $\calP_c(\bbX)$ and have finite $|c|$-energy, and $\mu-\nu=\frac1M\sum_i a_i\delta_{x_i}$, so that \eqref{eq:s_eps_infty} reads $-\frac{1}{2M^2}\sum_{i,j}a_ia_jc(x_i,x_j)\ge0$ and $c$ is thus negative definite. The last
	assertion follows from $\OT^\infty_c(\delta_x,\delta_y)=c(x,y)$.%

	\emph{\ref{it_infty_rep}.} Continuity of $c$ makes $s$ continuous, hence $\phi$ continuous by
	\eqref{eq:c_neg_def_kernel}; as $\bbX$ is Polish, $\phi(\bbX)$ is separable, so $\phi$ is strongly
	Borel measurable and $H_c$ may be taken separable. For $\mu$ as in the statement, $\int|\phi|^2_{H_c}\de\mu\le2\int\tilde k(x,x)\,\de\mu+2|\phi(x_0)|^2_{H_c}<\infty$, so that the Bochner integral $m_\mu$ is well defined, and Fubini gives
	$\int\!\!\int(\phi(x),\phi(y))_{H_c}\de\mu\de\nu=(m_\mu,m_\nu)_{H_c}$. Hence
	\begin{equation}\label{eq:OT_infty_expanded}
		\OT^\infty_c(\mu,\nu)=\int s\,\de\mu+\int s\,\de\nu+\int|\phi|^2_{H_c}\de\mu+\int|\phi|^2_{H_c}\de\nu-2(m_\mu,m_\nu)_{H_c}.
	\end{equation}
	Expanding \eqref{eq:Psi_infty_correct},
	$\Psi(\mu,z)=\int(s+|\phi|^2_{H_c})\de\mu+2|z-m_\mu|^2_{H_c}-|m_\mu|^2_{H_c}$, so
	$z\mapsto\Psi(\mu,z)+\Psi(\nu,z)$ is a strictly convex quadratic with gradient
	$8z-4(m_\mu+m_\nu)$, whose unique minimizer is $z^*=\tfrac12(m_\mu+m_\nu)$. Since
	$2|z^*-m_\mu|^2+2|z^*-m_\nu|^2=|m_\mu-m_\nu|^2$ and
	$|m_\mu-m_\nu|^2-|m_\mu|^2-|m_\nu|^2=-2(m_\mu,m_\nu)$, the minimum equals the r.h.s. of
	\eqref{eq:OT_infty_expanded}.

	\emph{\ref{it_infty_strict}.} For $\mu\neq\nu$ as in the statement, \eqref{eq:s_eps_infty} and \eqref{eq:S_infty_tilde k}, which applies by \ref{it_infty_rep}, give $\operatorname{S}^\infty_c(\mu,\nu)=|m_\mu-m_\nu|^2_{H_c}$, which is positive if and only if $m_\mu\neq m_\nu$.
\end{proof}

\begin{example}[Strictness fails exactly at $p=2$]\label{ex:p2_not_strict}
	For $c(x,y)=|x-y|^p_H$ on a Hilbert space $H$ with $p\in(0,2]$ one has $s=0$ and
	$|\phi(x)-\phi(y)|^2_{H_c}=|x-y|^p_H$ \cite[Exercise 2.13, p.~79]{Berg1984}. For $p\in(0,2)$ and $H=\R^d$ the mean embedding is injective on the probability measures with finite moment of order $p$, which are exactly those of $\calP_c(\bbX)=\calP(H)$; this is the characteristic property of the energy distance of exponent $p$ \cite{szekely2013energy}; for $p=1$ this extends to every separable Hilbert space \cite[Theorem 3.16]{lyons2013distance}. For $p=2$ strictness \emph{fails}: here
	$\phi=\mathrm{id}$, so $\operatorname{S}^\infty_c(\mu,\nu)=|m_\mu-m_\nu|^2_H$ by
	\eqref{eq:S_infty_MMD} and it vanishes for any two distinct measures with the same barycenter, e.g.\
	$\mu=\frac12(\delta_{-1}+\delta_{1})$ and $\nu=\delta_0$ on $\R$. The contrast with
	\Cref{thm:strict_neg_def} is instructive: at finite $\eps$, injectivity of
	$\phi$ suffices for strictness, because the relevant kernel on $H_c$ is the \emph{Gaussian} one $e^{-|h-h'|^2/\eps}$, which
	is integrally strictly positive definite; at $\eps=+\infty$ the relevant kernel is the \emph{linear}
	one $(h,h')_{H_c}$, which separates measures only through their barycenter. 
\end{example}

	\section{Beyond midpoints: time-varying interpolation}\label{sec:baryc_interpolation}

     To obtain some of our midpoint formulas, e.g.\  \eqref{eq:condition-const_2}, we used a parallelogram identity for squared Hilbertian distances over Hilbert spaces $H$. The latter actually holds for more general interpolations than midpoints, indeed for any $t\in(0,1)$, $x,y,z\in H$, we have
	\begin{equation}\label{eq:parallelogram_law}
		\frac{1}{t(1-t)}\left|z- (1-t)x-t y\right|^2=\frac{|z-y|^2}{1-t}+\frac{|z-x|^2}{t}-|x-y|^2.
	\end{equation}
     From \eqref{eq:parallelogram_law}, we can also write $|x-y|^2=\inf_{z\in H} \frac{|x-z|^2}{t}+\frac{|y-z|^2}{1-t}$, for any $t\in(0,1)$. This is the main idea pursued in \Cref{sec:barycentric} to obtain a barycentric formulation for entropic interpolation. 
     
     We can also iterate successively this splitting and obtain that, for any partition $\mathcal{T} = \{0 = t_0 < t_1 < \dots < t_N = 1\}$, setting $h_i \coloneqq t_i - t_{i-1}$, we have $|x-y|^2=\inf_{x_0=x,x_1,x_2,\dots, x_N=y}  \sum_{i=1}^N\frac{1}{h_i}
      |x_{i-1}-x_{i}|^2$. This suggests that, if a continuous limit exists, then it corresponds to the length distance. This is the main idea used in \Cref{sec:benamou-brenier} to obtain a Benamou-Brenier formula for entropic interpolation under the mild assumptions that  $\mu,\nu\in\mathcal{P}_2(\mathbb{R}^d)$ and $\KL(\nu | \lambda)<+\infty$.

    \subsection{A barycentric formulation for entropic interpolation}\label{sec:barycentric}
    
		Consider again the squared Euclidean distance cost $c(x,y)=|x-y|^2$ on $\R^d$ with the Lebesgue measure $\lambda$, as done in \Cref{ex:computation_midpoint_Gaussian} for the midpoint. Defining similarly to \eqref{eq:lse-cost} the cost $c^t_{\eps,\lambda}(x,y) \coloneqq -\eps \log \int_{\R^d} e^{-\frac{\psi_t(x,z)+\psi_{1-t}(y,z)}{\eps}}\de\lambda(z)$, for $\psi_t(x,z)\coloneqq\frac{|z-x|^2}{t}$, the same proof as in \Cref{thm:EOT_debiased} provides
		\begin{equation}\label{eq:ceps_Euclidean}
		    \OT_{c^t_{\eps,\lambda}}^\eps(\mu,\nu)
		=\inf_{\eta\in\calP({\R^d})}  \OT^\eps_{\psi_{t}}(\mu,\eta)+\OT^\eps_{\psi_{1-t}}(\nu,\eta)+ \eps \KL(\eta|\lambda).
		\end{equation}
		In this case, using \eqref{eq:parallelogram_law}, we have
		\[
		\int_{\R^d} e^{-\frac{\psi_{t}(x,z)+\psi_{1-t}(y,z)}{\eps}}\de\lambda(z)=e^{-\frac{|x-y|^2}{\eps}}\int_{\R^d} e^{-\frac{|z-(1-t)x-ty|^2}{ t(1-t) \eps}}\de\lambda(z)=e^{-\frac{|x-y|^2}{\eps}}\big(t(1-t)\pi\eps\big)^{\frac{d}{2}}
		\]
		and therefore $c^t_{\eps,\lambda}= c-\frac{d\eps}{2}\log(t(1-t)\pi\eps)$, so that
		\begin{multline}\label{eq:eta_t_Gaussien}
			\OT_c^\eps(\mu,\nu)
			=\inf_{\eta\in\calP({\R^d})}  \OT^\eps_{\psi_{t}}(\mu,\eta)+\OT^\eps_{\psi_{1-t}}(\nu,\eta)+ \eps \KL(\eta|\lambda)+\frac{d\eps}{2}\log(t(1-t)\pi\eps)\\
			=\inf_{\eta\in\calP({\R^d})}  \frac{1}{t}\OT^{\eps t}_{c}(\mu,\eta)+\frac{1}{1-t}\OT^{\eps (1-t)}_{c}(\nu,\eta)+ \eps \KL(\eta|\lambda)+\frac{d\eps}{2}\log(t(1-t)\pi\eps).
		\end{multline}
       For $\eps=0$, we have thus recovered the classical barycentric formulation of the geodesic interpolation according to the Wasserstein distance $W_2^2$. Hence, denoting the minimizer by $\eta_t$, classical computations would show it coincides with McCann's interpolant.

     For $\eps>0$, the above decomposition then provides a generalization of the barycentric formulation to an entropic interpolant. We define again the curve $t\mapsto \eta_t$ which assigns to a given $t\in(0,1)$ the optimal measure $\eta_t$, the latter being unique by strict convexity of $\eta\mapsto\KL(\eta|\lambda)$. As already computed in \Cref{ex:computation_midpoint_Gaussian}, for $\mu=\delta_{x}$, $\nu=\delta_{y}$ this provides the curve $\eta_t=\frac{1}{(2\pi)^{d/2} \sigma_t^d}e^{-\nicefrac{|z-m_t|^2}{2\sigma_t^2}}$ with $m_t=(1-t)x+ty$ and $\sigma_t=\sqrt{\frac{\eps t(1-t)}{2}}$. These are the marginals of a Brownian bridge. More generally, we now draw the connection with Schrödinger bridges.
    
The \emph{Schr\"odinger bridge problem} is the stochastic counterpart of optimal transport \cite{Leonard2014}: given a (possibly unbounded) reference path measure $R$, one seeks the path measure $P$ minimizing $\KL(P|R)$ subject to $P_0=\mu$, $P_1=\nu$. Although other choices of $R$ with a drift are possible \cite{Gentil2017}, here we take as canonical process $R^\mu$ a Brownian motion on $\R^d$  initialized at $\mu$, with generator $\frac{\eps}{4}\Delta$. Denote by
\begin{equation}\label{eq:heat-kernel}
    p_t(x,y)\coloneqq(\pi\eps t)^{-d/2}\exp\left(-\frac{|x-y|^2}{t\eps}\right), \qquad t\in(0,1],
\end{equation}
the Markov transition kernel of this process, so that the Chapman-Kolmogorov equation reads, for all $x,y\in\R^d$ and $t\in(0,1)$,
\begin{equation}\label{eq:CK}
    \int_{\R^d} p_t(x,z)\,p_{1-t}(z,y)\,\de z = p_1(x,y).
\end{equation}

Since $\mu,\nu\in\calP_2(\R^d)$, the static problem $\OT^\eps_c(\mu,\nu)$ with $c(x,y)=|x-y|^2$ admits
a unique optimal plan $\pi^*$, of the form
$\pi^*(x,y)=e^{(u(x)+v(y)-|x-y|^2)/\eps}\de\mu(x)\de\nu(y)$ for optimal entropic potentials $u,v$
(defined up to an additive constant), by \eqref{eq:opti-pi-characterization} and
\eqref{eq:sinkhorn-fixed-point}. We do not need the classical theory of the Schr\"odinger problem
with the unbounded reference measure $\lambda\otimes\lambda$, which would additionally require
$\KL(\mu|\lambda),\KL(\nu|\lambda)<+\infty$ (see \cite[Section~2]{Leonard2014} and the remark
following \Cref{thm:entropic_BB}). Instead, following \cite[Definition~3.8, Prop.~3.9]{Gentil2017},
we define the entropic interpolant directly as the time-$t$ marginal of the mixture of Brownian
bridges over $\pi^*$, whose density with respect to $\lambda$ is

\begin{equation}\label{eq:rho_t_def}
    \rho^{SB}_t(z) \coloneqq \int_{\R^d\times\R^d} \frac{p_t(x,z)\,p_{1-t}(z,y)}{p_1(x,y)}\,\de\pi^*(x,y) \propto \int_{\R^d\times\R^d} e^{\frac{u(x)+v(y)-\psi_{t}(x,z)-\psi_{1-t}(y,z)}{\eps}} \de\mu(x)\de\nu(y),
\end{equation}
which integrates to $1$ by \eqref{eq:CK} applied with $\de\pi^*$ and where we used the expression of $p_1$ from \eqref{eq:heat-kernel} to simplify some terms. For now, we use \eqref{eq:rho_t_def} as a definition rather than going through path measures.

We now show that the entropic interpolant $\rho^{SB}_t$ coincides with the minimizer $\eta_t$ of \eqref{eq:eta_t_Gaussien}, giving a construction of the dynamical bridge that is  purely static, variational and barycentric.

\begin{proposition}[Equivalence of entropic interpolation and the Schr\"odinger bridge]\label{prop:equivalence_schrodinger}
    Let $\bbX=\R^d$ with $c(x,y)=|x-y|^2$, let $\mu,\nu\in\calP(\R^d)$ with finite second moment, $\eps>0$, and let $\lambda$ be the Lebesgue measure on $\R^d$. For $t\in(0,1)$ set $\psi_t(x,z)=|x-z|^2/t$ and let $\eta_t\in\calP(\R^d)$ be the unique minimizer of
    \begin{equation}\label{eq:functional_F}
        \calF_t(\eta) \coloneqq \OT^\eps_{\psi_t}(\mu,\eta)+\OT^\eps_{\psi_{1-t}}(\nu,\eta)+\eps \KL(\eta|\lambda).
    \end{equation}
    Then $\eta_t=\rho^{SB}_t$, the time-$t$ marginal of the mixture of Brownian bridges \eqref{eq:rho_t_def}. If moreover $\KL(\nu|\lambda)<+\infty$, this mixture is the Schr\"odinger bridge between $\mu$ and $\nu$ with reference $R^\mu$, see Step~3 of the proof of \Cref{lem:gaussian_entropy_action_limit}.
\end{proposition}
The functional \eqref{eq:functional_F} is understood with the convention $(+\infty)+(-\infty)=+\infty$. This
is harmless: if $\eta\notin\calP_2(\R^d)$ then, since $\mu\in\calP_2(\R^d)$, one has
$\OT^\eps_{\psi_t}(\mu,\eta)=+\infty$ and $\calF_t(\eta)=+\infty$; and for $\eta\in\calP_2(\R^d)$ with
second moment $m_2$ the Gaussian maximum entropy inequality gives
$\KL(\eta|\lambda)\ge-\frac d2\log\big(\frac{2\pi e\,m_2}{d}\big)>-\infty$, so that no term is $-\infty$.
\begin{proof} We are in a time-dependent variant of the setting of \Cref{cor:EOT_debias_Gaussian_like}, where the Chapman-Kolmogorov equation \eqref{eq:CK} is nothing else than the extension of \eqref{eq:condition-const} beyond midpoint settings. As we only focus on the minimizer, we can discard the translation terms of \eqref{eq:eta_t_Gaussien} and focus on $ \calF_t$ as in \eqref{eq:functional_F}.

As discussed around \eqref{eq:ceps_Euclidean}, nothing in the proof of \Cref{thm:EOT_debiased} is specific to the form $\psi+\psi$, so the result holds identically for $\psi_t+\psi_{1-t}$. From this proof and in particular \eqref{eq:gibbs_measure}, we know that the optimal plan satisfies $\gamma^{*,t}=\rho^{*,t}_{x,y}\otimes \pi^*$
where
\[
\rho^{*,t}_{x,y}(z)=\frac{1}{A^t_{x,y}} e^{\frac{-\psi_{t}(x,z)-\psi_{1-t}(y,z)}{\eps}} \lambda(z),
\]
where $A^t_{x,y}=\int_\calZ e^{\frac{-\psi_{t}(x,z)-\psi_{1-t}(y,z)}{\eps}}\de \lambda(z)$ is the normalizing constant and $\pi^*(x,y)=e^{\frac{u(x)+v(y)-|x-y|^2}{\eps}}\de\mu(x)\de\nu(y)$. We thus just have to show that the marginal in $z$ of $\gamma^{*,t}$ coincides with $\rho^{SB}_t$.

Since we only manipulate probabilities, all computations can be done up to irrelevant factors depending on time only. Also note that by \eqref{eq:parallelogram_law} and Gaussian integration, we have $e^{\frac{|x-y|^2}{\eps}}A^t_{x,y}=C^t_\eps>0$ for all $x,y$. Consequently we have
\begin{align*}
    \eta_t (z)=(\p_3)_\#\gamma^{*,t} (z)&=\int_{\R^d\times\R^d} \de \gamma^{*,t}(x,y,z)=\lambda(z)\int_{\R^d\times\R^d} \frac{1}{A^t_{x,y}} e^{\frac{u(x)+v(y)-|x-y|^2-\psi_{t}(x,z)-\psi_{1-t}(y,z)}{\eps}} \de\mu(x)\de\nu(y)\\
    &=\frac{\lambda(z)}{C^t_\eps}\int_{\R^d\times\R^d} e^{\frac{u(x)+v(y)-\psi_{t}(x,z)-\psi_{1-t}(y,z)}{\eps}} \de\mu(x)\de\nu(y) \stackrel{\eqref{eq:rho_t_def}}{\propto} \rho^{SB}_t(z).
\end{align*}
Hence the minimizer of the barycentric formulation $\eta_t$ coincides with the entropic interpolant $\rho^{SB}_t$.
\end{proof}

\begin{remark}[Entropic interpolation for negative definite costs]\label{rem:general_kernels} 
    The argument extends verbatim to any continuous negative definite cost $c(x,y)=|\phi(x)-\phi(y)|_{H_c}^2+s(x)+s(y)$ as in \Cref{sec:EOT_neg_kernel}, using the parallelogram identity in $H_c$,
    \[
    |\phi(x)-\phi(y)|_{H_c}^2 = \tfrac1t|\phi(x)|_{H_c}^2+\tfrac1{1-t}|\phi(y)|_{H_c}^2-\Big|\sqrt{\tfrac{1-t}t}\phi(x)+\sqrt{\tfrac{t}{1-t}}\phi(y)\Big|_{H_c}^2,
    \]
    in place of \eqref{eq:parallelogram_law}, with $\calZ=\R^I$, $\langle\phi(x),z\rangle_{H_c}$ the Gaussian series and $\lambda$ the product Gaussian measure as in the proof of \Cref{prop:Gaussian-factorisation-separable}. Using the Gaussian integral identity $\int_{\mathcal{Z}} \exp(\langle V, z \rangle_{H_c}) \de\lambda(z) = \exp(\frac{1}{2}|V|_{H_c}^2)$ for $V \in H_c$, we can exponentiate the cross-term:
    $$\exp\left( \frac{1}{\eps} \left| \sqrt{\frac{1-t}{t}}\phi(x) + \sqrt{\frac{t}{1-t}}\phi(y) \right|_{H_c}^2 \right) = \int_{\mathcal{Z}} \exp\left( \sqrt{\frac{2}{\eps}} \left\langle \sqrt{\frac{1-t}{t}}\phi(x) + \sqrt{\frac{t}{1-t}}\phi(y) , z \right\rangle_{H_c} \right) \de\lambda(z).$$
    This allows us to explicitly define time-dependent potentials for any $t \in (0,1)$:
	$$ \psi_t(x,z) \coloneqq - \sqrt{2\eps \frac{1-t}{t}} \langle \phi(x), z \rangle_{H_c} + \frac{1}{t}|\phi(x)|_{H_c}^2 + s(x). $$
	Because $\lambda$ is Gaussian, unlike \eqref{eq:eta_t_Gaussien}, here no compensating constant is needed since $\lambda$ is already normalized, yielding exactly $\OT^\eps_c(\mu,\nu) = \inf_{\eta\in\calP(\calZ)} \OT^\eps_{\psi_{t}}(\mu,\eta) + \OT^\eps_{\psi_{1-t}}(\nu,\eta) + \eps \KL(\eta|\lambda)$. While the Schr\"odinger bridge describes a flow of probability measures over $\bbX$, $\eta_t$ is this time a measure over the index space $\calZ=\R^I\supseteq H_c$ like in \Cref{sec:eps_infty}.
\end{remark}

\subsection{Entropic Benamou-Brenier formula}\label{sec:benamou-brenier}

A consequence of our inf-representation formulas is that, for costs $c$ that are negative definite, the notion of infinite divisibility for the associated kernel $k$ (see after \Cref{prop:def_rep_kernels}) alludes to the fact that there could be an underlying continuous process and an infinitesimal generator, hence one can investigate a limiting inf-representation.

One such case is the entropic Benamou-Brenier formula. Its latest PDE-based proof by \cite{garatti2026pdeapproachbenamoubrenierformula} uses the assumption of marginals with $C^2$-log-concave densities, with Hessians positively bounded from above and below. This Fokker-Planck/HJB duality approach complements the compactly-supported assumption of \cite{Gigli2019}, and mentions some mendable flaws in the probabilistic proof of \cite{Gentil2017} where both marginals are assumed to have finite entropy. We here take a different direction. Similarly to \cite{Gentil2017} we use Girsanov's theorem and path measures, but we do not use any specific property of entropic interpolations, nor consider the dual HJB or perform integration by parts. While we need a drift projection similar to \cite[Theorem 4.6]{Gyngy1986} and a superposition principle for Fokker--Planck equations \cite[Theorem 2.5]{Trevisan2016} to obtain the Eulerian expression, the connection between discrete and continuous-time interpolation is novel. The proof also sheds light on the interpretation of the extra term $\eps \KL(\eta|\lambda)$ in \eqref{eq:functional_F} and \eqref{eq:eta_t_Gaussien}: it is a way of stitching the boundary conditions together, which disappears in a telescopic sum.

\begin{theorem}[Entropic Benamou-Brenier formula]
\label{thm:entropic_BB}Let $c(x,y)=|x-y|^2$ on $\R^d$ with the Lebesgue measure $\lambda$. Assume that $\mu,\nu\in\mathcal{P}_2(\mathbb{R}^d)$ and $\KL(\nu | \lambda)<+\infty$. Then we have
\begin{equation}\label{eq:BB_entropique}
      \OT_c^\eps(\mu,\nu)
  =
  \inf_{(\rho,v)\in\calA_\eps(\mu,\nu)}
  \left\{
    \int_0^1\!\int_{\R^d}
    |v_t(x)|^2\,\de\rho_t(x)\,\de t
  \right\}
  -\eps \KL(\nu|\lambda)
  -\frac{d\eps}{2}\log(\pi\eps),
\end{equation}
where $\calA_\eps(\mu,\nu)$ is the set of pairs $(\rho,v)$ with
$\rho \in C([0,1];\mathcal{P}_2(\mathbb{R}^d))$ equipped with the $\infty$-norm over $W_2$, and $v \in L^2([0,1]\times\mathbb{R}^d;\rho_t\,\de x\,\de t)$,
satisfying
\[
  \partial_t\rho_t + \nabla\cdot(\rho_t v_t) = \tfrac{\eps}{4}\Delta\rho_t
  \quad \text{in } \mathcal{D}'((0,1)\times\mathbb{R}^d), \qquad
  \rho_0 = \mu, \quad \rho_1 = \nu.
\]
Moreover, the infimum is attained by the entropic interpolation and its associated velocity field.
\end{theorem}

The constant $-\frac{d\eps}{2}\log(\pi\eps)$ in \eqref{eq:BB_entropique} is due to $k=e^{-c/\eps}$ being unnormalized, as further highlighted in \eqref{eq:entropy_gaussian}.

\begin{remark}[Relation with other entropic Benamou-Brenier formulas of the Schrödinger bridge] Usually the Schrödinger bridge problem is expressed with reference measure $\lambda \otimes \lambda$, while our $ \OT_c^\eps(\mu,\nu)$ uses as reference $\mu \otimes \nu$. Whereas $ \OT_c^\eps(\mu,\nu)$ is symmetric, note that \eqref{eq:BB_entropique} is not, as the Fokker-Planck equation will smooth out the initial measure, hence $\nu\ll \lambda$. As seen from the proof, the key reference measure is actually $\mu \otimes \lambda$ which does not give a symmetric role between $\mu$ and $\nu$. Pre-existing entropic Benamou-Brenier formulas such as \cite[Theorem 3.6]{garatti2026pdeapproachbenamoubrenierformula} or \cite[Theorem 5.1]{Gentil2017} are obtained by a change of reference measure similar to \cite[Remark 4.5]{Nutz2021IntroductionTE} leading to two extra terms $\eps \KL(\mu|\lambda)+\eps \KL(\nu|\lambda)$, the second being absorbed in \Cref{eq:BB_entropique}.
    
\end{remark}

\begin{proof} The proof is divided into two steps. The first is a $(N+1)$-point decomposition of $\OT_c^\eps(\mu,\nu)$ instead of the barycentric formula \eqref{eq:eta_t_Gaussien}. Then we use \Cref{lem:gaussian_entropy_action_limit} below to take the limit $N\to \infty$ and recover the Fokker-Planck equation and $\int_0^1\!\int_{\R^d}
        |v_t(x)|^2\,\de\rho_t(x)$ via Girsanov's theorem and path measures.

For any $\mu,\nu \in \calP_2(\R^d)$ and $t \in (0,1)$, we recall \eqref{eq:eta_t_Gaussien}:
\begin{equation}\label{eq:two-point}
  \OT_c^\eps(\mu,\nu)
  = \inf_{\eta \in \calP(\R^d)}
    \left\{
      \frac{1}{t}\,\OT^{\eps t}_c(\mu,\eta)
      + \frac{1}{1-t}\,\OT^{\eps(1-t)}_c(\eta,\nu)
      + \eps \KL(\eta|\lambda)
    \right\}
  + \frac{d\eps}{2}\log\bigl(t(1-t)\pi\eps\bigr).
\end{equation}
Let $\mathcal{T} = \{0 = t_0 < t_1 < \dots < t_N = 1\}$ be an arbitrary partition of $[0,1]$ with increments $h_i = t_i - t_{i-1}$. Splitting successively at intermediate times yields:
\begin{equation}\label{eq:dp_discrete_corrected}
  \OT_c^\eps(\mu,\nu)
  =
  \inf_{\eta_{t_0}=\mu,\ \eta_{t_1},\ \dots \eta_{t_N}=\nu}
  \left\{
    \sum_{i=1}^N\frac{1}{h_i}
      \OT^{\eps h_i}_c
      (\eta_{t_{i-1}},\eta_{t_i})
    +\eps\sum_{i=1}^{N-1}
      \KL(\eta_{t_i}|\lambda)
  \right\}
  + \frac{d\eps}{2}
  \left[
    (N-1)\log(\pi\eps)
    +\sum_{i=1}^N\log h_i
  \right].
\end{equation}
For $\mu',\nu'\in\calP_2(\R^d)$, by definition:
\[
  \frac{1}{h_i}\OT^{\eps h_i}_c(\mu',\nu')
  =
  \inf_{\gamma\in\Pi(\mu',\nu')}
  \left\{
    \frac{1}{h_i}\int_{\R^d\times\R^d}
      |x-y|^2\,\de\gamma(x,y)
    +\eps\,
      \KL(\gamma|\mu'\otimes\nu')
  \right\}.
\]
Hence, using \eqref{eq:dp_discrete_corrected}, we may optimize simultaneously over the intermediate marginals $(\eta_{t_i})_{i=1}^{N-1}$ and over couplings $\gamma_i\in\Pi(\eta_{t_{i-1}},\eta_{t_i})$.

The objective in \eqref{eq:dp_discrete_corrected} is understood with the convention
	$(+\infty)+(-\infty)=+\infty$. 
    We may
	therefore restrict the infimum to $\eta_{t_i}\in\calP_2(\R^d)$ with
	$\eta_{t_i}\ll\lambda$, since otherwise $\KL(\eta_{t_i}|\lambda)=+\infty$ while every
	other term is $>-\infty$. For such families, writing $m_2^{(i)}=\int_{\R^d}|x|^2\de\eta_{t_i}<+\infty$, we have $\KL(\eta_{t_i}|\lambda)\ \ge\ -\frac d2\log\!\Big(\frac{2\pi e\,m_2^{(i)}}{d}\Big)\ >\ -\infty$, so no term of the objective is $-\infty$, the sum is unambiguously defined, and the chain rule
\eqref{eq:chain-KL} may be applied to give
\[
\KL(\gamma_i| \eta_{t_{i-1}}\otimes\eta_{t_i}) +\KL(\eta_{t_i}|\lambda) = \KL(\gamma_i| \eta_{t_{i-1}}\otimes\lambda).
\]
Hence in \Cref{eq:dp_discrete_corrected}, the terminal entropy is the only one not cancelled:
\begin{multline}\label{eq:chain_rule}
 \sum_{i=1}^N
 \left[
   \frac{1}{h_i}\int_{\R^d\times\R^d} |x-y|^2\,\de\gamma_i(x,y)
   +\eps \KL(\gamma_i|\eta_{t_{i-1}}\otimes\eta_{t_i})
 \right]
 +\eps\sum_{i=1}^{N-1}\KL(\eta_{t_i}|\lambda) \\
 =
 \sum_{i=1}^N
 \left[
   \frac{1}{h_i}\int_{\R^d\times\R^d} |x-y|^2\,\de\gamma_i(x,y)
   +\eps \KL(\gamma_i|\eta_{t_{i-1}}\otimes\lambda)
 \right]
 -\eps \KL(\nu|\lambda).
\end{multline}
For every $\gamma\in\Pi(\mu,\nu)$ and every $h>0$, setting $p_h(x,y)
  =
  (\pi\eps h)^{-d/2}
  \exp\!\left(-\frac{|x-y|^2}{\eps h}\right)$ as in \eqref{eq:heat-kernel}, gives
\begin{equation}\label{eq:entropy_gaussian}
  \frac{1}{h}\int_{\R^d\times\R^d} |x-y|^2\,\de\gamma(x,y) +\eps \KL(\gamma|\mu\otimes\lambda)
  =
  \eps \KL\!\left( \gamma | p_h \mu \otimes \lambda \right)
  -\frac{d\eps}{2}\log(\pi\eps h).
\end{equation}
Combining \eqref{eq:dp_discrete_corrected}, \eqref{eq:chain_rule}, and \eqref{eq:entropy_gaussian}, we obtain
\begin{align}\label{eq:fully_discrete_action}
  \OT_c^\eps(\mu,\nu)
  &=-\eps \KL(\nu|\lambda)  -\frac{d\eps}{2}\log(\pi\eps) +\inf_{\substack{\eta_{t_0}=\mu,\ \eta_{t_1},\ \dots \eta_{t_N}=\nu\\
                  \gamma_i\in\Pi(\eta_{t_{i-1}},\eta_{t_i})}}
  \eps\sum_{i=1}^N
  \KL\!\left(
    \gamma_i|
    p_{h_i} \eta_{t_{i-1}} \otimes \lambda
  \right).
\end{align}
Finally, choose partitions $\mathcal{T}_n$ with
$|\mathcal{T}_n|\to 0$ and apply \Cref{lem:gaussian_entropy_action_limit} below to \eqref{eq:fully_discrete_action}. This yields exactly
\[
  \OT_c^\eps(\mu,\nu)
  =
  \inf_{(\rho,v)\in\calA_\eps(\mu,\nu)}
  \int_0^1\!\int_{\R^d}
  |v_t(x)|^2\,\de\rho_t(x)\,\de t
  -\eps \KL(\nu|\lambda)
  -\frac{d\eps}{2}\log(\pi\eps).
\]
\end{proof}

\begin{lemma}[Discrete Gaussian entropy--action convergence]
\label{lem:gaussian_entropy_action_limit}
Let $\mu,\nu\in\calP_2(\R^d)$ with $\KL(\nu|\lambda) < +\infty$ and let $\mathcal{T}_n$ be partitions of $[0,1]$ with $|\mathcal{T}_n|\to 0$. For a partition $\mathcal{T}=\{0=t_0<\cdots<t_N=1\}$, define
\[
  \calE_{\mathcal{T}}(\mu,\nu)
  \coloneqq
  \inf_{\substack{\eta_{t_0}=\mu,\ \eta_{t_1},\ \dots \eta_{t_N}=\nu\\
                  \gamma_i\in\Pi(\eta_{t_{i-1}},\eta_{t_i})}}
  \eps\sum_{i=1}^N
  \KL\!\left(
    \gamma_i|    p_{h_i} \eta_{t_{i-1}} \otimes \lambda
  \right),
\]
where $h_i \coloneqq t_i - t_{i-1}$ and $p_h(x,y)
  \coloneqq
  (\pi\eps h)^{-d/2}
  \exp\!\left(-\frac{|x-y|^2}{\eps h}\right)$. Then
\[
  \lim_{n\to\infty}
  \calE_{\mathcal{T}_n}(\mu,\nu)
  =
  \inf_{(\rho,v)\in\calA_\eps(\mu,\nu)}
  \int_0^1\!\int_{\R^d}
  |v_t(x)|^2\,\de\rho_t(x)\,\de t,
\]
and $\calE_{\mathcal{T}}(\mu,\nu)$ actually does not depend on $\mathcal{T}$.
\end{lemma}
Since the static entropic optimal transport problem only depends on the dual variables for the marginal constraints at $\mu$ and $\nu$, it is not surprising that the dynamic formulation does not depend on the partition.
\begin{proof}
Let $\Omega = C([0,1], \R^d)$ equipped with the uniform norm. Let $R^\mu \in \calP(\Omega)$ be the reference path measure corresponding to the scaled Brownian motion initialized at $\mu$:
\[
  \de X_t = \sqrt{\frac{\eps}{2}} \, \de W_t, \quad X_0 \sim \mu.
\]
Note that $R^\mu$ is a probability measure as it is initialized at $\mu$, so we do not have the issue of unbounded path measures of \cite[Section 2]{Leonard2014}, avoidable all the more as we do not use reversibility arguments. Set $\calP(\Omega; \mu, \nu) \coloneqq \{ Q \in \calP(\Omega) \mid (X_0)_\# Q = \mu, (X_1)_\# Q = \nu \}$ and, for any partition $\mathcal{T} = \{0 = t_0 < t_1 < \dots < t_N = 1\}$, let $R^\mu_{\mathcal{T}} = (X_{t_0}, X_{t_1}, \dots, X_{t_N})_\# R^\mu \in \calP((\R^d)^{N+1})$ be the finite-dimensional distribution of $R^\mu$.

\paragraph{Step 1: Reference measure and discrete KL equivalence.}

The transition kernel of $X_t$ over a time step $h_i = t_i - t_{i-1}$ is precisely $p_{h_i}(x,y) \, \de y$. Its density is given by
\[
  \de R^\mu_{\mathcal{T}}(x_0, x_1, \dots, x_N) = \mu(\de x_0) \prod_{i=1}^N p_{h_i}(x_{i-1}, x_i) \, \de x_i.
\]
Now, let $Q \in \calP(\Omega; \mu, \nu)$, and let $Q_{\mathcal{T}} = (X_{t_0}, \dots, X_{t_N})_\# Q$. By the chain rule for relative entropy,
\[
  \KL(Q_{\mathcal{T}} | R^\mu_{\mathcal{T}}) = \KL((X_{t_0})_\#Q | \mu) + \sum_{i=1}^N \int_{(\R^d)^i} \KL\!\left( Q(X_{t_i} \in \cdot \mid x_0, \dots, x_{i-1}) \,\middle|\, p_{h_i}(x_{i-1}, \cdot)\lambda\right) \de Q_{\mathcal{T}_{0:i-1}},
\]
and $\KL((X_{t_0})_\#Q|\mu)=0$ since $(X_{t_0})_\#Q=\mu$. Writing $\eta_{t_i}=(X_{t_i})_\#Q$ and
$\gamma_i=(X_{t_{i-1}},X_{t_i})_\#Q$, the joint convexity of the relative entropy and Jensen's inequality
applied to the disintegration of $Q_{\mathcal{T}}$ with respect to $(x_0,\dots,x_{i-2})$ give
\[
  \int_{(\R^d)^i} \KL\!\left( Q(X_{t_i} \in \cdot \mid x_0, \dots, x_{i-1}) \,\middle|\, p_{h_i}(x_{i-1}, \cdot)\lambda\right) \de Q_{\mathcal{T}_{0:i-1}} \ge \KL\!\left( \gamma_i  \,\middle|\, p_{h_i} \eta_{t_{i-1}} \otimes \lambda\right),
\]
with equality for every $i$ if the skeleton $Q_{\mathcal T}$ is a Markov chain, i.e.\
$Q(X_{t_i}\in\cdot\mid x_0,\dots,x_{i-1})=Q(X_{t_i}\in\cdot\mid x_{i-1})$ for $Q_{\mathcal T}$-a.e.\
$(x_0,\dots,x_{i-1})$ and every $i$. Consequently, taking the infimum over all
$Q\in\calP(\Omega;\mu,\nu)$, whose skeleton marginals $(\eta_{t_i})_i$ and two-time marginals
$(\gamma_i)_i$ are admissible for $\calE_{\mathcal T}(\mu,\nu)$, yields
\[
  \calE_{\mathcal{T}}(\mu,\nu) \le \inf_{Q \in \calP(\Omega; \mu, \nu)} \eps\KL(Q_{\mathcal{T}} | R^\mu_{\mathcal{T}}),
\]
with equality as soon as we exhibit a $Q$ with Markov skeleton which is optimal for
$\calE_{\mathcal{T}}$, as done in Step 3.

\paragraph{Step 2: Girsanov transformation and Superposition principle.}
Let $Q \in \calP(\Omega; \mu, \nu)$ with $\KL(Q|R^\mu) < \infty$. Recall that under the reference measure $R^\mu$, the canonical process satisfies $\de X_t = \sqrt{\eps/2} \, \de W_t$. By Girsanov's theorem as used in \cite[eq.(3.3)]{Follmer86},
under $Q$, the process $X_t$ admits a drift $\beta_t(X_{0:t})$ such that the relative entropy is given by the kinetic action:
\[
	\eps \KL(Q | R^\mu) = \frac{\eps}{2} \mathbb{E}_Q \int_0^1 \left| \left(\sqrt{\tfrac{\eps}{2}} I\right)^{-1} \beta_t(X_{0:t}) \right|^2 \de t = \int_0^1 \mathbb{E}_Q\!\left[|\beta_t(X_{0:t})|^2\right] \de t.
\]
Equivalently, $X_t=X_0+\int_0^t\beta_s\,\de s+\sqrt{\eps/2}\,W^Q_t$ for a $Q$-Brownian motion $W^Q$. Set $\rho_t \coloneqq (X_t)_\# Q$. We apply a Markovian drift projection as in \cite[Theorem 4.6]{Gyngy1986}, which is however stated for bounded coefficients, whereas $\beta$ is only square integrable; since we only need the time-marginals, we argue directly. By the Doob--Dynkin lemma, there is a Borel function $v:[0,1]\times\R^d\to\R^d$ with $v_t(X_t)=\mathbb E[\beta_t\mid t,X_t]$, the conditional expectation being taken on $([0,1]\times\Omega,\de t\otimes Q)$ given $(t,\omega)\mapsto(t,X_t(\omega))$, i.e.\ $v_t(x)=\mathbb{E}_Q[\beta_t \mid X_t = x]$. By Jensen's inequality, this conditioning lowers the kinetic action:
\[
	\eps\KL(Q|R^\mu) \ge \int_0^1 \int_{\R^d} |v_t(x)|^2 \,\de\rho_t(x)\,\de t.
\]
For $\varphi\in C^\infty_c((0,1)\times\R^d)$, It\^o's formula gives
\[
0=\varphi(1,X_1)-\varphi(0,X_0)=\int_0^1\Big(\partial_t\varphi+\beta_t\cdot\nabla\varphi+\tfrac\eps4\Delta\varphi\Big)(t,X_t)\,\de t+\sqrt{\tfrac\eps2}\int_0^1\nabla\varphi(t,X_t)\cdot\de W^Q_t .
\]
The stochastic integral has a bounded integrand, hence zero expectation, and $\mathbb E_Q\int_0^1|\beta_t|\,|\nabla\varphi(t,X_t)|\,\de t<+\infty$, so that taking expectations and conditioning on $(t,X_t)$ yields
\[
\int_0^1\int_{\R^d}\Big(\partial_t\varphi+v_t\cdot\nabla\varphi+\tfrac\eps4\Delta\varphi\Big)\,\de\rho_t\,\de t=0,
\]
i.e.\ the Fokker--Planck equation $\partial_t \rho_t + \nabla \cdot (\rho_t v_t) = \tfrac{\eps}{4}\Delta \rho_t$ holds in $\mathcal D'((0,1)\times\R^d)$, with $\rho_0=\mu$ and $\rho_1=\nu$. Finally, for $0\le s\le t\le 1$, the Cauchy--Schwarz inequality gives
\[
W_2^2(\rho_s,\rho_t)\le\mathbb E_Q|X_t-X_s|^2\le2(t-s)\int_s^t\mathbb E_Q|\beta_r|^2\,\de r+\eps d\,(t-s),
\]
so that $\rho_t\in\calP_2(\R^d)$ for all $t$, as $\mu\in\calP_2(\R^d)$, and $\rho\in C([0,1];\calP_2(\R^d))$.
Thus, $(\rho, v) \in \calA_\eps(\mu,\nu)$ and $\inf_{Q \in \calP(\Omega;\mu,\nu)} \eps\KL(Q|R^\mu) \ge \inf_{(\rho,v)\in\calA_\eps(\mu,\nu)} \int_0^1\!\int_{\R^d} |v_t(x)|^2\,\de\rho_t(x)\,\de t$.

Conversely, let $(\rho, v) \in \calA_\eps(\mu,\nu)$ have finite action. By Trevisan's superposition principle \cite[Theorem 2.5]{Trevisan2016}, there is $Q\in\calP(\Omega;\mu,\nu)$ with time-marginals $\rho_t$ under which $X_t=X_0+\int_0^tv_s(X_s)\,\de s+\sqrt{\eps/2}\,W^Q_t$. As Novikov's criterion need not hold for $v$, let $Q_n$ be the law of the process whose drift is stopped at $\tau_n\coloneqq\inf\{t:\int_0^t|v_s(X_s)|^2\de s\ge n\}\wedge1$. Its drift has energy at most $n$, so Girsanov's theorem gives $\eps\KL(Q_n|R^\mu)=\mathbb E_Q\int_0^{\tau_n}|v_s(X_s)|^2\de s$. Both processes coincide on $\{\tau_n=1\}$, hence $\|Q_n-Q\|_{TV}\le2Q(\tau_n<1)\to0$, and the lower semicontinuity of $\KL(\cdot|R^\mu)$ yields
\[
	\eps\KL(Q|R^\mu) \le \int_0^1 \int_{\R^d} |v_t(x)|^2 \,\de\rho_t(x)\,\de t .
\]
Taking the infimum over all $(\rho,v) \in \calA_\eps(\mu,\nu)$ yields the reverse inequality, and hence $\inf_{Q \in \calP(\Omega;\mu,\nu)} \eps\KL(Q|R^\mu) =\inf_{(\rho,v)\in\calA_\eps(\mu,\nu)} \int_0^1\!\int_{\R^d} |v_t(x)|^2\,\de\rho_t(x)\,\de t$. 	We record for later use that, by monotonicity of the relative entropy under push-forward,
	$\KL(Q_{\mathcal{T}}|R^\mu_{\mathcal{T}})\le\KL(Q|R^\mu)$ for every partition $\mathcal{T}$ and every $Q\in\calP(\Omega;\mu,\nu)$,
	so that $\calE_{\mathcal{T}}(\mu,\nu)\le\inf_Q\eps\KL(Q|R^\mu)$ by Step 1.

\paragraph{Step 3: the explicit bridge mixture is optimal.}
Similarly to the discussion around \eqref{eq:rho_t_def}, define $Q^* \in \calP(\Omega;\mu,\nu)$ as the
mixture of Brownian bridges
\begin{equation}\label{eq:def_pstar_brownian_bridge}
    Q^* \coloneqq \int_{\R^d\times\R^d} R^{x,y}_{0,1} \, \de\pi^*(x,y),
\end{equation}
where $R^{x,y}_{0,1} \in \calP(\Omega)$ is the Brownian bridge from $x$ to $y$ on $[0,1]$ with
generator $\frac{\eps}{4}\Delta$, whose time-$t$ marginal has density
$z \mapsto \frac{p_t(x,z) p_{1-t}(z,y)}{p_1(x,y)}$. Integrating over $\pi^*$ gives
$(X_t)_\# Q^* = \rho^{SB}_t$ as defined in \eqref{eq:rho_t_def}; in particular
$\rho^{SB}_t\in\calP_2(\R^d)$, being a mixture over $\pi^*\in\Pi(\mu,\nu)$ of Gaussian measures with
means on the segment $[x,y]$ and variance $\frac{\eps t(1-t)}{2}I_d$, and $\mu,\nu$ having finite
second moments. For any partition, substituting
$p_1(x_0,x_N)=(\pi\eps)^{-d/2}e^{-|x_0-x_N|^2/\eps}$ into
$\de Q^*_{\mathcal{T}}=\big(\prod_{i}p_{h_i}(x_{i-1},x_i)\big)\frac{\de\pi^*(x_0,x_N)}{p_1(x_0,x_N)}\de x_1\cdots\de x_{N-1}$
cancels the quadratic cost in the density of $\pi^*$ and factorizes it as
\[
\de Q^*_{\mathcal{T}}(x_0, \dots, x_N) = (\pi \eps)^{d/2}\left[ e^{\frac{u(x_0)}{\eps}}\de\mu(x_0) \right] \left( \prod_{i=1}^N p_{h_i}(x_{i-1}, x_i) \de x_i \right) \left[  e^{\frac{v(x_N)}{\eps}} \frac{\de\nu}{\de\lambda}(x_N) \right],
\]
the law of a Markov chain with initial law $\mu$; this holding for every $\mathcal{T}$, $Q^*$ is
itself Markov.

We check by induction on $\#\mathcal{T}$ that
$\mathbf a^*_{\mathcal T}\coloneqq\big((\rho^{SB}_{t_i})_{i=0}^N,\big((X_{t_{i-1}},X_{t_i})_\#Q^*\big)_{i=1}^N\big)$
is optimal for $\calE_{\mathcal{T}}(\mu,\nu)$, recalling that, by \eqref{eq:chain_rule} and
\eqref{eq:entropy_gaussian}, the minimizers of $\calE_{\mathcal T}$ and of
\eqref{eq:dp_discrete_corrected} are in bijection. For $\mathcal{T}=\{0,1\}$,
$\calE_{\mathcal T}(\mu,\nu)=\inf_{\gamma\in\Pi(\mu,\nu)}\eps\KL(\gamma|p_1\mu\otimes\lambda)$ differs by a
constant from the objective of $\OT^\eps_c(\mu,\nu)$ by \eqref{eq:entropy_gaussian} with $h=1$, hence is
attained at $\pi^*$. If $\mathcal{T}'=\mathcal{T}\cup\{t\}$ with $t\in(t_{i-1},t_i)$, only the $i$-th
term of \eqref{eq:dp_discrete_corrected} is split. By the Markov property of $Q^*$, its restriction to
$[t_{i-1},t_i]$ is, after the time change $r\mapsto t_{i-1}+rh_i$, the mixture
\eqref{eq:def_pstar_brownian_bridge} over $(X_{t_{i-1}},X_{t_i})_\#Q^*$, which by the above
factorization has a density of the form \eqref{eq:opti-pi-characterization} with $\eps h_i$ in place of
$\eps$, hence is the optimal plan of $\OT^{\eps h_i}_c(\rho^{SB}_{t_{i-1}},\rho^{SB}_{t_i})$
\cite[Theorem 4.2(b) and Remark 4.12]{Nutz2021IntroductionTE}, the value being finite since both
marginals lie in $\calP_2(\R^d)$. \Cref{prop:equivalence_schrodinger}, with $\eps$ replaced by $\eps h_i$,
then shows that $\rho^{SB}_t$ and the two-time marginals of $Q^*$ are optimal in the two-point
splitting \eqref{eq:two-point} of this term, so that $\mathbf a^*_{\mathcal{T}'}$ is optimal for
$\calE_{\mathcal{T}'}$.

Since $Q^*_{\mathcal{T}}$ is Markov with initial law $\mu$, Step~1 holds with equality for $Q^*$, so
that $\eps\KL(Q^*_{\mathcal{T}}|R^\mu_{\mathcal{T}})=\calE_{\mathcal{T}}(\mu,\nu)$ for every $\mathcal{T}$; in particular
$\mathcal{T}\mapsto \eps\KL(Q^*_{\mathcal{T}}|R^\mu_{\mathcal{T}})$ is constant by \eqref{eq:fully_discrete_action}, and we may replace $(\mathcal{T}_n)$ by the
refining sequence $\mathcal{T}_1\subseteq\mathcal{T}_1\cup\mathcal{T}_2\subseteq\cdots$, whose union
is dense in $[0,1]$, so that the corresponding $\sigma$-algebras generate $\calB(\Omega)$ by
continuity of the paths. The Gelfand--Yaglom--Perez theorem (see \cite[Theorem 4.5]{Polyanskiy2024}
for an easy-to-access formulation), applicable since $R^\mu$ is a probability measure, then gives
$\lim_{n}\eps\KL(Q^*_{\mathcal{T}_n}|R^\mu_{\mathcal{T}_n})=\eps\KL(Q^*|R^\mu)$, whence
$\lim_{n}\calE_{\mathcal{T}_n}(\mu,\nu)=\eps\KL(Q^*|R^\mu)$.

Finally $\eps\KL(Q^*|R^\mu)=\inf_Q\eps\KL(Q|R^\mu)$: the inequality $\le$ follows from
$\eps\KL(Q^*|R^\mu)=\lim_n\calE_{\mathcal{T}_n}(\mu,\nu)\le\inf_Q\eps\KL(Q|R^\mu)$ recorded at the end of Step 2. Hence, denoting by $(\rho^{SB},v^*)$ the drift projection of $Q^*$,
\begin{align*}
    \lim_{n\to\infty} \calE_{\mathcal{T}_n}(\mu,\nu)=\eps\KL(Q^*|R^\mu)\ \ge\ \int_0^1\!\!\int_{\R^d}|v^*_t|^2\de\rho^{SB}_t\de t
  \ &\ge \inf_{(\rho,v)\in\calA_\eps(\mu,\nu)} \int_0^1\!\!\int_{\R^d} |v_t|^2\,\de\rho_t\,\de t\\
  &=\inf_{Q}\eps\KL(Q|R^\mu)=\eps\KL(Q^*|R^\mu),
\end{align*}
so that all these quantities coincide and the infimum is attained at $(\rho^{SB},v^*)$.

\end{proof}

For completeness and to stress the proof analogy, we now reprove the classical Benamou-Brenier formula through path-measures as an alternative to the classical proofs of e.g.\ \cite[Lecture 17]{Ambrosio2024}. Note that, in both the entropic and classical case, partitions are not necessary, one could deal directly with the whole interval $[0,1]$. Indeed, by guessing the optimal path measure, one can derive directly the expression $W_2^2(\mu,\nu)$ or $\eps\KL(Q^*|R^\mu)$. Nevertheless iterated inf-representations allow one to guess the continuous-time formulation.

\begin{proposition}[Benamou-Brenier formula for $W_2^2$ via path measures]
\label{prop:BB_W2}
Assume that $\mu,\nu\in\mathcal{P}_2(\mathbb{R}^d)$. Then we have
\begin{equation}\label{eq:BB_W2}
	W_2^2(\mu,\nu)
	=
	\inf_{(\rho,v)\in\calA_0(\mu,\nu)}
	\int_0^1\!\int_{\R^d}
	|v_t(x)|^2\,\de\rho_t(x)\,\de t,
\end{equation}
where $\calA_0(\mu,\nu)$ is the set of pairs $(\rho,v)$ with
$\rho \in C([0,1];\mathcal{P}_2(\mathbb{R}^d))$,
$v \in L^2([0,1]\times\mathbb{R}^d;\de t\,\rho_t(\de x))$,
satisfying the continuity equation in the weak sense over $\mathcal{D}'((0,1)\times\mathbb{R}^d)$.
\[
\partial_t\rho_t + \nabla\cdot(\rho_t v_t) = 0
, \qquad
\rho_0 = \mu, \quad \rho_1 = \nu.%
\]
\end{proposition}

\begin{proof}
Let $\Omega = AC([0,1]; \R^d)$ be the space of absolutely continuous paths, and denote by $\calP(\Omega;\mu,\nu)$ the set of path measures $P \in \calP(\Omega)$ such that $(e_0)_\# P = \mu$ and $(e_1)_\# P = \nu$, where $e_t(\omega) = \omega(t)$ is the evaluation map.

\paragraph{Step 1: Multi-point discrete splitting.}
For any partition $\mathcal{T} = \{0 = t_0 < t_1 < \dots < t_N = 1\}$ with increments $h_i \coloneqq t_i - t_{i-1}$, the iterated inf-representation yields:
\begin{equation}\label{eq:W2_discrete_split}
	W_2^2(\mu,\nu) = \inf_{\eta_{t_0}=\mu,\ \eta_{t_1},\ \dots \eta_{t_N}=\nu} \sum_{i=1}^N \frac{1}{h_i} W_2^2(\eta_{t_{i-1}}, \eta_{t_i}).
\end{equation}
This can be rewritten as an infimum over discrete joint distributions $P_{\mathcal{T}} \in \Pi(\mu, \eta_{t_1}, \dots, \eta_{t_{N-1}}, \nu)$ on $(\R^d)^{N+1}$:
\[
\calE_{\mathcal{T}}(\mu,\nu) \coloneqq \inf_{P_{\mathcal{T}}} \sum_{i=1}^N \frac{1}{h_i} \int_{(\R^d)^{N+1}} |x_{i-1} - x_i|^2 \,\de P_{\mathcal{T}}(x_0, \dots, x_N).
\]
From \eqref{eq:W2_discrete_split}, $\calE_{\mathcal{T}}(\mu,\nu) = W_2^2(\mu,\nu)$ holds, independent of the chosen partition $\mathcal{T}$. 

\paragraph{Step 2: Path-space action and Superposition principle.}
For any $P \in \calP(\Omega;\mu,\nu)$, as absolutely continuous curves $\omega(\cdot)$ have a $L^1$-metric derivative $|\dot{\omega}(\cdot)|$, we define the continuous kinetic energy (or action) as
\[
\calE(P) \coloneqq \int_{\Omega} \left( \int_0^1 |\dot{\omega}(t)|^2 \,\de t \right) \de P(\omega)
\]
for $\dot{\omega}\in L^2$ and $+\infty$ otherwise.  By Cauchy-Schwarz, for any curve $\omega \in \Omega$, we have $\frac{1}{h_i}|\omega(t_i) - \omega(t_{i-1})|^2 \le \int_{t_{i-1}}^{t_i} |\dot{\omega}(t)|^2 \,\de t$. Integrating this bound over $P$ guarantees that $\calE_{\mathcal{T}}(\mu,\nu) \le \calE(P)$ for all path measures $P$. 

To pass to the Eulerian formulation, we rely on the Ambrosio-Gigli-Savaré framework. If $\calE(P) < \infty$, the time-marginals $\rho_t \coloneqq (e_t)_\# P$ form an absolutely continuous curve in $(\calP_2(\R^d), W_2)$. By \cite[Theorem 8.3.1]{ags08}, there exists a minimal measurable velocity field $v_t \in L^2(\rho_t)$ solving the continuity equation $\partial_t \rho_t + \nabla \cdot (\rho_t v_t) = 0$ in $\mathcal{D}'((0,1)\times\R^d)$ such that $\|v_t\|_{L^2(\rho_t)}^2 \le \int_{\Omega} |\dot{\omega}(t)|^2 \,\de P(\omega)$ for a.e.\ $t\in(0,1)$.. Integrating over time yields:$$\calE(P) \ge \int_0^1 \int_{\R^d} |v_t(x)|^2 \,\de\rho_t(x)\,\de t,$$which implies $\inf_{P} \calE(P) \ge \inf_{(\rho,v)} \int_0^1\int_{\R^d} |v_t|^2 \,\de\rho_t\,\de t$.

Conversely, by the Superposition Principle \cite[Theorem 8.2.1]{ags08}, any Eulerian solution $(\rho, v) \in \calA_0(\mu,\nu)$ with finite energy lifts to a path measure $P \in \calP(\Omega;\mu,\nu)$ such that $|\dot{\omega}(t)| = |v_t(\omega(t))|$ for $P$-a.e.\ $\omega$. Thus, the action of $P$ matches the Eulerian energy. Taking the infimum over all $(\rho,v)\in\calA_0(\mu,\nu)$, we conclude that $\inf_{\calP(\Omega;\mu,\nu)}\calE(P)  = \inf_{(\rho,v)\in\calA_0(\mu,\nu)} \int_0^1\int_{\R^d} |v_t|^2 \,\de\rho_t\,\de t$.

\textbf{Step 3: Equality via the McCann's/geodesic interpolant.} Since $\R^d$ is geodesic, setting $ P^* \coloneqq (\Phi)_\# \gamma^*$ where $\Phi(x,y)(t) = (1-t)x + ty$ and $\gamma^* \in \Pi(\mu,\nu)$ is the optimal coupling for $W_2^2(\mu,\nu)$, we have that $P^* \in \calP(\Omega;\mu,\nu)$ is concentrated on constant-speed geodesics. For such paths, the discrete increments exactly match the continuous energy: $\frac{1}{h_i}|\omega(t_i) - \omega(t_{i-1})|^2 = \int_{t_{i-1}}^{t_i} |\dot{\omega}(t)|^2 \,\de t$. Therefore, for every partition $\mathcal T$, Steps 1 and 2 give
\begin{align*}
    W_2^2(\mu,\nu)=\calE_{\mathcal T}(\mu,\nu)\ \le\ \inf_{P\in\calP(\Omega;\mu,\nu)}\calE(P)\ \le\ \calE(P^*)
&=\sum_{i=1}^N\frac1{h_i}\int_\Omega|\omega(t_i)-\omega(t_{i-1})|^2\,\de P^*(\omega)\\
&=\int_{\R^d\times\R^d}|x-y|^2\,\de\gamma^*(x,y)=W_2^2(\mu,\nu),
\end{align*}
so that all these quantities coincide, and in particular, along any sequence of partitions $\mathcal T_n$ with $|\mathcal T_n|\to0$,
\[
\lim_{n\to\infty} \calE_{\mathcal{T}_n}(\mu,\nu) = \inf_{P \in \calP(\Omega; \mu, \nu)} \calE(P) = \calE(P^*).
\]
Combined with Step 2, this gives \eqref{eq:BB_W2}.

\end{proof}

The parallelism between the entropic and standard proofs for Benamou-Brenier's formula highlights how stochastic and metric tools perform identical roles at every step of the derivation. This complements the known analogies of \cite{Gentil2017} between OT and entropic OT: namely the correspondence between displacement and entropic interpolations, between the Benamou--Brenier formula and its entropic counterpart, and between the Kantorovich and Schr\"odinger dual formulations.

At the discrete level, both formulations rely on a multi-point dynamic programming splitting, obtained via inf-representation in our case, that is invariant to the choice of partition. The entropic approach iterates a barycentric/Schrödinger decomposition, while the unregularized metric approach simply iterates the triangle and parallelogram inequalities.

When considering the continuous limit, the underlying path spaces reflect the nature of the transport. The entropic formulation operates over continuous, non-differentiable paths (diffusions) and measures the action via relative entropy $\eps \KL(Q|P)$. The discrete-to-continuous limit is then obtained by the Gelfand-Yaglom-Perez theorem. In contrast, the standard metric approach restricts to absolutely continuous curves, measures the action via the classical metric kinetic energy $\calE$, and obtains its limit from the explicit geodesic interpolant.

Finally, both proofs require projecting a path-measure framework onto an Eulerian vector field. In the entropic setting, Föllmer's drift projection extracts a minimal Markovian velocity field that satisfies Fokker-Planck. In exact analogy, the classical metric setting utilizes an Eulerian vector field associated with the continuity equation with energy smaller than the metric derivative. In both cases, a superposition principle is used for the converse. 

	\section*{Acknowledgements}
	
	GT and VE acknowledge the financial support of European Research Council (ERC) under the European Union’s
	Horizon 2020 Research and Innovation Programme – Grant Agreement n°101077204 HighLEAP. PCA benefited from Hi! PARIS and State funding managed by the French National Research Agency (ANR) under the France 2030 program, reference ANR-23-IACL-0005.
\subsection*{Disclosure of AI use}
In line with the Leiden Declaration on Artificial Intelligence and Mathematics \cite{leiden2026declaration}, we disclose the use of AI tools (Claude, Gemini, ChatGPT) in the preparation of this paper. A full article was written by the authors first. The tools were used for verification and then for rewriting proofs. Among the key ideas used, only that of reverse Cauchy-Schwarz was obtained through AI tools.

    \bibliographystyle{alpha}
	\bibliography{bib_EOT}
    
	\newpage
	
	\appendix
	
	\titleformat{\section}[hang]
	{\normalfont\Large\bfseries}
	{Appendix~\thesection:}
	{0.5em}
	{}
	
	\section{Proof of \Cref{prop:tpsd_feat_map}}\label{proof:tpsd_feat_map}

	\ref{it_tpsd}$\Rightarrow$\ref{it_featMap}.
	We always use below the rule $\infty-\infty=+\infty$. Take $\calZ=\calX\times \calX$, and consider the function $\psi$ such that, for all $x,y\in \calX$, $\psi(x,(x,y))=c(x,x)/2$ and $\psi(x,(y,x))=c(x,y)- c(y,y)/2$. We set to $+\infty$ all the other values of $\psi(x,(u,v))$, for which $x\not \in \{u,v\}$, whence $\psi$ takes its values in $\R\cup\{+\infty\}$. There are two cases to consider: $c(x,x)$ and $c(x,y)$ for $x\neq y$.
	
	Take $x\neq y \in \calX$. Then, by definition of $\psi$, the only $z$ for which the values $\psi(x,z)$ and $\psi(y,z)$ can be finite are $z\in\{(x,y),(y,x)\}$. As $c$ is symmetric, we obtain that
	\begin{align*}
		\inf_{z\in \calZ} \psi(x,z) + \psi(y,z)
		&= \min \left(\psi(x,(x,y)) + \psi(y,(x,y)), \psi(x,(y,x)) + \psi(y,(y,x))\right)\\
		&= \min\left(\frac{c(x,x)}{2}+c(y,x)-\frac{c(x,x)}{2}, c(x,y)-\frac{c(y,y)}{2}+\frac{c(y,y)}{2}\right)=c(y,x), 
	\end{align*}
	the equality holding for $c(x,x)=+\infty$ or $c(y,y)=+\infty$, since we then have $c(y,x)=+\infty$ by the debiasability assumption.
	
	Concerning $c(x,x)$, we write similarly
	\begin{align*}
		\inf_{z\in \calZ} \psi(x,z) + \psi(x,z)
		&= 2\inf_{y\in \calX} \min\left(\psi(x,(x,y)), \psi(x,(y,x))\right)\\
		&= 2\min\left(\frac{c(x,x)}{2}, \inf_{y\in \calX}[c(x,y)-\frac{c(y,y)}{2}]\right)= c(x,x), 
	\end{align*}
	where we used that $c(x,x)+c(y,y) \leq 2 c(x,y)$.

	\ref{it_featMap}$\Rightarrow$\ref{it_tpsd}. The cost $c$ as defined by \eqref{eq:feature_map_carac} is symmetric and we assumed it takes its values in $\R\cup\{+\infty\}$. Then, for all $z\in \calZ$,
	\begin{equation*}
		2(\psi(x,z) + \psi(y,z))\ge 2 \inf_{z\in \calZ} \psi(x,z)+ 2\inf_{z\in \calZ}\psi(y,z)= c(x,x)+c(y,y),
	\end{equation*} 
	taking the infimum over $z\in \calZ$ allows us to conclude.

Assume now that $\calX$ is Hausdorff, that $c$ is real-valued on the diagonal, and equip
$\calZ=\calX\times\calX$ with the product topology, which is Hausdorff.

\medskip
\emph{Lower semicontinuity of $\psi(x,\cdot)$.} Fix $x\in\calX$ and $t\in\R$, and set
$h_x(u)\coloneqq c(x,u)-\tfrac12 c(u,u)$. Directly from the definition of $\psi$,
\begin{equation}\label{eq:sublevel_psi_z}
\{z\in\calZ:\ \psi(x,z)\le t\}
=\Big[\mathbf 1_{\{c(x,x)\le 2t\}}\cdot\big(\{x\}\times\calX\big)\Big]\ \cup\ \Big(\{h_x\le t\}\times\{x\}\Big),
\end{equation}
where the first set is to be read as empty when $c(x,x)>2t$. Since $\calX$ is Hausdorff, the singleton
$\{x\}$ is closed, so $\{x\}\times\calX$ and $\calX\times\{x\}$ are closed in $\calZ$. Moreover $h_x$
is l.s.c.\ as the sum of the l.s.c.\ function $c(x,\cdot)$ and of $-\tfrac12 c(\cdot,\cdot)$
restricted to the diagonal, which is l.s.c.\ precisely because we assumed that $u\mapsto c(u,u)$ is u.s.c.\ and
finite. Hence $\{h_x\le t\}$ is closed and the r.h.s. of \eqref{eq:sublevel_psi_z} is a
union of two closed sets, thus closed. Therefore $\psi(x,\cdot)$ is l.s.c.\ for every $x\in\calX$.

\medskip

\emph{Joint lower semicontinuity of $\psi$.} Assume in addition that $c$ is jointly l.s.c. Composing
with the continuous map $x\mapsto (x,x)$ shows that $x\mapsto c(x,x)$ is l.s.c., hence continuous
under our assumptions. Set
\[
E_1\coloneqq\{(x,u,v)\in\calX\times\calZ:\ x=u\},\qquad
E_2\coloneqq\{(x,u,v)\in\calX\times\calZ:\ x=v\},
\]
both closed, being preimages of the (closed, as $\calX$ is Hausdorff) diagonal of $\calX\times\calX$
under continuous projections. For every $t\in\R$,
\begin{equation}\label{eq:sublevel_psi_joint}
\{\psi\le t\}=\Big(E_1\cap\{(x,u,v):\ c(x,x)\le 2t\}\Big)\ \cup\
\Big(E_2\cap\{(x,u,v):\ c(x,u)-\tfrac12 c(u,u)\le t\}\Big).
\end{equation}
The map $(x,u,v)\mapsto c(x,x)$ is l.s.c., and so is $(x,u,v)\mapsto c(x,u)-\tfrac12 c(u,u)$, as the
sum of the jointly l.s.c.\ function $(x,u)\mapsto c(x,u)$ and of the l.s.c.\ function
$u\mapsto -\tfrac12 c(u,u)$. Both sets on the r.h.s. of \eqref{eq:sublevel_psi_joint} are
therefore closed, whence $\{\psi\le t\}$ is closed and $\psi$ is jointly l.s.c.

	\section{Useful lemmas}
	
	We gather here a few important results that are used throughout the article. We recall that we will always consider Polish spaces endowed with their Borel $\sigma$-algebra. The set $([0,1],\calL,\lambda)$ is the probability space $[0,1]$ endowed with the Lebesgue
$\sigma$-algebra $\calL$, i.e.\ the $\lambda$-completion of the Borel $\sigma$-algebra.
	
	\subsection{Probabilistic form of $\OT$}\label{sec:int-inf_prob-form}
	
	For the sake of completeness, we include the following lemma on the probabilistic form of optimal transport, which is well-known in the community, see e.g.\ \cite[p3]{Rachev1998} or \cite[p3]{Villani2003}, but not as a statement.    
	\begin{lemma}\label{lem:prob-form}
		Assuming $\bbX$ is a Polish space, we can equivalently write
		\begin{equation}\label{eq:equiv-ot-proba}
			\OT_{c}(\mu,\nu)=\inf_{\pi \in \Pi(\mu,\nu)} 
			\int_{\bbX\times\bbX} c(x,y) \de\pi(x,y)=\inf_{(T_1)_\# \lambda=\mu, (T_2)_\# \lambda=\nu}  \int_{[0,1]} c(T_1(w),T_2(w)) \de \lambda(w),
		\end{equation}
		where
        the infimum is taken over random variables $T_1,T_2$, i.e.\ $\calL$-measurable maps defined on $[0,1]$ and valued in $\bbX$.
	\end{lemma}
	\begin{proof}
It suffices to show that, for any $E$ Polish, any $\gamma\in\calP(E)$ is the law of a Borel map $T:[0,1]\to E$: applied to $E=\bbX\times\bbX$ and $\gamma=\pi\in\Pi(\mu,\nu)$, this gives $T=(T_1,T_2)$ with $(T_1,T_2)_\#\lambda=\pi$, so that the l.h.s.\ of \eqref{eq:equiv-ot-proba} is larger than the r.h.s., while conversely $(T_1,T_2)_\#\lambda\in\Pi(\mu,\nu)$ for any admissible $(T_1,T_2)$. By Kuratowski's theorem \cite[Theorem 15.6]{kechris1995classical}, there is a Borel isomorphism $\iota$ from $E$ onto a Borel set $B\subseteq\R$. Let $F$ be the distribution function of $\iota_\#\gamma$ and $G(w)\coloneqq\inf\{t\in\R:F(t)\ge w\}$, $w\in(0,1)$, its quantile; $G$ is nondecreasing, hence Borel, and $G_\#\lambda=\iota_\#\gamma$. Since $\iota_\#\gamma(B)=1$, the Borel set $N\coloneqq\{0,1\}\cup G^{-1}(\R\setminus B)$ is $\lambda$-negligible, and we can define $T\coloneqq\iota^{-1}\circ G$ on $[0,1]\setminus N$, whereas we can fix $T\equiv x_0$ on $N$ for an arbitrary point $x_0\in E$. Then, $T$ is Borel with $T_\#\lambda=\gamma$.
	\end{proof}
	
	\subsection{Commuting infimum and integral}
	
	To achieve our inf-representation in \Cref{fact:ot}, we used the following lemma to commute infimum and integral.

	\begin{lemma}\label{lem:comm_inf-int}
		Let \((\bbX,\mathcal A,\mu)\) be a complete probability space and let \(\calZ\) be a Polish space.
		Let $\tilde f:\bbX\times\calZ\to\R\cup\{+\infty\}$ be \(\mathcal A\otimes\mathcal B(\calZ)\)-measurable. Define
		\[
		f(x)\coloneqq \inf_{z\in\calZ}\tilde f(x,z)\in[-\infty,+\infty],\qquad x\in\bbX.
		\]
		Then \(f\) is $\calA$-measurable and
		\[
		\int_\bbX f(x)\,\de\mu(x)
		=
		\inf_{T}\int_\bbX \tilde f\bigl(x,T(x)\bigr)\,\de\mu(x),
		\]
		where the infimum is taken over all measurable maps $T:(\bbX,\calA)\to(\calZ,\calB(\calZ))$. Here the integral of an $\calA$-measurable $g:\bbX\to[-\infty,+\infty]$ is $\int g^+\de\mu-\int g^-\de\mu$, with the rule $\infty-\infty=+\infty$ of \Cref{def:debiasable}.
	\end{lemma}
    
	\begin{proof}
        To show that $f$ is $\calA$-measurable, we must show that $f^{-1}((-\infty, a)) \in \mathcal A$ for all $a\in\R$. Rewriting, we obtain
\begin{multline*}
    f^{-1}((-\infty,a))=\{x\in\bbX:f(x)<a\}=\{x\in\bbX:\inf_z \tilde f(x,z)<a\}=\{x\in\bbX:\exists z\in\calZ:\tilde f(x,z)<a\}\\
    =\operatorname{proj}_\bbX\{(x,z)\in\bbX\times\calZ:\tilde f(x,z)<a\}.
\end{multline*}
    The set $\{(x,z)\in\mathbb{X}\times\mathcal{Z}:\tilde{f}(x,z)<a\}$ belongs to $\calA\otimes\mathcal{B}(\mathcal{Z})$ by assumption. Because $\mathcal{Z}$ is a Polish space, the projection of this measurable set onto $\mathbb{X}$ is an $\calA$-Souslin (or analytic) set \cite[Theorem 6.9.12]{bogachev2007measure}. Moreover, since the probability space $(\mathbb{X},\calA,\mu)$ is complete, the $\sigma$-algebra $\calA$ is stable under the Souslin operation. This guarantees that the projection belongs to $\calA$. Consequently, $f^{-1}((-\infty,a))\in\calA$.

		With the above rule, the integral is monotone: if $g\le h$, either $\int h^+\de\mu=+\infty$, and then $\int h\,\de\mu=+\infty$, or $\int g^+\de\mu\le\int h^+\de\mu<+\infty$, and then the usual monotonicity applies. Since $f\le\tilde f(\cdot,T(\cdot))$ for every measurable $T$, this gives
		\[
			\inf_{T} \int_\bbX \tilde f(x,T(x))\, \de\mu(x) \ \ge\ \int_\bbX f(x)\,\de\mu(x).
		\]
		Let us show the converse. If $\int f^+\de\mu=+\infty$, then $\int f\,\de\mu=+\infty$ and there is nothing to prove. Otherwise, $f<+\infty$ $\mu$-a.e. For $\eta\in(0,1]$, set
		\[
		g_\eta\coloneqq\max(f,-1/\eta)+\eta,\qquad
		A_\eta(x) \coloneqq \{z \in \calZ : \tilde f(x,z) \le g_\eta(x)\}.
		\]
		Then $g_\eta$ is $\calA$-measurable, with $-1/\eta+\eta\le g_\eta\le f^++1$, so that $g_\eta\in L^1(\mu)$. Each $A_\eta(x)$ is nonempty: if $f(x)\in\R$, then $g_\eta(x)\ge f(x)+\eta>\inf_z\tilde f(x,z)$; if $f(x)=-\infty$, then $g_\eta(x)>-\infty=\inf_z\tilde f(x,z)$; and if $f(x)=+\infty$, then $A_\eta(x)=\calZ$. The graph $\{(x,z):\tilde f(x,z)\le g_\eta(x)\}$ belongs to $\calA\otimes\calB(\calZ)$, $\tilde f$ and $g_\eta$ being measurable. By \cite[Theorem 6.9.13]{bogachev2007measure}, there exists a measurable selector $T_\eta:\bbX\to\calZ$ with $T_\eta(x)\in A_\eta(x)$ for $\mu$-a.e.\ $x$, which we may modify on a $\mu$-null set, $\calA$ being complete. Hence $\tilde f(x,T_\eta(x))\le g_\eta(x)$ for $\mu$-a.e.\ $x$, and by monotonicity
		\[
		\inf_T\int_\bbX \tilde f(x,T(x))\,\de\mu(x)\ \le\ \int_\bbX \tilde f\big(x,T_\eta(x)\big)\,\de\mu(x)\ \le\ \int_\bbX g_\eta\,\de\mu .
		\]
		Finally, $g_\eta$ decreases pointwise to $f$ as $\eta\downarrow0$, with $g_\eta\le g_1\in L^1(\mu)$, and we conclude by dominated convergence.
	\end{proof}
    
	\subsection{$\KL$ decomposition}\label{proof:chain rule}
	Let us first recall a few elements on the disintegration of measures to set our notation.   Throughout this appendix the second argument of $\KL$ is allowed to be a $\sigma$-finite
nonnegative measure, the first one being finite; the integrability conditions under which the identities below hold are those of \Cref{lem:chain-KL-sigma}.
	\medskip
	
	Let $X,Y$ be Polish spaces and $\alpha\in\calP(X\times Y)$. We call disintegration of $\alpha$ with respect to its second marginal, denoted $\nu$, the measurable (with respect to the Borel $\sigma$-algebra of $Y$) family $(\alpha_y)_{y\in Y}\subset \calP(X)$ of probability measures such that, for any Borel function $f:X\times Y\to \R$,
	\[
	\int_{X\times Y} f(x,y)\de\alpha(x,y)=\int_Y \left(\int_X f(x,y) \de\alpha_y(x)\right) \de\nu(y),
	\]
	and we denote for concision (but with a slight abuse of notation) $\alpha=\alpha_y\otimes \nu$.
	This is more precisely the disintegration with respect to the projection to the second component of $X\times Y$, and more general disintegrations can be defined with respect to any Borel map, see for example \cite[Chapter 5]{ags08} for more details.
	One may analogously define the disintegration with respect to the first marginal. The family $(\alpha_y)_{y\in Y}$ always exists in our setting and it is furthermore $\nu$-a.e. unique. In the particular case of $\alpha=\mu\otimes\nu$, for $\mu\in\calP(X)$ and $\nu\in\calP(Y)$, we have $\alpha_y=\mu$ for any $y\in Y$ and $\alpha_x=\nu$ for any $x\in X$. Finally, note that for any $\alpha\in\Pi(\mu,\nu)$, $\alpha\ll \mu\otimes\nu$ if and only if $\alpha_y\ll \mu$ for $\nu$-a.e.\ $y\in Y$, if and only if $\alpha_x\ll\nu$ for $\mu$-a.e.\ $x\in X$.

	One important application of the disintegration of measures is the so-called gluing of transport plans. Namely, consider a third Polish space $Z$, a measure $\eta\in\calP(Z)$ and two couplings $\pi_1\in\Pi(\mu,\eta)$, $\pi_2\in\Pi(\nu,\eta)$. Disintegrating these in the form $\pi_1=(\pi_1)_z\otimes \eta$ and $\pi_2=(\pi_2)_z\otimes \eta$, we can construct a coupling $\gamma$ as $\gamma=((\pi_1)_z\otimes(\pi_2)_z) \otimes \eta$, a probability measure on $X\times Y\times Z$. One can check in particular that $\gamma\in\Pi(\mu,\nu,\eta)$, i.e.\ the marginals of $\gamma$ are $\mu,\nu$ and $\eta$ respectively, and, by definition, its disintegration with respect to $\eta$ is precisely $(\pi_1)_z\otimes(\pi_2)_z$, the product measure between the respective disintegrations of $\pi_1$ and $\pi_2$. 
	This construction serves specifically to produce probability measures on larger product spaces by gluing together measures defined on lower-dimensional spaces that share common marginals.
	Note in particular that $(\p_1,\p_3)_\#\gamma=\pi_1$ and $(\p_2,\p_3)_\#\gamma=\pi_2$. Such a measure $\gamma$ so constructed is not the only one with this property, as one may define analogously $\gamma_z \otimes \eta$ for any measurable family $(\gamma_z)_{z\in Z}$ with $\gamma_z\in\Pi\big((\pi_1)_z,(\pi_2)_z\big)$ for $\eta$-a.e.\ $z$. See \cite{ags08,santambrogio2015optimal} for more details on the gluing of transport plans (therein referred to as \textit{gluing lemma}) and its application in optimal transport.
	
	We recall now a few important facts about the Kullback-Leibler ($\KL$) divergence. Since in the article we use disintegration for more general nonnegative measures, we first state the chain rule on such measures.

    \begin{lemma}[Chain rule for nonnegative measures]\label{lem:chain-KL-sigma}

    	Let $X,Y$ be Polish spaces, let $\alpha$ be a finite and $\beta$ a $\sigma$-finite nonnegative measure on $X\times Y$, with second marginals denoted respectively by $\nu_1$ and $\nu_2$, $\nu_2$ being $\sigma$-finite as well. Then there exist families of probability measures $(\alpha_y)_{y\in Y},(\beta_y)_{y\in Y}\subset \calP(X)$ such that $\alpha=\alpha_y\otimes \nu_1$ and $\beta=\beta_y\otimes \nu_2$, and, provided that $\big(\log\frac{\de\nu_1}{\de\nu_2}\big)^-\in L^1(\nu_1)$ when $\nu_1\ll\nu_2$, which holds in particular when $\beta$ is finite,
	\begin{equation}\label{eq:chain-KL}
		\KL(\alpha|\beta)=\int_Y \KL(\alpha_y|\beta_y)\,\de\nu_1(y)+\KL(\nu_1|\nu_2),
	\end{equation}
	both sides being well defined in $(-\infty,+\infty]$.
\end{lemma}
\begin{proof}
    The existence of $(\alpha_y)_{y\in Y},(\beta_y)_{y\in Y}\in \calP(Y)$ follow from \cite[Theorem~14.D.10]{baccelli2020random}, since both $\alpha$ and $\beta$ are $\sigma$-finite with $\sigma$-finite second marginals.
	Both sides equal $+\infty$ unless $\nu_1\ll\nu_2$ and $\alpha_y\ll\beta_y$ for $\nu_1$-a.e.\ $y$, in
	which case $\alpha\ll\beta$ with
	$\frac{\de\alpha}{\de\beta}(x,y)=\frac{\de\alpha_y}{\de\beta_y}(x)\,\frac{\de\nu_1}{\de\nu_2}(y)$,
	$\beta$-a.e.. Taking the logarithm and integrating against $\alpha$, whose second marginal is $\nu_1$, gives
	\eqref{eq:chain-KL}.
\end{proof}
    An analogous formula holds in the case we disintegrate with respect to the first marginal instead of the second one. In particular, if $\alpha$ and $\beta$ share the same second marginal, i.e.\ $\nu_1=\nu_2$, then the second term on the r.h.s.\ of \eqref{eq:chain-KL} vanishes. Since the $\KL$ divergence is always nonnegative, a simple consequence of the chain rule is that for any $\mu_1,\mu_2\in\calP(X)$, $\eta_1,\eta_2\in\calP(Y)$ and any $\alpha\in \Pi(\mu_1,\eta_1)$, $\beta\in\Pi(\mu_2,\eta_2)$,
	\begin{equation}\label{eq:marginalization}
		\KL(\alpha|\beta)= \int_X \KL(\alpha_x|\beta_x) \de\mu_1(x)+\KL(\mu_1 |\mu_2)\ge	\KL(\mu_1 |\mu_2).
	\end{equation}
	Finally, for product measures $\mu_1\otimes\nu_1$ and $\mu_2\otimes\nu_2$,
	\begin{equation}\label{eq:KL-product}
		\begin{aligned}\KL(\mu_1\otimes\nu_1|\mu_2\otimes\nu_2)&=
			\int_{X} \log\left(\frac{\de\mu_1}{\de\mu_2}(x)\right)\de\mu_1(x) + \int_{Y} \log\left(\frac{\de\nu_1}{\de\nu_2}(y)\right)\de\nu_1(y)\\
			&=\KL(\mu_1|\mu_2)+\KL(\nu_1|\nu_2).
		\end{aligned}
	\end{equation}
	
	We conclude with a useful result on the decomposition of the Kullback-Leibler divergence for three-marginals probability measures. 
	\begin{lemma}\label{lem:chain_rule_KL}\label{lem:KL-decomp2}
		Let $X,Y,Z$ be Polish spaces. Take $\mu\in\calP(X), \nu\in \calP(Y)$ and $\eta\in\calP(Z)$. Then
		\begin{enumerate}[label=\roman*)]
			\item for any $\gamma\in\Pi(\mu,\nu,\eta)$,
			\begin{equation}\label{eq:KL_decomp-gen}
				\KL(\gamma|\mu\otimes\nu\otimes\eta)=\KL(\alpha|\mu\otimes\eta)+\KL(\beta|\nu\otimes\eta)+I(\gamma)
			\end{equation}
			where $\alpha=(\p_1,\p_3)_\#\gamma$, $\beta=(\p_2,\p_3)_\#\gamma$ and $I(\gamma)\coloneqq \int_Z \KL(\gamma_z|(\p_1)_\#\gamma_z\otimes (\p_2)_\#\gamma_z) \de\eta(z) \ge 0$, and in particular
			\begin{equation}\label{eq:KL_decomp-ineq}
				\KL(\gamma|\mu\otimes\nu\otimes\eta)\ge \KL(\alpha|\mu\otimes\eta)+\KL(\beta|\nu\otimes\eta);
			\end{equation}
			\item if $\gamma$ can be disintegrated in the form $\gamma=\left(\alpha_z\otimes\beta_z\right) \otimes \eta$, i.e.\ it coincides with the gluing of $\alpha$ and $\beta$,
			\begin{equation}\label{eq:KL_decomp-spec}
				\KL(\gamma|\mu\otimes\nu\otimes\eta)=\KL(\alpha|\mu\otimes\eta)+\KL(\beta|\nu\otimes\eta).
			\end{equation}
			\item more generally, for $\lambda\in\calP(Z)$, any $\pi_1\in \Pi(\mu,\eta)$ and $\pi_2\in\Pi(\nu,\eta)$, and $\gamma=\left((\pi_1)_z\otimes (\pi_2)_z\right) \otimes \eta$ the gluing of $\pi_1$ and $\pi_2$,
			\begin{equation}\label{eq:KL-decomp2-gen}
				\KL(\pi_1|\mu\otimes\eta)+\KL(\pi_2|\nu\otimes\eta)+\KL(\eta|\lambda)=\KL(\gamma|\mu\otimes\nu\otimes\lambda),
			\end{equation}
			both sides being $+\infty$ unless $\eta\ll\lambda$, while for a general $\gamma\in \Pi(\mu,\nu,\eta)$, with $\pi_1=(\p_1,\p_3)_\#\gamma$ and $\pi_2=(\p_2,\p_3)_\#\gamma$,
			\begin{equation}\label{eq:KL-decomp2-ineq}
				\KL(\pi_1|\mu\otimes\eta)+\KL(\pi_2|\nu\otimes\eta)+\KL(\eta|\lambda)\le\KL(\gamma|\mu\otimes\nu\otimes\lambda).
			\end{equation}
		\end{enumerate}
	\end{lemma}
        
    \Cref{lem:chain_rule_KL} holds verbatim with $\lambda\in\calM_+(Z)$
$\sigma$-finite in place of $\lambda\in\calP(Z)$: in both proofs $\lambda$ appears only as the base of
the disintegration of $\mu\otimes\nu\otimes\lambda$ along $Z$, whose conditionals are the probability
measure $\mu\otimes\nu$, so that \Cref{lem:chain-KL-sigma} applies; the only difference is that
$\KL(\eta|\lambda)$ may now take the value $-\infty$, in which case both sides of
\eqref{eq:KL-decomp2-gen} equal $-\infty$ and \eqref{eq:KL-decomp2-ineq} is read in $[-\infty,+\infty]$.
	\begin{proof}
		$i)$ Disintegrate $\gamma$ as $\gamma=\gamma_z\otimes \eta$.  If $\gamma_z$ is not absolutely continuous with respect to $\mu\otimes\nu$, then both the l.h.s.\ and the r.h.s.\ are infinite.
		Assume thus that $\gamma_z\ll (\mu\otimes\nu)$, for any $z\in Z$. Then, $(\p_1)_\#\gamma_z\otimes (\p_2)_\#\gamma_z\ll (\mu\otimes\nu)$ and $\gamma_z\ll (\p_1)_\#\gamma_z\otimes (\p_2)_\#\gamma_z$ and we can write, for any $z\in Z$,
		\[
		\begin{aligned}
			\KL(\gamma_z|\mu\otimes\nu)&=\int_{X\times Y} \log\left(\frac{\de\gamma_z}{\de(\mu\otimes\nu)}(x,y)\right) \de \gamma_z(x,y) \\
			&=\int_{X\times Y} \left[ \log\left(\frac{\de\gamma_z}{\de(\p_1)_\#\gamma_z\otimes (\p_2)_\#\gamma_z}(x,y)\right) + \log\left(\frac{\de(\p_1)_\#\gamma_z\otimes (\p_2)_\#\gamma_z}{\de(\mu\otimes\nu)}(x,y)\right) \right]\de \gamma_z(x,y).
		\end{aligned}	
		\]
		Since both $(\p_1)_\#\gamma_z\otimes (\p_2)_\#\gamma_z$ and $\mu\otimes\nu$ are products measures,
		\begin{equation*}
			\int_{X\times Y} \log\left(\frac{\de(\p_1)_\#\gamma_z\otimes (\p_2)_\#\gamma_z}{\de\mu\otimes\nu}(x,y)\right) \de \gamma_z(x,y)=\KL((\p_1)_\#\gamma_z|\mu)+\KL((\p_2)_\#\gamma_z|\nu),
		\end{equation*}
		and we obtain
		\[
		\KL(\gamma_z|\mu\otimes\nu)=\int_{X\times Y} \log\left(\frac{\de\gamma_z}{\de(\p_1)_\#\gamma_z\otimes (\p_2)_\#\gamma_z}(x,y)\right) \de \gamma_z(x,y)+\KL((\p_1)_\#\gamma_z|\mu)+\KL((\p_2)_\#\gamma_z|\nu).
		\]
		Now, integrating with respect to $\eta$ and using the chain-rule \eqref{eq:chain-KL},
		\[
		\begin{gathered}
			\int_Z \KL(\gamma_z|\mu\otimes\nu) \de\eta(z)=\KL(\gamma|\mu\otimes\nu\otimes\eta), \\
			\int_Z \KL((\p_1)_\#\gamma_z|\mu) \de\eta(z)=  \KL(\alpha|\mu\otimes\eta), \\
			\int_Z \KL((\p_2)_\#\gamma_z|\nu) \de\eta(z)=  \KL(\beta|\nu\otimes\eta),
		\end{gathered}
		\]
		and we obtain \eqref{eq:KL_decomp-gen}.
		
		$ii)$ If $\gamma$ can be disintegrated in the form $\gamma=\left(\alpha_z\otimes\beta_z\right)\otimes\eta$, then we have $\gamma_z=\alpha_z\otimes\beta_z$, $(\p_1)_\#\gamma_z=\alpha_z$ and $(\p_2)_\#\gamma_z=\beta_z$ and the term $I$ in \eqref{eq:KL_decomp-gen} vanishes.
        
		$iii)$ By the chain rule \eqref{eq:chain-KL} along the third marginal, $\KL(\gamma|\mu\otimes\nu\otimes\lambda)=\KL(\gamma|\mu\otimes\nu\otimes\eta)+\KL(\eta|\lambda)$ for any $\gamma\in\Pi(\mu,\nu,\eta)$, both sides being $+\infty$ if $\eta\not\ll\lambda$; \eqref{eq:KL-decomp2-gen} and \eqref{eq:KL-decomp2-ineq} then follow from \eqref{eq:KL_decomp-spec} and \eqref{eq:KL_decomp-ineq} respectively, the two terms $\KL(\pi_i|\cdot)$ being nonnegative so that no indeterminate form occurs.
	\end{proof}
	
\end{document}